\documentclass[preprint,12pt]{elsarticle}

\usepackage{amsfonts}
\usepackage{amsmath,amssymb,amsthm,amsthm}
\usepackage{mathtools}
\usepackage{dsfont}
\usepackage{etoolbox}
\usepackage{color}
\usepackage{url}
\newcommand{\comment}[1]{}
\usepackage{algorithm}
\usepackage{algpseudocode}
\algdef{SE}[SUBALG]{Indent}{EndIndent}{}{\algorithmicend\ }%
\algtext*{Indent}
\algtext*{EndIndent}

\newtheorem{theorem}{Theorem}[section]

\newtheorem{definition}[theorem]{Definition}

\newtheorem{lemma}[theorem]{Lemma}

\newtheorem{remark}[theorem]{Remark}
\numberwithin{equation}{section}

\usepackage[normalem]{ulem}

\newcommand{\BEA}{\begin{eqnarray}}
\newcommand{\EEA}{\end{eqnarray}}

\usepackage[textsize=small]{todonotes}

\newcommand{\jh}[1]{\textcolor{red}{#1}}

\journal{Appl. Comput. Harmon. Anal.}
\begin{document}

\begin{frontmatter}

\title{Fractional Diffusion Maps
}

\author[rvt]{Harbir Antil\corref{cor}}
\ead{hantil@gmu.edu}
\author[rvt]{Tyrus Berry}
\ead{tberry@gmu.edu}
\author[rvt2]{John Harlim}
\ead{jharlim@psu.edu}
\address[rvt]{Department of Mathematical Sciences and the Center for Mathematics and Artificial Intelligence (CMAI), George Mason University, Fairfax, VA 22030, USA.}
\address[rvt2]{Department of Mathematics, Department of Meteorology and Atmospheric Science, \& Institute for Computational and Data Sciences, the Pennsylvania State University, University Park, PA 16802, USA.}

\begin{abstract}
{In this paper, we extend the diffusion maps algorithm on a family of heat kernels that are either local (having exponential decay) or nonlocal (having polynomial decay), arising in various applications. For example, these kernels have been used as a regularizer in various supervised learning tasks for denoising images. Importantly, these heat kernels give rise to operators that include (but are not restricted to) the generators of the classical Laplacian associated to Brownian processes as well as the fractional Laplacian associated with $\beta$-stable L\'evy processes. For local kernels, while the method is a version of the diffusion maps algorithm, we show that the applications with non-Gaussian local heat kernels approximate temporally rescaled Laplace-Beltrami operators. For the non-local heat kernels, we modify the diffusion maps algorithm to estimate fractional Laplacian operators. Here, the graph distance is used to approximate the geodesic distance with appropriate error bounds. While this approximation becomes numerically expensive as the number of data points increases, it produces an accurate operator estimation that is robust to the choice of the kernel bandwidth parameter value. In contrast, the local kernels are numerically more efficient but more sensitive to the choice of kernel bandwidth parameter value. In an application to estimate non-smooth regression functions, we find that using the nonlocal kernel as a regularizer produces a more robust and accurate estimate than using local kernels. For manifolds with boundary, we find that the proposed fractional diffusion maps framework implemented with non-local kernels approximates the regional fractional Laplacian.}
\end{abstract}

%

\end{frontmatter}

\section{Introduction}

{ An important idea in machine learning and nonparametric regression is the so-called ``kernel trick" in which a linear algorithm is extended to capture nonlinear features by replacing occurrences of the Euclidean dot product with the evaluation of a kernel function on all pairs of data points.  Classical statistical learning theory analysis would suggest that these kernel methods suffer from the curse-of-dimensionality in terms of the ambient dimension of the data, however, when the data can be assumed to lie near a submanifold of the data space, the curse-of-dimensionality can be reduced to the intrinsic dimension of the manifold \cite{diffusion,singer}.  Rigorous interpretation of these methods and their assumptions (especially in the limit of large data) requires formulating the kernel as a regularization term that should be understood in terms of intrinsic norms and operators on the underlying manifold.  This regularization is determined by the choice of the kernel function, for example, the Gaussian function is connected to the heat kernel of the Laplace-Beltrami operator on a compact Riemannian manifold \cite{diffusion} and naturally introduces an $H^1$ regularization \cite{regularization}.  Thus, a fundamental question in this branch of machine learning is: What regularizations can we achieve with kernel methods and which kernels are associated with which regularizations?

\comment{
Secondarily, there has been significant recent interest in Monte-Carlo based solvers for PDEs \cite{MCpde}, and kernel-based approximation of differential operators is the key to Monte-Carlo based solvers for PDEs on submanifolds \cite{JohnPDE1,JohnPDE2}.  The results in this manuscript are applicable to these PDE solvers and extend their applicability to fractional Laplace operators on submanifolds.  Finally, we provide a kernel-based perspective which is related to the popular Isomap and Parallel Transport Unfolding (PTU) algorithms.  Isomap \cite{tenenbaum2000global,MBernstein_VDSilva_JCLangford_JBTenenbaum_2000a} and PTU \cite{PTU} can be easily interpreted under the restrictive assumptions of manifolds that are both topologically trivial (contractible) and geometrically trivial (developable, which implies flat in the Riemannian sense).  However, Isomap and PTU are still useful for many examples \jh{that} do not satisfy these assumptions, and in these contexts, there is a limited theoretical interpretation of their behavior.  While our analysis cannot address these algorithms directly, our kernel-based algorithms are closely related, and we rely on their methods for the key step of approximating geodesic distances.  Thus, our results give an alternative final step to Isomap and PTU, by replacing Multi-Dimensional Scaling (MDS) with a \jh{rigorously interpretable} kernel method \sout{that is rigorously interpretable.}  Fractional diffusion maps are the analog of the diffusion maps for a fractional Laplacian, and we show that in the limit of large data it differs from the standard diffusion maps only by a scaling of the coordinates (for manifolds without boundary).  
}


The primary goal of this manuscript is to expand the range of regularizations accessible through kernel methods to include the very powerful fractional Laplacian regularizers \cite{HAntil_SBartels_2017a,HAntil_CNRautenberg_2019a,HAntil_ZDi_RKhatri_2020a} for supervised learning of non-smooth functions. In the process, we are able to exhibit a fundamental dichotomy in the types of kernel functions based on their rate of decay, using recent advances in the understanding of heat kernels on manifolds \cite{AGrigoryan_TKumagai_2008a}.  This dichotomy helps delineates the range of regularizations accessible to kernel methods.} In particular, it was shown that in locally compact separable spaces, the so-called stochastically complete, $\beta$-scale invariant heat kernels are either \emph{local} (having exponential decay) or \emph{non-local} (having polynomial decay).  The class of local kernels includes the Gaussian kernel which yields an estimate of the heat kernel on the manifold with error bounds that depend on the curvature and its derivatives \cite{rosenberg}.  The Gaussian kernel was also used in the diffusion maps algorithm \cite{diffusion} to estimate the Laplace-Beltrami {operator, $\Delta$,} which is the generator of Brownian motion on a manifold. On the other hand, the class of nonlocal kernels includes the Poisson kernel which gives rise to the generator $(-\Delta)^{1/2}$ in $\mathbb{R}^n$. In fact, the class of heat kernels with polynomial decay in \cite{AGrigoryan_TKumagai_2008a} is associated with the fractional Laplacian $(-\Delta)^{\beta/2}$ for $\beta\in(0,2)$.  In other words, these nonlocal kernels generate the $\beta$-stable L\'evy processes in $\mathbb{R}^n$ \cite{grigor2003heat}.  {Every heat kernel on a manifold must be of either local or non-local type, so for a kernel-based manifold learning approach to have a well-defined limit the kernel used must converge to either a local or non-local heat kernel.}    

In this paper we will develop consistent estimators for the two classes of semigroup operators arising in the dichotomy of \cite{AGrigoryan_TKumagai_2008a} using only samples of data that lie on or near an embedded manifold where neither the manifold nor the embedding function is explicitly known.  {The key difficulty is that since we do not have explicit knowledge of} the embedding (or, equivalently, the Riemannian metric) we cannot {directly} evaluate the geodesic distance on the {manifold.  Consequently, we cannot} directly evaluate either of the two classes of heat kernels in \cite{AGrigoryan_TKumagai_2008a}, which are both defined as functions of the geodesic distance. {We will show that in order to have a consistent estimator of a heat kernel, the method used for approximating the geodesic distances will differ based on whether the desired heat kernel is local or non-local.}

{For local kernels, we will show that geodesic distances do not need to be explicitly estimated.  We show this} by generalizing the theory of \cite{diffusion}, which showed that kernels with exponential decay localize the interactions between points so that the ambient Euclidean distance is sufficiently close to the geodesic {distance (up to an error proportional to the geodesic distance squared).}  This result was used in the diffusion maps algorithm to show that the Laplace-Beltrami operator on a manifold can be estimated by a weighted graph Laplacians (where the graph is constructed by connecting data points sampled on an embedding of the manifold in Euclidean space). If a data set is assumed to lie on or near an embedded manifold, the diffusion maps result provides a rigorous foundation for so-called `kernel methods' used for unsupervised learning algorithms and dimensional reduction. In particular, it approximates the semigroup associated with the Laplace-Beltrami operator with a discretization of an integral operator, defined with any kernel functions that decay to zero exponentially as the distance between data points increases. It was shown in \cite{localkernels} that this class of so-called ``local kernels" allows one to estimate non-symmetric Kolmogorov operators defined with respect to the Riemannian metric inherited by the manifold from the ambient space. 

For non-local heat kernels, we cannot use the Euclidean distance of the embedded data to approximate the geodesic distance since this approximation is only valid locally. In this case, we consider the 
graph distance as an estimator for the geodesic distance, which is the idea behind the Isomap algorithm \cite{tenenbaum2000global,MBernstein_VDSilva_JCLangford_JBTenenbaum_2000a}. Numerically, the graph distance will be computed using Dijkstra's algorithm which finds the shortest path between the training data points. Using the error estimate for the geodesic distance approximation that was formulated in \cite{MBernstein_VDSilva_JCLangford_JBTenenbaum_2000a}, we derive error bounds for nonlocal kernels and their associated semigroups and generators.  { \comment{We note that improved geodesic distance estimators \cite{PTU} can also be used in practice. , but for the error bounds in \cite{MBernstein_VDSilva_JCLangford_JBTenenbaum_2000a} are required for our analysis.} We should note that since we require graph distances between all pairs of points, Dijkstra's algorithm may be improved by Johnson's algorithm which requires $\mathcal{O}(N^2 \log N)$ computations, where $N$ is the number of data points.  Since local kernels do not require this step, this shows that local kernels have an advantage in computational complexity.  On the other hand, while polynomial kernels are sometimes used in practice without estimating geodesic distances, our results show that these will not have well-defined heat kernel limits without this additional step.
}

The remainder of this paper will be organized as follows: In Section~\ref{section2}, we will briefly review the relevant results from \cite{AGrigoryan_TKumagai_2008a} which serves as the foundation for this work. In Section~\ref{section3}, we derive the theory for estimating the semigroup operators in \cite{AGrigoryan_TKumagai_2008a} using data sampled from an embedded manifold. In Section~\ref{application}, we provide a detailed numerical algorithm including the necessary modifications for application with non-local kernels. In Section~\ref{numerics}, we provide numerical examples to support the theoretical results deduced in Section~\ref{section3}. {In addition, we also compare the local and non-local heat kernels as regularizers in a kernel ridge regression application.} In Section~\ref{conclusion}, we close the paper with a summary and outlook of open problems. { \comment{We should note that throughout this paper we assume that the manifold is compact. In addition, we will specify whenever we need the manifold to be closed. Generalizing beyond these assumptions is an important direction for future research. In particular, the heat kernel dichotomy of \cite{AGrigoryan_TKumagai_2008a} holds on a much larger class of structures known as metric measure spaces.}  In order to help guide future research in this direction we will occasionally comment on considerations that may become relevant to such generalizations.}

\section{Notation and preliminaries}\label{section2}

\subsection{Notation}\label{notation}

We will denote the Euclidean norm by $|\cdot|$ and reserve $\|\cdot\|$ for function space norms. Throughout this paper we will consider a compact $C^3$ Riemannian manifold $\mathcal{M}$, isometrically embedded into a Euclidean space by a $C^3$ function $\iota:\mathcal{M}\to\mathbb{R}^n$.  Equivalently, given an arbitrary embedding, we may say that we are interested in the Riemannian geometry, $g$, inherited from the embedding.  

We will denote by $d_g(x,y)$ the geodesic distance which is identical to the intrinsic distance on a compact manifold, $\mathcal{M}$, given by the infimum over piecewise differentiable paths $\gamma:[0,1]\to \mathcal{M}$ between $\gamma(0)=x$ and $\gamma(1) =y$
\[ d_g(x,y) = \inf_{\gamma} \int_0^1 \sqrt{g_{\gamma(t)}(\nabla\gamma(t),\nabla\gamma(t))} \, dt. \]
Since the embedding $\iota$ is continuous and $\mathcal{M}$ is compact, the ratio between the Euclidean distance in the embedding space and the geodesic distance, 
\[ R(x,y) = \frac{|\iota(x)-\iota(y)|}{d_g(x,y)}, \]
is bounded away from zero.  Moreover, when $y$ is sufficiently close to $x$ we have the following relationship between the geodesic distance and the Euclidean distance in the embedding space. 
\begin{lemma}[Distance comparison]\label{distcomp} Let $x,y\in \mathcal{M}$ with $d_g(x,y)$ less than the injectivity radius at $x$ and let $\iota:\mathcal{M}\to\mathbb{R}^n$ be an isometric embedding, then
\begin{equation}\label{distances} |\iota(y)-\iota(x)|^{\alpha} = d_g(x,y)^{\alpha} + \mathcal{O}(d_g(x,y)^{\alpha+2}), \end{equation}
for any $\alpha > 0$.
\end{lemma}
The proof of Lemma \ref{distcomp} is in \ref{proofs}.  Equation \eqref{distances} will be the key to connecting certain integral operators with kernels defined in the embedding space to intrinsic heat kernels.  Finally, since the injectivity radius on $\mathcal{M}$ is bounded away from zero, on all sufficiently small balls, we have $d_g(x,y)  < c|\iota(x)-\iota(y)|$ since $R(x,y)$ is bounded away from zero, where $c > 0$ is a positive constant.  Thus from \eqref{distances}, we have 
\begin{equation}\label{distance2}  |\iota(y)-\iota(x)|^{\alpha} = d_g(x,y)^{\alpha} + \mathcal{O}(|\iota(y)-\iota(x)|^{\alpha+2}). \end{equation}
for $d_g(x,y)$ sufficiently small.

\subsection{Dichotomy in heat kernel}\label{heatkerneldichotomy}

In this section we briefly summarize a powerful result of \cite{AGrigoryan_TKumagai_2008a} which will form the cornerstone of this paper.  We should note that these results hold for locally compact separable metric spaces \cite{AGrigoryan_TKumagai_2008a}, however we will restrict our attention to Riemannian manifolds $(\mathcal{M},g)$ here.  We first state the definition of a heat kernel following \cite{AGrigoryan_TKumagai_2008a}.

\begin{definition}[Heat kernel] A function $k : [0,\infty) \times \mathcal{M} \times \mathcal{M} \to [0,\infty)$ is called a heat kernel if for almost every $x,y \in \mathcal{M}$ and all $s,t \geq 0$ we have
\begin{enumerate}
\item Positivity: $k(t,x,y) \geq 0$.
\item Total mass inequality: $\int_{y\in\mathcal{M}}k(t,x,y) \, d\textup{vol} \leq 1$.
\item Symmetry: $k(t,x,y)=k(t,y,x)$.
\item Semi-group: $k(s+t,x,y) = \int_{z\in\mathcal{M}}k(s,x,z)k(t,z,y) \, d\textup{vol}$.
\item Approximation of identity:\[\lim_{t\to 0^+} \left|\left| \int_{y\in\mathcal{M}} k(t,x,y)f(y)\, d\textup{vol} - f(x) \right|\right|_{L^2(\mathcal{M},g)} = 0, \mbox{ for } f\in L^2(\mathcal{M},g).\] 
\end{enumerate}
\end{definition}

A heat kernel gives rise to an associated semigroup
\BEA \mathcal{K}_t f(x) = \int_{y\in\mathcal{M}}k(t,x,y)f(y)\, d\textup{vol} \label{semigroup}\EEA
and we say that $k$ is \emph{stochastically complete} if $\mathcal{K}_t 1 = 1$ for all $t>0$.  The semigroup gives rise to a quadratic form 
\[ \xi(f) = \lim_{t\to 0^+} \left<\frac{f - \mathcal{K}_t f}{t},f \right>_{L^2(\mathcal{M},g)} \]
and a generator $\mathcal{L}f = \lim_{t\to 0^+} \frac{f - \mathcal{K}_t f}{t}$, provided that the limit holds in $L^2(\mathcal{M},g)$.  We say that $\xi$ is \emph{regular} if there exists a set $\mathcal{C}$ of continuous functions with compact support ($C_c(\mathcal{M}))$ that are also in the domain, $\mathcal{D}(\xi)$, of $\xi$ (in other words $\mathcal{C} \subset C_c(\mathcal{M})\cap \mathcal{D}(\xi)$) such that $\mathcal{C}$ is dense both in $C_c(\mathcal{M})$ and $ \mathcal{D}(\xi)$ under appropriate norms (see \cite{AGrigoryan_TKumagai_2008a}).  The results developed in \cite{AGrigoryan_TKumagai_2008a} connect a large class of heat kernels to generators of important Markov processes.  Namely, we only need to assume that our heat kernel has the following type of scale invariance.

\begin{definition}[$\beta$-scale invariant] 
\label{def:scaleinvariance}
We say that a heat kernel is $\beta$-scale invariant if for almost every $x,y\in\mathcal{M}$ and every $t> 0$ it satisfies
\begin{equation}\label{scaleinvariance}
 \frac{c_1}{t^{d/\beta}} \Phi\left(C_1 \frac{d_g(x,y)}{t^{1/\beta}} \right) 
  \le k(t,x,y) 
  \le \frac{c_2}{t^{d/\beta}} \Phi\left(C_2 \frac{d_g(x,y)}{t^{1/\beta}} \right)  
\end{equation}
where $\beta, c_1, c_2,C_1,C_2$ are positive constants and $\Phi : [0,\infty) \rightarrow [0,\infty)$ is monotone decreasing.
\end{definition}

Here, $\mathcal{M}$ is a $d$-dimensional manifold (the result in \cite{AGrigoryan_TKumagai_2008a} also applies to fractional dimensional sets).  Then we have the following result [Theorem 4.1 of \cite{AGrigoryan_TKumagai_2008a}].

\begin{theorem}[Heat kernel dichotomy \cite{AGrigoryan_TKumagai_2008a}]\label{heatdichotomy} Let all { balls, defined with metric $d_g$ in $\mathcal{M}$,} to be relatively compact and $k(t,x,y)$ be a stochastically complete heat kernel with $\beta$-scale invariant such that the associated quadratic form is regular.  Then $\beta \leq d+1$ and either
\begin{enumerate}
\item $k$ is local, meaning $\beta \geq 2$ and \eqref{scaleinvariance} holds with $\Phi(a) = \exp\left({-a^{\frac{\beta}{\beta-1}}}\right)$ or
\item $k$ is non-local, meaning $0<\beta < 2$ and \eqref{scaleinvariance} holds with $\Phi(a) = (1+a)^{-(d+\beta)}$.
\end{enumerate}
\end{theorem}

We can now connect the $\beta$-scale invariant heat kernels to the intrinsic fractional Laplacian operators on a compact manifold. 
Let $\Delta$ denote the intrinsic (negative semi-definite) Laplacian-Beltrami operator. 
Since $\mathcal{M}$ is compact \cite[Theorem~3.2.1]{JJost_2017a}, the eigenvalue problem $-\Delta \phi_i = \lambda_i \phi_i$ has countably many eigenvalues with orthonormal eigenfunctions in $L^2(\mathcal{M})$. The eigenvalues fulfill 
$0 = \lambda_1 < \lambda_2 \le \lambda_3 \le \dots$ with $\lim_{i\rightarrow \infty} \lambda_i = \infty$. Moreover, the eigenfunction $\phi_1$ corresponding to $\lambda_1$ is a constant. Finally for any $f \in L^2(\mathcal{M})$ we have 
$f(x) = \sum_{i=1}^\infty \left<f,\phi_i\right>_{L^2(\mathcal{M})} \phi_i(x)$. For a closed manifold $\mathcal{M}$, one can define the spectral fractional Laplacian as:

\begin{definition}[fractional Laplacian on closed manifolds]
\label{specdef}  
Let $\mathcal{M}$ be a compact manifold without boundary. The intrinsic fractional Laplacian operator $(-\Delta)^{s}$ is the generator of the semigroup $e^{-t(-\Delta)^s}$ with heat kernel
\[ G_{s}(t,x,y) = \sum_{i=1}^{\infty} e^{-t \lambda_i^{s}}\phi_i(x)\phi_i(y)  \hspace{40pt} e^{-t(-\Delta)^s}f(x) = \int_{\mathcal{M}} G_s(t,x,y)f(y) \, d\textup{vol} \]
and the fractional Laplacian can be written as 
\[ (-\Delta)^{s}f(x) = \sum_{i=1}^{\infty} \lambda_i^{s}\left<f,\phi_i\right>_{L^2(\mathcal{M})}\phi_i(x) \] 
for $f$ in the domain of $(-\Delta)^s$.  
\end{definition}

In $\mathbb{R}^n$ and on a flat torus, the spectral fractional Laplacian coincides with the integral definition (see e.g.~\cite{kwasnicki2017ten} or \cite[Pg.~15]{NAbatangelo_EValdinoci_2017a}).
For general compact manifolds without boundary, the fractional Laplacian can be approximately represented in a Cauchy Principal Value integral form (see \cite[Theorem~1.2 and Theorem~1.4]{alonso2018integral} for the exact error term of this representation). For equivalent definitions of the fractional Laplacian, some of which generalize to non-compact manifolds see \cite{VBanica_MadMarGonzalez_MSaez_2015a,SMMolcanov_EOstrovski_1969a,LCaffarelli_LSilvestre_2007a,PRStinga_JLTorrea_2010a,antil2018fractional,antilshort,YSire_EValdinoci_2010a,SYAChang_MadMarGonzalez_2011a,MadMarGonzalez_MSaez_YSire_2014a}.

For compact manifolds, the standard ($s=1$) heat kernel has the expansion  (see for example \cite{rosenberg})
\[ G_1(t,x,y) = (4\pi t)^{-d/2} \exp\left( -\frac{d_g(x,y)^2}{4t} \right)(1 + \mathcal{O}(t)) \] 
 which agrees with the $\beta$-scale invariant  kernel with $\beta=2$, $c_1=c_2$, and $C_1=C_2$ (again up to a constant term and rescaling time) in the limit as $t\to 0$.  Furthermore, \cite[Theorem 4.2]{gimperlein2014heat} shows that for any compact manifold and any $0<s<1$, we have
\BEA\label{gimperleinBounds} \frac{c_1}{t^{d/(2s)}}\left(\frac{d_g(x,y)}{t^{1/(2s)}} + 1\right)^{-(d+2s)} + \mathcal{O}(t) &\leq& G_s(t,x,y) \nonumber\\ &\leq& \frac{c_2}{t^{d/(2s)}}\left(\frac{d_g(x,y)}{t^{1/(2s)}} + 1\right)^{-(d+2s)} + \mathcal{O}(t). \nonumber  \EEA
Setting $\beta = 2s$ this result shows that for all $0 <\beta < 2$ as $t\to 0$ the $\beta$-scale invariant kernel in Definition \ref{def:scaleinvariance} with $c_1=c_2$ and $C_1=C_2$ recovers the heat kernel $G_{\beta/2}$ associated to the fractional Laplacian $(-\Delta)^{\beta/2}$.  

Finally, we should note that if $C_1 \neq C_2$ then the operator obtained can be very different.  For example, in \cite{localkernels} they consider kernels of the form $e^{-(x-y)^\top A(x)(x-y)}$ where the symmetric matrix $A(x)$ varies as a function of $x$.  Notice that as long as the eigenvalues of $A(x)$ are bounded away from zero and infinity for all $x$, inequalities \eqref{scaleinvariance} (with $\beta=2$) hold with $C_1$ and $C_2$ given by infimum and supremum of the eigenvalues respectively.  However, in \cite{localkernels} these kernels are shown to change the differential operator that is estimated, or (equivalently) they can be viewed as changing the Riemannian metric on the manifold $\mathcal{M}$.  The theory of \cite{localkernels} is developed for manifolds without boundary, but the estimates can be extended to manifolds with boundary by restricting to Neumann functions as in \cite{diffusion}.
 
We now turn to the theoretical and practical implications of Theorem \ref{heatdichotomy} for manifold learning applications.

\section{Dichotomy in diffusion maps}\label{section3}

The dichotomy in the heat kernel is also reflected in the method that can be used in order to obtain an estimate of the heat kernel from embedded data. The key point is that given data sampled from an embedding $\iota(\mathcal{M})\subset\mathbb{R}^n$ we will only have direct access to the ambient Euclidean metric, rather than the intrinsic metric required for the construction of the heat kernel in the previous section.  In order to obtain an intrinsic operator on the data, we will need to follow very different strategies based on which type of operator we are approximating. 

For local diffusions, we will be able to simply use a kernel-based on the Euclidean distance in the embedding space, and asymptotically (as $t\to 0$) we will recover the intrinsic heat kernel.  On the other hand, for nonlocal processes we will have to explicitly estimate the intrinsic distance using Dijkstra's algorithm in order to obtain a consistent estimator of the intrinsic heat kernel.  These different approaches will have significant consequences for both computational efficiency and quality of results as we will show in Section \ref{numerics}.  We first consider local processes in section \ref{local} and then turn to nonlocal processes in section \ref{nonlocal}.

\subsection{$\alpha$-local kernels}\label{local}

In this section we consider the local processes and show that they are generated by a class of kernel functions which are defined on the ambient space $\mathbb{R}^n$.  The class of kernels will be called $\alpha$-local. 

\begin{definition}[$\alpha$-local kernels] Let $h : [0,\infty) \times \mathcal{M} \times \mathcal{M} \to [0,\infty)$ then we say that $h$ is an $\alpha$-local kernel if $\alpha > 1$ and if for some isometric embedding $\iota:\mathcal{M}\to\mathbb{R}^n$ we have
\[ c_1 e^{-\frac{|\iota(x)-\iota(y)|^{\alpha}}{t^{\alpha-1}}} \leq h(t,x,y) \leq c_2 e^{- \frac{|\iota(x)-\iota(y)|^{\alpha}}{t^{\alpha-1}}} \]
for some $c_1,c_2 > 0$, and for all $x,y \in \mathcal{M}$.\label{alphalocalkernel}
\end{definition}  
Notice that the $\alpha$-local kernels are defined on the intrinsic manifold $\mathcal{M}$ but instead of being intrinsically bounded above and below, they are bounded in an embedding space $\mathbb{R}^n$ in terms of the ambient metric $|\cdot|$.  The reason for this is that although we are interested in an intrinsic kernel defined on the manifold, we assume that we only have access to data points which are in the Euclidean space $\mathbb{R}^n$ which is the image of $\iota$.  For example a typical kernel would be of the form,
\BEA h(t,x,y) = e^{-|\iota(x)-\iota(y)|^2/t}. \label{kernelh}\EEA
The $\alpha$-local kernels (denoted by $h$) should be clearly differentiated from the intrinsic heat kernels (denoted by $k$) which have the form,
\BEA
k(t,x,y) := ct^{-d/\beta} e^{-\left(\frac{d_g(x,y)}{t^{1-1/\alpha}}\right)^{\alpha}} = ct^{-d/\beta} e^{-\left(\frac{d_g(x,y)}{t^{1/\beta}}\right)^{ \frac{\beta}{\beta-1}}} = ct^{-d/\beta} \Phi\left( \frac{d_g(x,y)}{t^{1/\beta}} \right) \nonumber
\EEA 
for $1+\frac{1}{d}<\alpha\leq 2$ where $\beta = \frac{\alpha}{\alpha-1}$ and $\Phi(a) = \exp\left(-a^{\frac{\beta}{\beta-1}}\right)$ as in Theorem.~\ref{heatdichotomy}(a).

We call these kernels $\alpha$-local because the exponential decay implies fast decay of the tails.  Crucially, this decay is so fast that if we integrate outside of ball of radius $\frac{1}{t^\gamma}$ for \emph{any} $\gamma>0$, then as $t\to 0^+$ the integral outside this ball will decay to zero faster than any power of $t$.
\begin{lemma}[Fast decay of exponential tails] \label{fastdecay} Let $c,\gamma>0$ and $\alpha \geq 1$ then 
\[  \lim_{t\to 0^+} t^\ell \int_{|z|>t^{-\gamma}} e^{-c |z|^{\alpha}} \, dz = 0   \]
for every $\ell\in \mathbb{R}$.
\end{lemma}
The proof of Lemma \ref{fastdecay} is straightforward and is included in \ref{proofs} for completeness. 
Next, we show that for $\alpha$-local kernels we can localize the integral operator to a neighborhood of $x$ where $|y-x|<t^{1-1/\alpha}$.  We do this by showing that the integral outside that region decays faster than any polynomial in $t$.

\begin{lemma}[Localization of $\alpha$-local kernels]\label{localization}  Let $h(t,x,y)$ be an $\alpha$-local kernel for some $\alpha>1$.  For any $x,y \in \mathcal{M}$ and any $f \in L^p(\mathcal{M})$ where $p \in (1,\infty]$ we have
\[  \lim_{t\to 0^+} t^\ell \int_{y \in \mathcal{M}, |\iota(y)-\iota(x)|>t^{1-1/\alpha-\gamma}} h(t,x,y)f(y)\, d\textup{vol} = 0 \]
for each $\gamma > 0$ and for any $\ell \in \mathbb{R}$.
\end{lemma}

Lemma \ref{localization} follows directly from application of H\"{o}lder's inequality, the proof is included in \ref{proofs}.  Now that we have localized the kernel, we can connect the Euclidean distance in the embedding space to the intrinsic/geodesic distance using the exponential coordinates $y=\exp_x(s)$.

\begin{theorem}[$\alpha$-local heat kernels]\label{localthm} Let $h$ be an $\alpha$-local kernel of Definition~\ref{alphalocalkernel} for some $1+1/d < \alpha \leq 2$, then
\BEA ct^{-d/\beta}  \int_{y \in \mathcal{M}} h(t,x,y)f(y)\, d\textup{vol} = {\mathcal K}_t f(x)\Big(1 + \mathcal{O}(t^{2-2/\alpha})\Big), \label{equality1} \EEA
for some $c>0$ and all $x\in\mathcal{M}$. Here, $f\in L^2(\mathcal{M})$ where $\beta = \frac{\alpha}{\alpha-1}$ so that $2\leq \beta < d+1$. The operator $\mathcal{K}_t$ denotes the semigroup of a stochastically complete, $\beta-$scale invariant local kernel of part (a) in Theorem~\ref{heatdichotomy} and $\mathcal{M}$ is a $d$-dimensional compact manifold. 

\end{theorem}
\begin{proof}
We first consider the kernel function $e^{-\left|\frac{\iota(y)-\iota(x)}{t^{1-1/\alpha}}\right|^{\alpha}}$ which appears in the upper and lower bounds of an $\alpha$-local kernel.  Applying Lemma \ref{localization} we localize the integral and then apply \eqref{distance2} to rewrite the Euclidean distance in terms of the geodesic distance so that for any $\gamma>0$ we have, 
\[\int_{y\in \mathcal{M}} e^{-\left|\frac{\iota(y)-\iota(x)}{t^{1-1/\alpha}}\right|^{\alpha}} f(y)\, d\textup{vol}  = \int_{\stackrel{y\in\mathcal{M},}{|\iota(x)-\iota(y)|<t^{1-1/\alpha-\gamma}}} e^{-\left|\frac{\iota(y)-\iota(x)}{t^{1-1/\alpha}}\right|^{\alpha}} f(y) \, d\textup{vol} + \mathcal{O}(t^\ell) \]
\begin{align}
&= \int_{\stackrel{y\in\mathcal{M},}{\frac{|\iota(x)-\iota(y)|^{\alpha}}{t^{\alpha-1}}<t^{-\gamma\alpha}}} e^{-\frac{d_g(x,y)^{\alpha} + \mathcal{O}(|\iota(x)-\iota(y)|^{\alpha+2})}{t^{\alpha-1}}} f(y) \, d\textup{vol} + \mathcal{O}(t^\ell) \nonumber \\
&= \int_{\stackrel{y\in\mathcal{M},}{\frac{|\iota(x)-\iota(y)|^{\alpha}}{t^{\alpha-1}}<t^{-\gamma\alpha}}} e^{-\frac{d_g(x,y)^{\alpha}}{t^{\alpha-1}}} f(y) \left(1+ \mathcal{O}\left(\frac{|\iota(x)-\iota(y)|^{\alpha+2}}{t^{\alpha-1}}\right)\right) \, d\textup{vol} + \mathcal{O}(t^\ell) \nonumber \\
&= (1 + \mathcal{O}(t^{2-2/\alpha})) \int_{y\in\mathcal{M}} e^{-\left(\frac{d_g(x,y)}{t^{1-1/\alpha}}\right)^{\alpha}} f(y) \, d\textup{vol} \nonumber
\end{align} 
where we use $|\iota(x)-\iota(y)|<t^{1-1/\alpha}$ to rewrite the error in terms of $t$, and we choose $\ell$ large enough that the $\mathcal{O}(t^{2-2/\alpha})$ term dominates. In the last step we return the integral to the whole manifold which also incurs an error which is higher order than any polynomial in $t$.  

For $\beta = \frac{\alpha}{\alpha-1}$ and $1+\frac{1}{d}<\alpha\leq 2$ recall that,
\BEA
k(t,x,y) = ct^{-d/\beta} e^{-\left(\frac{d_g(x,y)}{t^{1-1/\alpha}}\right)^{\alpha}} = ct^{-d/\beta} e^{-\left(\frac{d_g(x,y)}{t^{1/\beta}}\right)^{ \frac{\beta}{\beta-1}}}\nonumber
\EEA
is a stochastically complete $\beta$-invariant local kernel for some $c>0$; this is the local kernel in part (a) of Theorem~\ref{heatdichotomy}. That is,
\BEA
\mathcal{K}_t 1(x) = \int_{\mathcal{M}} k(t,x,y) \, d\textup{vol} = 1,\nonumber
\EEA
which implies that,
\BEA
c t^{-d/\beta}  \int_{y\in \mathcal{M}} e^{-\left|\frac{\iota(y)-\iota(x)}{t^{1-1/\alpha}}\right|^{\alpha}} f(y)\, d\textup{vol} &=&  \Big(
\int_{y\in\mathcal{M}} k(t,x,y) f(y) \, d\textup{vol}\Big)  (1 + \mathcal{O}(t^{2-2/\alpha}))\nonumber \\  &=& \mathcal{K}_t f(x) (1 + \mathcal{O}(t^{2-2/\alpha}) ).\nonumber
\EEA
Since the $\alpha$-local kernel $h$ differs from its bound, $e^{-\left|\frac{\iota(y)-\iota(x)}{t^{1-1/\alpha}}\right|^{\alpha}}$, by constants, the proof is complete.
\end{proof}

Theorem \ref{localthm} shows that we can use $\alpha$-local kernels to estimate the semigroup of a stochastically complete, $\beta-$scale invariant local kernel of part (a) in Theorem~\ref{heatdichotomy}.  {Thus, Theorem \ref{localthm} achieves the first part of our goal of developing consistent estimators for the two classes of heat kernels.}  In a general metric measure space the associated generator of $\mathcal{K}_t$ will depend on the choice of $\beta = \frac{\alpha}{\alpha-1}$ as shown in \cite{AGrigoryan_TKumagai_2008a}.

It was shown in Theorem 2 of \cite{diffusion} that for $f\in\mathcal{C}^3(\mathcal{M})$ and $\epsilon>0$, 
\begin{align} \label{Heps}
\mathcal{H}_\epsilon f(x) &:=  \epsilon^{-d/2}\int_{y\in \mathcal{ M}} \Psi \left( \frac{|\iota(x)-\iota(y)|^2}{\epsilon} \right) f(y)\, d\textup{vol}  \nonumber \\ &= m_0 f(x) + \epsilon \frac{m_2}{2}(\omega(x)f(x) - \Delta f(x)) + \mathcal{O}(\epsilon^2), \end{align} 
for any exponentially decaying function $\Psi:[0,\infty) \to [0,\infty)$ where $m_0 = \int_{\mathbb{R}^d}\Psi(z)\,dz$ and $m_2=\int_{\mathbb{R}^d}z_j^2\Psi(z)\,dz$ are moments of $\Psi$. Since the upper and lower bounds on an $\alpha$-local kernel $h$ can be written as 
\[ e^{-\left|\frac{\iota(y)-\iota(x)}{t^{1-1/\alpha}}\right|^{\alpha}} = e^{-\left( \frac{|\iota(x)-\iota(y)|^2}{t^{2-2/\alpha}}\right)^{\alpha/2}}  = \Psi \left(\frac{|\iota(x)-\iota(y)|^2}{t^{2-2/\alpha}}\right) \] 
where $\Psi(s) = e^{-|s|^{\alpha/2}}$, setting $\epsilon = t^{2-2/\alpha}$, we have,
\BEA
\int_{y\in \mathcal{M}} h(t,x,y)f(y) \, d\textup{vol}   = \epsilon^{d/2} \mathcal{H}_{\epsilon} {f(x)},\nonumber
\EEA
{where $h(t,x,y) = e^{-\left|\frac{\iota(y)-\iota(x)}{t^{1-1/\alpha}}\right|^{\alpha}}$.}

Substituting the above equation into \eqref{equality1}, we have,
\BEA
 \mathcal{K}_t f(x) (1 + \mathcal{O}(t^{2-2/\alpha}) ) &=& c t^{-d/\beta}   \epsilon^{d/2} \mathcal{H}_{\epsilon} f(x),\nonumber \\
&=& ct^{-d/\beta{+}d/2(2-2/\alpha)} \mathcal{H}_{t^{2-2/\alpha}} f(x)\nonumber \\ 
&=& ct^{-d({-}1+1/\beta{+}1/\alpha)} \mathcal{H}_{t^{2-2/\alpha}} f(x) \nonumber \\ &=& c\mathcal{H}_{t^{2-2/\alpha}} f(x), \label{sgequality}
\EEA
{ the last line is obtained from the fact that the exponent $\frac{1}{\alpha}+\frac{1}{\beta}-1=0$, which can be realized from $\beta=\frac{\alpha}{\alpha-1}$.}
Since $\mathcal{K}_t 1 = 1$ for all $t$ one can show that $c = 1/m_0$.
From the asymptotic expansion in \eqref{Heps}, we obtain,
\BEA
 \mathcal{K}_t f(x) (1 + \mathcal{O}(t^{2-2/\alpha}) ) &=&  f(x) + t^{2-2/\alpha}\frac{m_2}{2 m_0}(\omega(x)f(x) - \Delta f(x)) + \mathcal{O}(t^{2(2-2/\alpha)}). \nonumber
\EEA
This means,
\BEA
\frac{f(x) -  \mathcal{K}_t f(x)}{t} = \mathcal{O}(t^{1-2/\alpha})
\EEA
which diverges as $t\to 0$ for $\alpha<2$. {This is critical since it shows that certain kernels do not have well-defined limits and so may be inappropriate for machine learning applications.  Moreover,} we cannot access the generator of $\mathcal{K}_t$ for $\alpha < 2$ using the kernel $h$. When $\alpha=2$, $\mathcal{K}_t f(x) (1 + \mathcal{O}(t)) = e^{t\Delta}$ (see e.g. \cite{rosenberg}). In this case, \eqref{sgequality} implies that the integral operator $\mathcal{H}_t$ approximates the semigroup of Laplace-Beltrami operator. The associated generator can be obtained using an appropriate algebraic manipulation. That is, we recover the standard diffusion maps algorithm (following \cite{diffusion} we assume $m_2/m_0=2$ which is equivalent to $C_1=C_2=1/4$),
 \BEA
\frac{f(x) - ( \mathcal{H}_t 1(x))^{-1}\mathcal{H}_t f(x)}{t} =   \Delta f(x) + \mathcal{O}(t).\nonumber
\EEA
Note that dividing by $\mathcal{H}_t 1(x)$ cancels the constant $c$ and also removes the $\omega(x)$ term (see Section \ref{leftnormsection} for details). However we should note that the expansion in \eqref{Heps} holds only in the interior of a manifold with boundary or for functions $f$ satisfying Neumann boundary conditions \cite{diffusion}.

\comment{
When $\alpha=2$ the generator can be obtained using the left normalization discussed in Section \ref{leftnormsection}, however we should note that the expansion \ref{Heps} holds only in the interior of a manifold with boundary or for functions $f$ satisfying Neumann boundary conditions \cite{diffusion}.

Numerically, however, for any $\alpha > 1+1/d$ an algebraic manipulation of $\mathcal{H}_{t^{2-2/\alpha}}$ allows us to obtain the standard and time-rescaled Laplacian operators. In particular, for $\alpha=2$, we recover the usual diffusion maps algorithm (following \cite{diffusion} we assume $m_2/m_0=2$ which is equivalent to $C_1=C_2=1/4$),
 \BEA
\frac{f(x) - ( \mathcal{H}_t 1(x))^{-1}\mathcal{H}_t f(x)}{t} =   \Delta f(x) + \mathcal{O}(t).\nonumber
\EEA
note that dividing by $\mathcal{H}_t 1(x)$ cancels the constant $c$ and also removes the $\omega(x)$ term (see Section \ref{leftnormsection} for details).}

If we follow the same algebraic manipulation for the case of $1+1/d < \alpha < 2$, we obtain a rescaled Laplacian,
 \BEA
\frac{f(x) - ( \mathcal{H}_{t^{2-2/\alpha}} 1(x))^{-1}\mathcal{H}_{t^{2-2/\alpha}} f(x)}{t} =  t^{1-2/\alpha} \Delta f(x) + \mathcal{O}(t^{2(2-2/\alpha) - 1}).\label{rescaled_laplacian}
\EEA
where $\frac{1-d}{1+d} < 1-2/\alpha < 0$ and the higher order term is $\frac{1-3d}{1+d} < 2(2-2/\alpha) - 1 < 1$) so as $t\to 0$ the term on the right goes to infinity when $\alpha \neq 2$.  This is due to the fact that the walk dimension of any Riemannian manifold is $\beta = 2$ \cite{grigor2003heat} (notice that $\beta=2$ when $\alpha=2$) and when $\beta \neq 2$ we do not obtain a semi-group on the manifold.  We should note that for real-valued functions, $f$, the time-rescaled Laplacian in \eqref{rescaled_laplacian} is associated with the Fokker-Planck-Kolmogorov type operator \cite{hahn2011fokker} associated to stochastic differential equations driven by driftless constant coefficient fractional Brownian motions (fBM) with Hurst parameter $1-\frac{1}{1+1/d}<H \leq 1/2$; see e.g., Theorem~4.1 of It\^o formula for fBM \cite{bender2003} with arbitrary Hurst parameter $0<H<1$.

Furthermore, for finite $t$ we can still obtain an estimation of the intrinsic Laplacian (up to scalar multiple depending on $t$) on the manifold using \emph{any} value of $\alpha$.  In particular, if we divide by $\epsilon$ rather than $t$ (as in the diffusion maps algorithm) we find
 \BEA
\frac{f(x) - ( \mathcal{H}_{t^{2-2/\alpha}} 1(x))^{-1}\mathcal{H}_{t^{2-2/\alpha}} f(x)}{\epsilon} &=& \frac{f(x) - ( \mathcal{H}_{t^{2-2/\alpha}} 1(x))^{-1}\mathcal{H}_{t^{2-2/\alpha}} f(x)}{t^{2-2/\alpha}} \nonumber \\ &=& \Delta f(x) + \mathcal{O}(t^{1-2/\alpha}) = \Delta f(x) + \mathcal{O}\left(\frac{\epsilon}{t}\right).\nonumber
\EEA
We demonstrate this surprising result numerically in Section \ref{numerics} where the spectrum estimated by the diffusion maps algorithm will be independent of the choice of $\alpha$ for local kernels.

This result has important implications for potential generalizations to non-smooth metric-measure spaces where the walk dimension will typically be unknown.  It suggests that we can use an $\alpha$-local kernel with any value of $1+1/d < \alpha \leq 2$ to approximate the \emph{intrinsic} Laplacian operator up to a scalar multiple.  In particular, if this fact holds beyond the context of manifolds it would allow estimation of the intrinsic Laplacian without needing to know the walk dimension of the space.

\subsection{Nonlocal kernels}\label{nonlocal}

The goal of this section is to consider a class of nonlocal kernel functions on the embedding
space $\mathbb{R}^n$ that gives rise to a semi-group which is an infinitesimal generator of 
a nonlocal process on the manifold. The class of kernels will be called nonlocal kernels.

\begin{definition}[nonlocal kernels]
\label{def:nonloc} 
The kernel $k : [0,\infty) \times \mathcal{M} \times \mathcal{M} \to [0,\infty)$ in 
Definition~\ref{def:scaleinvariance} is called nonlocal if 
$\Phi$ is given by  
\[
 \Phi(s) = (1+s)^{-(d+\beta)} . 
\] 
\end{definition}
The kernel in Definition~\ref{def:nonloc} is associated with stable-like processes
\cite{ZQChen_TKumagai_2003a}. 
In particular when $\beta \in (0,2)$ the heat kernel of the 
$\beta$-stable processes in $\mathbb{R}^d$ is included here. Notice that the generator of the 
$\beta$-stable process is given by the fractional Laplacian 
$\mathcal{L} = (-\Delta)^{\beta/2}$. Moreover, for $x,y\in \mathbb{R}^d$, the heat kernel on $\mathbb{R}^d$ when $\beta = 1$ is given by 
\[
 k(t,x,y) = \frac{c_d}{t^d} \left(1+\frac{|x-y|}{t^2}\right)^{-\frac{d+1}{2}}
 \quad \mbox{with} \quad c_d = \Gamma\left(\frac{d+1}{2}\right)/\pi^{(d+1)/2} ,
\]
which immediately fulfills \eqref{scaleinvariance}. 

In the case of local kernels, as discussed in the previous
section, we are able to construct a localization argument which leads to an 
approximation of the geodesic distance $d_g$ by the Euclidean distance $|\cdot|$. 
However, this is no longer the case for nonlocal kernels as we do not have exponential 
decay at the tails. { In particular, Lemma \ref{localization} relies on the $\alpha$-local kernel decaying faster than all polynomials, and for any nonlocal kernel (with polynomial decay) this result will fail.  In particular, for $t$ sufficiently small we have
\begin{align}  &t^\ell \int_{|\iota(y)-\iota(x)|>t^{1/\beta-\gamma}}t^{-d/\beta} \left(1+ \left|\frac{\iota(x)-\iota(y)}{t^{1/\beta}}\right| \right)^{-d-\beta} \, d\textup{vol} \nonumber \\ 
&= t^{\ell} \int_{|s|>t^{-\gamma}} \left(1+ \left| s \right| \right)^{-d-\beta} \, ds + h.o.t.  \nonumber \\  &= \mathcal{O}(t^{\ell+\gamma(d+\beta)}) \nonumber \end{align}
where $\gamma \in (0,1)$ and $s = (x-y)/t^{1/\beta}$ and the Jacobian of the transformation is $t^{d/\beta}$ up to higher order terms. So for $\ell<-\gamma(d+\beta)$ the expression above diverges as $t\to 0$.  Since the tail integrals are large for nonlocal kernels we cannot localize the integral to a small region around $x$ the way we could for local kernels.  This localization was the key to avoiding explicit estimation of geodesic distances (since in a sufficiently small region the geodesic distance becomes well approximated by the Euclidean distance in the embedding space).  For nonlocal kernels this is no longer the case and we need an explicit estimator of the geodesic distance.

As a result we are confronted with a problem of estimation of the geodesic 
distance $d_g$.  In order to address this problem, we step back from the continuous integral operators that we have been analyzing and return to our initial assumption that we are given data sampled on the manifold.  Ultimately, we estimate the continuous integral operator using these discrete data sets, so we only need to evaluate the kernel functions on pairs of discrete data points.  Thus, we only need an approach to estimate the geodesic distances between all pairs of data points.  This pairwise geodesic distance approximation} is the same problem that motivated the Isomap algorithm \cite{tenenbaum2000global,MBernstein_VDSilva_JCLangford_JBTenenbaum_2000a}.
As in the Isomap algorithm, we will approximate the geodesic distance with the graph distance which is accessible using data sampled on the manifold.  

Specifically, 
given a finite set of data points $\{x_i\}$ in $\mathcal{M}$ and a graph $G$ 
with vertices $\{x_i\}$, we can approximate the geodesic distance, $d_g$, by the so-called \emph{graph distance} given by
\[
 d_G(x,y) := \min_P\left( |\iota{(x_0)}-\iota{(x_1)}| + \dots + |\iota{(x_{p-1})} - \iota{(x_p)}| \right)
\]
where the minimum over all paths $P$ can be computed using Dijkstra's algorithm without prior knowledge of the manifold $\mathcal{M}$. 
Here $x,y \in \{x_i\}$ and $P = (x_0,\dots,x_p)$ varies over all paths along the 
edges of $G$ connecting $x = x_0$ to $y = x_p$.

In the remainder of this section, we will use the error estimate of this approximation, which was derived in \cite{MBernstein_VDSilva_JCLangford_JBTenenbaum_2000a}, to determine the error bound in approximating the nonlocal kernels  from Definition~\ref{def:nonloc} and their associated semigroup operators and generators. To give a complete discussion, we state the following relevant definition. 

\begin{definition}[minimum branch separation]
\label{def:coverminb}
Let $r_0 = r_0(\mathcal{M})$ denotes the minimum radius of curvature of 
$\mathcal{M}$. 
Then the minimum branch separation $s_0 = s_0(\mathcal{M})$ is defined
to be the largest positive number such that if $|\iota(x)-\iota(y)| < s_0$ then  
$d_g(x,y) \le \pi r_0$, for $x,y \in \mathcal{M}$
\end{definition}

We note that since $\mathcal{M}$ is assumed to be compact and $C^3$ the curvature is bounded above so $r_0>0$ is well-defined, and similarly the minimum branch separation is well defined since $\iota$ is $C^3$.  After these preparations we are now ready to state the main result of 
\cite[Main Theorem A]{MBernstein_VDSilva_JCLangford_JBTenenbaum_2000a}. 

\begin{theorem}
\label{thm:Ten}
Let $\mathcal{M}$ be a compact manifold embedded in the Euclidean space 
$\mathbb{R}^n$ with minimum radius of curvature $r_0$ and minimum branch
separation $s_0$. Let $\{x_i\}$ be a finite set of data points in $\mathcal{M}$
and they generate a graph $G$. Let $0 < \varepsilon_1, \varepsilon_2 < 1$.
Under the following assumptions
 \begin{enumerate}
  \item {\bf Graph condition I.} $G$ contains all edges $xy$ so that 
   \[{|\iota{(x)}-\iota{(y)}|} \le \frac{4}{\varepsilon_2} C(d_g,\{x_i\}),\]
   where $C(d_g,\{x_i\}):=\max_{x\in \mathcal{M}}\min_{\{x_i\}} d_g(x_i,x)$. 
  \item {\bf Graph condition II.} The edges $xy$ in $G$ fulfills 
   \[ {|\iota{(x)}-\iota{(y)}|} \le (2/\pi) r_0 \sqrt{24 \varepsilon_1} < s_0. \]     
  \item The manifold $\mathcal{M}$ is geodesically convex, i.e., any two
   points $x,y \in \mathcal{M}$ are connected by a geodesic of length $d_g(x,y)$. 
 \end{enumerate} 
 Then for all $x,y \in \{x_i\}$ we have:
 \begin{equation}\label{eq:dgdG}
  (1-\varepsilon_1) d_g(x,y) \le d_G(x,y) \le (1+\varepsilon_2) d_g(x,y) . 
 \end{equation}
\end{theorem}

\begin{remark}[geodesically convex]
 {\rm 
  In $\mathbb{R}^n$ convex domains are geodesically convex, so are compact 
 Riemannian manifolds without boundary. In general a compact Riemannian
  manifold is geodesically convex if and only if its boundary is convex. 
 }
\end{remark}

{We should also note that the number of data points required to fulfill the assumptions of Theorem \ref{thm:Ten} will depend strongly on the uniformity of the sampling density.  We assume that the manifold is the set of points with positive sampling density and that this set is compact and the density is smooth, which implies that the density is bounded away from zero (since it obtains its lower bound on the compact set and is always positive).  The positive lower bound on the sampling density insures that for a sufficiently large data set, the minimum distance between points can be made arbitrarily small with high probability.  Of course, a smaller lower bound (meaning less uniform sampling) will lead to higher data requirements.}

Notice that by controlling the tolerances $\varepsilon_1, \varepsilon_2$
one can control how well $d_G$ approximates $d_g$. For simplicity of exposition, 
we let $\varepsilon := \min \{\varepsilon_1, \varepsilon_2 \}$ in 
Theorem~\ref{thm:Ten} and we obtain the following estimate
 \begin{equation}\label{eq:dgdG_1}
  |d_G(x,y)-d_g(x,y)| \le \varepsilon d_g(x,y) .
 \end{equation}
We further emphasize that the constant $C(d_g,\{x_i\})$ 
measures how well the points $\{x_i\}$ covers $\mathcal{M}$ and Graph condition I
implies that such a covering should be fine enough. The Graph condition I when
compared to condition II implies that the covering scale should be small enough in comparison to 
the scales implied by the radius of curvature and branch separation. 

Next we will exploit the result of Theorem~\ref{thm:Ten} to approximate the 
nonlocal kernels in Definition~\ref{def:nonloc}. 

\begin{theorem}[approximation of nonlocal kernel]
\label{thm:abkerest}
 Let the assumptions of Theorem~\ref{thm:Ten} hold and set $0<\epsilon:=\min\{\epsilon_1,\epsilon_2\}<1$. Then for every $x,y \in \{x_i\}$
 and $t > 0$, there exists a sufficiently small $\varepsilon>0$ such that, 
 \[
   |k(t,x,y) - k_G(t,x,y)| \le  
       C_{d,\beta} \ d_g(x,y)^2 \left( t^{-\frac{(d+2)}{\beta}} \varepsilon \right)
     + \mathcal{O}\left( t^{-\frac{(d+4)}{\beta}} \varepsilon^2 \right),  .   
 \]
where $k_G(t,x,y):= C_{d,\beta}t^{-d/\beta}\Phi\Big(\frac{d_G(x,y)}{t^{1/\beta}}\Big)$.
\end{theorem}
\begin{proof}
 Let $x, y \in \{x_i\}$ then for a positive constant $C_{d,\beta}$ we have 
 \begin{align*}
   k(t,x,y) - k_G(t,x,y) 
    &= \frac{C_{d,\beta}}{t^{d/\beta}}
     \left( \left(1+\frac{d_g(x,y)}{t^{1/\beta}}\right)^{-(d+\beta)} 
      -\left(1+\frac{d_G(x,y)}{t^{1/\beta}}\right)^{-(d+\beta)} \right) \nonumber \\
 \end{align*}
From \eqref{eq:dgdG_1} we have that 
$-\varepsilon d_g(x,y) \le d_G(x,y) - d_g(x,y) \le \varepsilon d_g(x,y)$ thus
$d_G(x,y) \le (1+\varepsilon) d_g(x,y)$ whence 
$\left(1+\frac{(1+\varepsilon)d_g(x,y)}{t^{1/\beta}}\right)^{-(d+\beta)} \le \left(1+\frac{d_G(x,y)}{t^{1/\beta}}\right)^{-(d+\beta)}$. By denoting
$\zeta(t,\beta) := 1+\frac{d_g(x,y)}{t^{1/\beta}}$, we obtain,
 \begin{align}\label{eq:kest1} 
   |k(t,x,y) - k_G(t,x,y)| 
    &\le \frac{C_{d,\beta}}{t^{d/\beta}}
     \left| \zeta(t,\beta)^{-(d+\beta)} 
      -\left(\zeta(t,\beta)+\varepsilon\frac{d_g(x,y)}{t^{1/\beta}}\right)^{-(d+\beta)} \right| \nonumber\\
      &=C_{d,\beta} \ t^{-\frac{d}{\beta}} \zeta(t,\beta)^{-(d+\beta)}
     \left| 1 
      -\left(1+\varepsilon \ {\zeta(t,\beta)^{-1}} \frac{d_g(x,y)}{t^{1/\beta}}\right)^{-(d+\beta)} \right|.
 \end{align}      
For a sufficiently small $\varepsilon$, we can expand
\begin{equation}\label{eq:aux_est}
    1 -\left(1+\varepsilon \widetilde{\zeta}(t,\beta)\right)^{-(d+\beta)}
     = \varepsilon \widetilde{\zeta}(t,\beta) 
      + \mathcal{O} \left( \varepsilon^2 \widetilde{\zeta}(t,\beta)^2 \right), 
 \end{equation}
where we have denoted $\widetilde{\zeta}(t,\beta) := {\zeta(t,\beta)^{-1} (\zeta(t,\beta)-1)= \frac{d_g(x,y)}{t^{1/\beta}+d_g(x,y)}<1}$.
Since $\zeta(t,\beta)^{-(d+\beta)} \le 1$, the proof is completed by substituting  \eqref{eq:aux_est}
into the absolute value term in the right hand side of \eqref{eq:kest1}.
\end{proof}
 
Recall that the semigroup $\mathcal{K}_t$ generated by the nonlocal kernel $k$ is defined as
 \[
    \mathcal{K}_t f(x) = \int_{\mathcal{M}} k(t,x,y) f(y)\, d\textup{vol} , \quad 
    \forall x\in \mathcal{M} \mbox{ and } f \in L^2(\mathcal{M},d_g). 
 \]
We shall approximate $\mathcal{K}_t$ by 
 \[
    \mathcal{K}_{t,G} f(x) = \mathcal{Q}[k_G(t,x,\cdot) f(\cdot)] , \quad 
    \forall x \in \{x_i\} 
 \]
where $\mathcal{Q}$ indicates a quadrature approximation of the integral in the
definition of $\mathcal{K}_t$. Next we shall provide approximation error estimate 
between $\mathcal{K}_t$ and $\mathcal{K}_{t,G}$. 
 
\begin{lemma}[semigroup estimate]
\label{lem:semigest_nonloc}
 Let the assumptions of Theorem~\ref{thm:abkerest} holds. Then for every 
 $x \in \{x_i\}$ and $t > 0$ we have 
 \[
   |\mathcal{K}_tf(x)-\mathcal{K}_{t,G}f(x)| 
   \le |{Quad}_{err}(t)| 
     + C_{d,\beta} \; t^{-\frac{(d+2)}{\beta}} \, \varepsilon \,
       \mathcal{Q}\left[d_g(x,\cdot)^2 |f(\cdot)|\right] 
     +  \mathcal{O}\left( t^{-\frac{(d+4)}{\beta}} \varepsilon^2 \right)  ,
 \]
where  
\begin{equation}\label{eq:quaderr}
 {Quad}_{err}(t) := \int_{\mathcal{M}} k(t,x,y) f(y)\, d\textup{vol} 
    - \mathcal{Q}[k(t,x,\cdot)f(\cdot)] . 
\end{equation}
\end{lemma}
\begin{proof}
 From the definition of $\mathcal{K}$ and $\mathcal{K}_G$ we obtain that 
 \begin{align*}
  \mathcal{K}_tf(x)-\mathcal{K}_{t,G}f(x) 
   &= \int_{\mathcal{M}} k(t,x,y) f(y)\, d\textup{vol} 
    - \mathcal{Q}[k_G(t,x,\cdot)f(\cdot)] \\
   &= {Quad}_{err}(t) 
    + \mathcal{Q}[\left(k(t,x,\cdot)-k_G(t,x,\cdot)\right)f(\cdot)] 
 \end{align*}    
where ${Quad}_{err}(t)$ denotes the quadrature error \eqref{eq:quaderr}. 
Using the error estimate in Theorem~\ref{thm:abkerest} we then arrive at 
 \[
   |\mathcal{K}_tf(x)-\mathcal{K}_{t,G}f(x)| 
   \le |{Quad}_{err}(t)| 
     + C_{d,\beta} \left( t^{-\frac{(d+2)}{\beta}} \varepsilon \right)
       \mathcal{Q}\left[d_g(x,\cdot)^2 |f(\cdot)|\right] 
     +  \mathcal{O}\left( t^{-\frac{(d+4)}{\beta}} 
         \varepsilon^2 \right) 
 \]
and the proof is complete. 
\end{proof}
We conclude this section with an error estimate for the generator. 

\begin{theorem}
 Let the assumptions of Lemma~\ref{lem:semigest_nonloc} holds. Then for every 
 $x \in \{x_i\}$, we have
 \[
  \lim_{t\rightarrow 0} \frac{f(x)-\mathcal{K}_{t,G}f(x)}{t} = (-\Delta)^{\beta/2} f(x) 
 \]
 where $(-\Delta)^{\beta/2}$ is a generator of $\mathcal{K}_t$, provided 
 $t^{-1} \, {Quad}_{err}(t)$ converges to 0 as $t \rightarrow 0$. 
\end{theorem}
\begin{proof}
 Consider the limit 
 \begin{align*}
  \lim_{t\rightarrow 0} \frac{f(x)-\mathcal{K}_{t,G}f(x)}{t} 
   &= \lim_{t\rightarrow 0} \frac{f(x)-\mathcal{K}_{t}f(x)}{t} 
    +\lim_{t\rightarrow 0} \frac{\mathcal{K}_t f(x)-\mathcal{K}_{t,G}f(x)}{t} \\
   &= (-\Delta)^{\beta/2} f(x) 
    +\lim_{t\rightarrow 0} \frac{\mathcal{K}_t f(x)-\mathcal{K}_{t,G}f(x)}{t} 
 \end{align*}
where { $(-\Delta)^{\beta/2}$ is the generator of $\mathcal{K}_t$ in compact manifold setting (see \cite[Theorem 4.2]{gimperlein2014heat}).} It then
remains to show that, 
\[\lim_{t\rightarrow 0} \frac{\mathcal{K}_t f(x)-\mathcal{K}_{t,G}f(x)}{t} = 0.\] 
Using Lemma~\ref{lem:semigest_nonloc} we deduce that 
\[
 \frac{|\mathcal{K}_t f(x)-\mathcal{K}_{t,G}f(x)|}{t} 
 \le t^{-1} \, |{Quad}_{err}(t)| + C_{d,\beta} t^{-\frac{(d+2)}{\beta}-1} \varepsilon ,
\]
Choose $\varepsilon >0$ such that $\lim_{t\rightarrow 0} 
t^{-\frac{(d+2)}{\beta}-1} \varepsilon = 0$. Then the result follows after using the assumption 
$\lim_{t\rightarrow 0} t^{-1} \, |{Quad}_{err}(t)| = 0$.
\end{proof}

The quadrature error depends on how the data points are generated, and will typically decay as the number of data points increases and diverge as $t\to 0$ for a fixed number of data points.  Thus, insuring the condition $t^{-1} \, {Quad}_{err}(t)$ converges to 0 as $t \rightarrow 0$ will require assuming that the number of data points $N$ grows sufficiently fast as $t \to 0$.  We will discuss this issue more in the next section.

\section{Application to diffusion maps allowing for non-local kernels}\label{application}

In this section we {extend} the diffusion maps algorithm for non-local kernels.  In particular, the diffusion maps algorithm is designed to {be invariant to the sampling density of the data.  We will show below that the diffusion maps algorithm must be modified when non-local kernels in order to maintain this invariance.  Moreover, non-local kernels require approximation of geodesic distances on the manifold, and we use Dijkstra's algorithm to compute graph distances in order to approximate the geodesic distances.}

\subsection{Compensating for non-uniform sampling of data (right normalization)}\label{nonuniform}

At this point we should explain that for data applications one typically assumes only that each data point is sampled from a distribution having a smooth density function $q:\mathcal{M} \to (0,\infty)$ where $q$ is the density relative to the natural volume form on the manifold $d\textup{vol}$.  Since the data points are the only information we have about the manifold, the only quadrature rule that we have access to is the Monte-Carlo quadrature which says that
\[ \lim_{N\to\infty} \frac{1}{N}\sum_{j=1}^N f(x_j) = \mathbb{E}[f(X)] = \int_{y\in \mathcal{M}} f(y)q(y) \, d\textup{vol}. \]
So given a kernel function $k$ we can estimate the kernel integral operator by
\BEA
\lim_{N\to\infty} \mathcal{Q}[k(t,x,\cdot)f(\cdot)] &=& 
\lim_{N\to\infty} \frac{1}{N}\sum_{j=1}^N k(t,x,x_j)f(x_j) \nonumber \\ &=& \int_{y\in \mathcal{M}} k(t,x,y) f(y)q(y) \, d\textup{vol} = e^{t(-\Delta)^{\beta/2}}(fq)(x), \nonumber
\EEA
where the last equality holds for nonlocal kernels $k$ for $0<\beta<2$ as well as the local (Gaussian) kernel with $\beta=2$.
The fact that the sampling density, $q$, changes the operator as shown above was observed in the original diffusion maps paper \cite{diffusion} and they introduced a method to correct this by pre-dividing by a kernel density estimate.  Namely, let $J(t,x,y)$ be any smooth kernel function which decays exponentially in $|x-y|$ (such as a standard Gaussian kernel), then we can use $J$ to estimate the density $q(x_i)$ since
\[ \lim_{N\to\infty}\frac{1}{N}\sum_{j=1}^N J(\tau,x_i,x_j) = \int_{y\in \mathcal{M}} J(\tau,x_i,y) q(y) \, d\textup{vol} \propto q(x_i) + \mathcal{O}(\tau) \]
where the proportionality constant is due to the integral of the kernel $J$. {A detailed analysis of the error of the discrete estimate, $\hat q(x_i)=  \frac{1}{N}\sum_{j=1}^N J(t,x_i,x_j)$, of $q(x_i)$\cite{singer,variablebandwidth},  shows that the error variance is of order $\mathcal{O}(N^{-1}\tau^{-1-d/2})$, where $d$ is the dimension of the manifold.  Notice that since the Monte-Carlo error diverges as $\tau\to 0$, one finds that the optimal $\tau$ value is a function of the amount of data, $N$. Thus,} we can pre-divide by the estimate to fix our quadrature rule
\[  \lim_{N\to\infty} \frac{1}{N}\sum_{j=1}^N \frac{k(t,x,x_j)}{\hat q(x_j)}f(x_j) \propto \int_{y\in \mathcal{M}} k(t,x,y) f(y) \frac{q(y)}{q(y)+\mathcal{O}({t})} \, d\textup{vol} \]
and as long as {$t \ll \min_y q(y)$ then $\frac{q(y)}{q(y)+\mathcal{O}(t)} \approx 1$} and we will have a good estimate of the desired integral.

Often we only want to evaluate the kernel on the data points, so we compute the kernel matrix $K_{ij} = k(t,x_i,x_j)$.  In this case, pre-dividing by the density estimate corresponds to right multiplication by the inverse of a diagonal matrix $D_{ii} = \sum_{j=1}^N J(t,x_i,x_j)$.  In fact, in the original diffusion maps algorithm, the kernel matrix $k$ had exponential decay, so they actually set $J=k$ and in this case the normalized kernel matrix $KD^{-1}$ is simply dividing each column by the column sum ($K$ is symmetric).  {We should note that the diffusion maps algorithm uses a symmetric form of this normalization, $D^{-1}KD^{-1}$, however the next step (left normalization) cancels the effect of multiplying by $D^{-1}$ on the left.  The symmetric version of the `right normalization' is important for maintaining a symmetric matrix for numerical reasons, for a more detailed explanation see Section \ref{algorithm} below and \cite{localkernels}.}

Since polynomial kernels do not give density estimators with the same error as exponential kernels, we will always set $J$ to be a standard Gaussian kernel.  The matrix $KD^{-1}$ is called the `right-normalized' kernel and is an estimator of the heat kernel in the sense that if $f$ is a smooth function defined on the manifold, then we can discretize $f$ as $\vec f_i = f(x_i)$ and the matrix-vector product $\left(KD^{-1}\vec f \right)_i$ is a pointwise estimator of the operator $e^{t(-\Delta)^{\beta/2}}f(x_i)$ (up to a constant which may depend on $t$).

\subsection{Removing proportionality constants and low order error terms (left normalization)}\label{leftnormsection}

As we have seen above, the choice of kernel function and the normalizations to remove sampling bias can introduce many constants of proportionality.  In this section we motivate the `left normalization' as a method of removing the influence of these unknown constants. It turns out (as first shown in \cite{diffusion}), that this normalization removes certain types of error terms.  Recall that the heat semigroup was assumed to be stochastically complete, namely $\mathcal{K}_t 1(x) = 1$ for all $t$.  The discrete version of this property is that the matrix which approximates the operator should be row stochastic (each row should sum to 1).  Notice that even if the appropriate constants in the expansion of $\mathcal{K}_t$ were used in constructing the discrete version, due to the higher order error terms the resulting matrix would not be exactly row stochastic.  

To motivate this normalization, assume that for a kernel function $\tilde{k}$ we have 
\[ \int_{\mathcal{M}} \tilde{k}(t,x,y)f(y)\,d\textup{vol} = c(x)t^a \mathcal{K}_t f(x)(1+t^{b_1} \omega(x) + \mathcal{O}(t^{b_2})) \] 
for some $a \in\mathbb{R}$ and $b_2 \geq b_1 > 0$.  Since $\mathcal{K}_t 1 = 1$, plugging $f(x)=1$ into the above equation we have $\int_{\mathcal{M}} \tilde{k}(t,x,y)\,d\textup{vol} = c(x) t^a(1+t^{b_1} \omega(x) + \mathcal{O}(t^{b_2}))$ so
\[ \frac{\int_{y\in \mathcal{M}}\tilde{k}(t,x,y)f(y)\, d\textup{vol}}{\int_{y\in \mathcal{M}}\tilde{k}(t,x,y)\, d\textup{vol}} = \frac{\mathcal{K}_t f(x)(1+t^{b_1} \omega(x) + \mathcal{O}(t^{b_2}))}{(1+t^{b_1}\omega(x) + \mathcal{O}(t^{b_2}))} = \mathcal{K}_t f(x)(1+\mathcal{O}(t^{b_2})).  \]
Notice that order-$t^{b_2}$ terms do not necessarily cancel since they could depend on derivatives of $f$, the cancellation of the $t^{b_1}$ term is only valid since the function $\omega(x)$ is assumed to be fixed and independent of $f$. { In fact, letting $\tilde{k}$ to be $h$ in \eqref{kernelh}, this equation is identical to \eqref{equality1} with $b_2=2-2/\alpha$.}  

The final step in the standard diffusion maps algorithm is motivated by the above computation, and estimates the normalized heat kernel (thus removing any proportionality constants appearing in the kernel or constants arising from the right normalization).  Since $KD^{-1}$ is an estimator of the heat kernel, we can estimate applying the kernel to the identity function by computing the row sums $\hat D_{ii} = \sum_{j=1}^N \left(KD^{-1}\right)_{ij}$.  The diagonal matrix $\hat D_{ii}$ containing the row sums is then used to divide each row by the row sum, forming the `left normalized' kernel matrix $\hat D^{-1}KD^{-1}$.  This is the matrix we would like to compute the eigenvalues and eigenvectors of, however it is clearly not symmetric, so in the next section we discuss a numerical scheme which converts this to a symmetric eigenproblem.

\subsection{Numerical dichotomy in the diffusion maps algorithm}\label{algorithm}

Given data $\{\iota(x_i)\}_{i=1}^N \subset \iota(\mathcal{M})\subset \mathbb{R}^n$ (where $x_i \in \mathcal{M}$ are sampled on the manifold but only their embedded coordinates $\iota(x_i)$ are available as a data set) the first step of the algorithm will always be to compute the matrix of pairwise distances $d_{ij} = |\iota(x_i) - \iota(x_j)|$.  We should note that in fact the algorithm will only require the distances to the $\kappa$-nearest neighbors of each point, however determining an appropriate neighbor parameter $\kappa$ will depend on the bandwidth chosen (with bandwidth parameter $\epsilon$), and so for simplicity we will simply consider computing all of the pairwise distances.  

Next we need to choose a bandwidth $\epsilon$ which will have `units' of `distance-squared', since we will ultimately apply our kernel function $\Phi$ to the ratio $\frac{d}{\sqrt{\epsilon}}$.  In the examples in the next section we will consider a large range of bandwidth parameters to demonstrate how different kernel functions respond to the bandwidth.  For more information on tuning the bandwidth we refer the reader to \cite{IDM}.  We only mention that a good heuristic is to take the average of squared distances from each point to its $\kappa$-th nearest neighbor where $\kappa$ is typically on the order of $\log N$.  

We now enter into the dichotomy, in the first case (local kernels) we can use the standard diffusion maps algorithm \cite{diffusion} (described below).  However, in the second case (nonlocal kernels) we will need to first apply Dijkstra's algorithm to estimate the geodesic distances, and we will also need to modify the normalization procedure used in diffusion maps.

In the first case of the dichotomy we consider an $\alpha$-local kernel such as $\Phi(s) = e^{-s^\alpha}$.  In this case, because the kernel has fast decay and localizes the distances, we can directly apply the kernel to the matrix of pairwise distances to form the kernel matrix $K_{ij} = \Phi\left(\frac{d_{ij}}{\sqrt{\epsilon}}\right)$.  Here we are implicitly choosing $\sqrt{\epsilon} = t^{1-1/\alpha} = t^{1/\beta}$ so we can define $t = \epsilon^{\beta/2}$ (recall that $\beta = \frac{\alpha}{\alpha-1}$).  Next, we apply the `right normalization', defining a diagonal matrix $D_{ii} = \sum_{j=1}^N K_{ij}$ and the normalized kernel matrix $\tilde K = D^{-1}KD^{-1}$.  Notice that instead of only dividing the columns by the column sums (as described in the previous section) we divide both the rows and columns by the column sums.  This is to maintain the symmetry of the matrix $\tilde K$, and one can easily see that the left multiplication by $D^{-1}$ will be cancelled out in the `left normalization' step.  

The next step is the `left normalization', where we define a diagonal matrix $\tilde D_{ii} = \sum_{j=1}^N \tilde K_{ij}$ and the final kernel matrix $\tilde D^{-1}\tilde K$, however this would not yield a symmetric eigenproblem.  (Notice that $\tilde D$ is not the same as $\hat D$, but one can show that $\tilde D^{-1}\tilde K$ is the same as $\hat D^{-1} K D^{-1}$ from the previous section).  Instead of solving the eigenproblem $\tilde D^{-1}\tilde K v = \lambda v$, we first multiply both sides by $\tilde D^{1/2}$ to write the problem as
$\tilde D^{-1/2}\tilde K \tilde D^{-1/2}\tilde D^{1/2}v = \lambda \tilde D^{1/2}v$.  
Then setting $w = \tilde D^{1/2}v$, we see that we can solve the symmetric eigenproblem for $\hat K = \tilde D^{-1/2}\tilde K \tilde D^{-1/2}$ to find the eigenvalues $\eta_j$ and eigenvectors $\psi_j$ of the symmetric matrix $\hat K$.  Finally, to produce estimates of the eigenfunctions we form the vectors $\varphi_j = \tilde D^{-1/2}\psi_j$ so that $\varphi_j$ solves the original non-symmetric eigenvalue problem.  Alternatively one can use a generalized eigensolver, but we have found this approach to be more numerically stable.  Finally, the eigenvalues $\eta_j$ estimate the eigenvalues of the semigroup $e^{-t\Delta}$, so to estimate the eigenvalues of $\Delta$ we set $\lambda_j = \frac{-\log(\eta_j)}{t} = \frac{-\log(\eta_j)}{\epsilon^{\beta/2}}$.

In the second case of the dichotomy, we consider a non-local kernel such as $\Phi(x) = (1+s)^{-(d+\beta)}$.  In this case, because the kernel does not have fast decay, we need to first estimate the geodesic distances.  Since we only have access to the data in the embedding space, we first compute the Euclidean distances $d_{ij} = |\iota(x_i) - \iota(x_j)|$, however only the short Euclidean distances will be good approximations of geodesic distances.  However, based on the results in Section \ref{nonlocal}, we can approximate the geodesic distance by first building a weighted graph containing only edges between points with Euclidean distance less than $\delta$ and then computing the graph (shortest path) distance $d_G(\iota(x_i),\iota(x_j))$ between all pairs of points using Dijkstra's algorithm.  Once we estimate the matrix of graph distances $G_{ij} = d_G(\iota(x_i),\iota(x_j))$, then we can apply the kernel $\Phi$ to form the kernel matrix $K_{ij} =\Phi\left(\frac{G_{ij}}{\sqrt{\epsilon}}\right)$.  The algorithm then proceeds exactly as above, applying the two normalizations and solving the symmetric eigenproblem.  In this case, the eigenvalues $\eta_j$ estimate the eigenvalues of the semigroup $e^{t(-\Delta)^{\beta/2}}$, so to estimate the eigenvalues of $(-\Delta)^{\beta/2}$ we define $\lambda_j$ the same as above namely, $\lambda_j = \frac{-\log(\eta_j)}{t} = \frac{-\log(\eta_j)}{\epsilon^{\beta/2}}$.

{
\begin{algorithm}[H] 
{
\caption{Fractional Diffusion Map Algorithm} 
\label{FDM} 
\begin{algorithmic} 
	\State {\bf Inputs:} Data set, $\{\iota(x_i)\}_{i=1}^N \subset \iota(\mathcal{M}) \subset \mathbb{R}^n$, fractional power, $\beta$, bandwidth, $\epsilon$, intrinsic dimension, $d$, and number of requested eigenvectors, $\ell$
	\State {\bf Outputs:} $N\times N$ matrix, $H$, approximating the fractional heat kernel, eigenvectors $v_i$ and eigenvalues $\lambda_i$ estimating the eigenfunctions and eigenvalues of the fractional Laplacian
	\State
	\Indent
	\State Compute the $N \times N$ matrix of pairwise distances $A_{ij} = |\iota(x_i) - \iota(x_j)|$
	\State Compute the $N\times 1$ density estimate 
	$\hat q_i = \frac{(2\pi\epsilon)^{-d/2}}{N} \sum_{j=1}^N \exp\left(-\frac{A_{ij}^2}{2\epsilon} \right)$
	\State Set $t = \epsilon^{\beta/2}$
	\If{$\beta \geq 2$} 
		\State Set $\alpha = \frac{\beta}{\beta-1}$
		\State Compute the $N\times N$ local kernel matrix $K_{ij} = \exp\left(-\left(\frac{A_{ij}}{\sqrt{\epsilon}}\right)^{\alpha} \right)$
	\Else 	
		\State Zero out each entry of $A$ with $A_{ij} \geq \sqrt{\epsilon}$
		\State Construct weighted graph $G$ with $N$ nodes and adjacency matrix $A$
		\State Compute the $N\times N$ matrix $d_G$ of pairwise graph distances in $G$
		\State Compute the $N\times N$ nonlocal kernel matrix $K_{ij} = \left(1+\frac{(d_G)_{ij}}{\sqrt{\epsilon}}\right)^{-d-\beta}$
	\EndIf
	\State Form the diagonal `right normalization' matrix with $D_{ii} = \hat q_i$
	\State Perform the symmetric `right normalization', $\tilde K = D^{-1}K D^{-1}$
	\State Form the diagonal `left normalization' matrix with $\tilde D_{ii} = \sum_{j=1}^N \tilde K_{ij}$
	\State Define the Markov matrix, $H = \tilde D^{-1}\tilde K$
	\State Perform the symmetric `left normalization', $\hat K = \tilde D^{-1/2}\tilde K \tilde D^{-1/2}$
	\State Ensure numerical symmetry by setting $\hat K = (\hat K + \hat K^\top)/2$
	\State Compute the $\ell+1$ eigenvectors $\psi_0,...,\psi_{\ell}$ of $\hat K$ with maximal eigenvalues  \[1=\eta_0 \geq \eta_1 \geq \cdots \geq \eta_\ell \]
	\State Compute the fractional Laplacian eigenvalues $\lambda_i = - t^{-1}\log \eta_i$
	\State Compute the approximate eigenfunctions $\varphi_i = \tilde D^{-1/2}\psi_i$
	\EndIndent
	\State Return heat kernel $H$, approximate eigenfunctions $\varphi_i$, and eigenvalues $\lambda_i$, satisfying $H\varphi_i = e^{t\lambda_i}\varphi_i$
    \State 
\end{algorithmic}
}
\end{algorithm}
}


\section{Numerical examples}\label{numerics}

In this section we show some numerical results to verify the above theory and to compare the different kernel functions.  In particular we will be interested in the effect of the number of data points, the bandwidth parameter $t$, and how the data points are distributed on the manifold.  Tuning the bandwidth $t$ is a matter of balancing the bias error with the quadrature error.  Specifically, the bias error is the $\mathcal{O}(t^{2-2/\alpha})$ term in Theorem \ref{localthm}, and the {quadrature error is the} $\mathcal{O}(t^{-\frac{d+2}{\beta}}\varepsilon)$ term in Theorem \ref{lem:semigest_nonloc}, notice that $\varepsilon$ must depend on $t$ so that $\lim_{t\to 0} t^{-\frac{d+2}{\beta}}\varepsilon = 0$ so the bias term is decreased by taking $t$ small.  For randomly sampled data the variance/quadrature error is $\mathcal{O}(N^{-1}t^{-1-d/2})$ as described in Section \ref{nonuniform} above, and the variance term is decreased by taking $t$ large.  As we will see below, the quadrature error is significantly smaller when the data set is a uniform grid of points rather than randomly sampled data points.  Thus for a uniform grid of data points we will be able to take $t$ much smaller and obtain much lower error with very small data sets.  Of course, uniform grids are unlikely to occur in real data sets and so these results are intended only to demonstrate and numerically validate the above theory.  Results for the randomly sampled data are much more indicative of what should be expected for most data sets.  

Note that {the analysis of diffusion maps begins with the pointwise convergence of \cite{diffusion} that is generalized above, but a more subtle issue is the spectral convergence.  A series of results starting with \cite{von2008consistency} and more recently \cite{spec1,spec2,spec3,berry2020spectral} show that for compact manifolds the eigenvectors of the diffusion maps discrete Laplacian matrix converge to the eigenfunctions of the Laplacian operator on the underlying manifold.  Thus, the diffusion maps algorithm} produces eigenvectors $\varphi_j$ such that the $i$-th entry is an estimate of the true eigenfunction evaluated on the data set $i$-th data point, so $(\varphi_j)_i \approx \phi_j(x_i)$.  Thus our standard measure of error will be the root mean squared error (RMSE) computed on the data set as $\left( \frac{1}{N}\sum_{i=1}^N ((\varphi_j)_i - \phi(x_i))^2 \right)^{1/2}$.  To evaluate { the accuracy of the eigenvalues, one aspect that we can examine is}  the power law associated to the growth of eigenvalues.  Weyl's law states that the eigenvalues of the Laplace-Beltrami operator have a power law growth $\lambda_j \propto j^{d/2}$ for $j$ large {(for the examples under consideration below this power law appears accurate for $j$ fairly small, but generally it may require $j$ large).}  For the fractional Laplacians the eigenvalues are simply raised to the power $\beta/2$ so we have $\lambda_j \propto j^{\beta d/4}$.  

Below we consider three examples. First, we consider a unit circle where the spectral fractional Laplacian is identical to the integral definition \cite{NAbatangelo_EValdinoci_2017a}. The second example is a unit sphere in $\mathbb{R}^3$ where the spectral fractional Laplacian can be closely represented in an integral form \cite{alonso2018integral}. In these two cases, we can use the spectral definition to verify the accuracy of the estimated eigenfunctions. In the third example, we consider the closed and bounded unit interval for which the spectral fractional Laplacian is different from the integral definition. In this case, we will numerically verify that our fractional diffusion maps implemented with the non-local kernels yields estimates that are close to the regional fractional Laplacian. {Finally, we include a numerical example involving a supervised learning task and show how the nonlocal kernels are beneficial in recovering non-smooth regression functions.}

\subsection{Example 1: Circle}

In this section we first verify the above theory numerically by applying the two different approaches from the dichotomy to reconstruction various fractional Laplacians on the unit circle.  We then compare the exponential and polynomial kernels over a range of different methods of sampling data points from the unit circle.  For all $\beta$, the eigenfunctions of $(-\Delta)^{\beta/2}$ are the Fourier functions, $\phi_1 = 1$, and $\phi_{2j}(\theta) = \sin j\theta, \phi_{2j+1}(\theta) = \cos j\theta$ defined in the intrinsic coordinate $\theta \in [0,2\pi)$ on the unit circle and the associated eigenvalues are $\lambda_1 =0$ and $\lambda_{2j}=\lambda_{2j+1} =j^{\beta}$.  Notice that the repeated eigenvalue is due to the rotational symmetry, and the fact that each eigenspace is two dimensional means that a numerical eigensolver will find two linear combinations of $\sin j\theta$ and $\cos j\theta$.  Thus, in order to compute the error in the estimated eigenfunctions, we first estimate the best linear transformation (a 2-by-2 matrix) which maps each pair of the estimated eigenfunctions to the true eigenfunctions. 

\begin{figure}[h!]
\centering\includegraphics[width=0.48\textwidth]{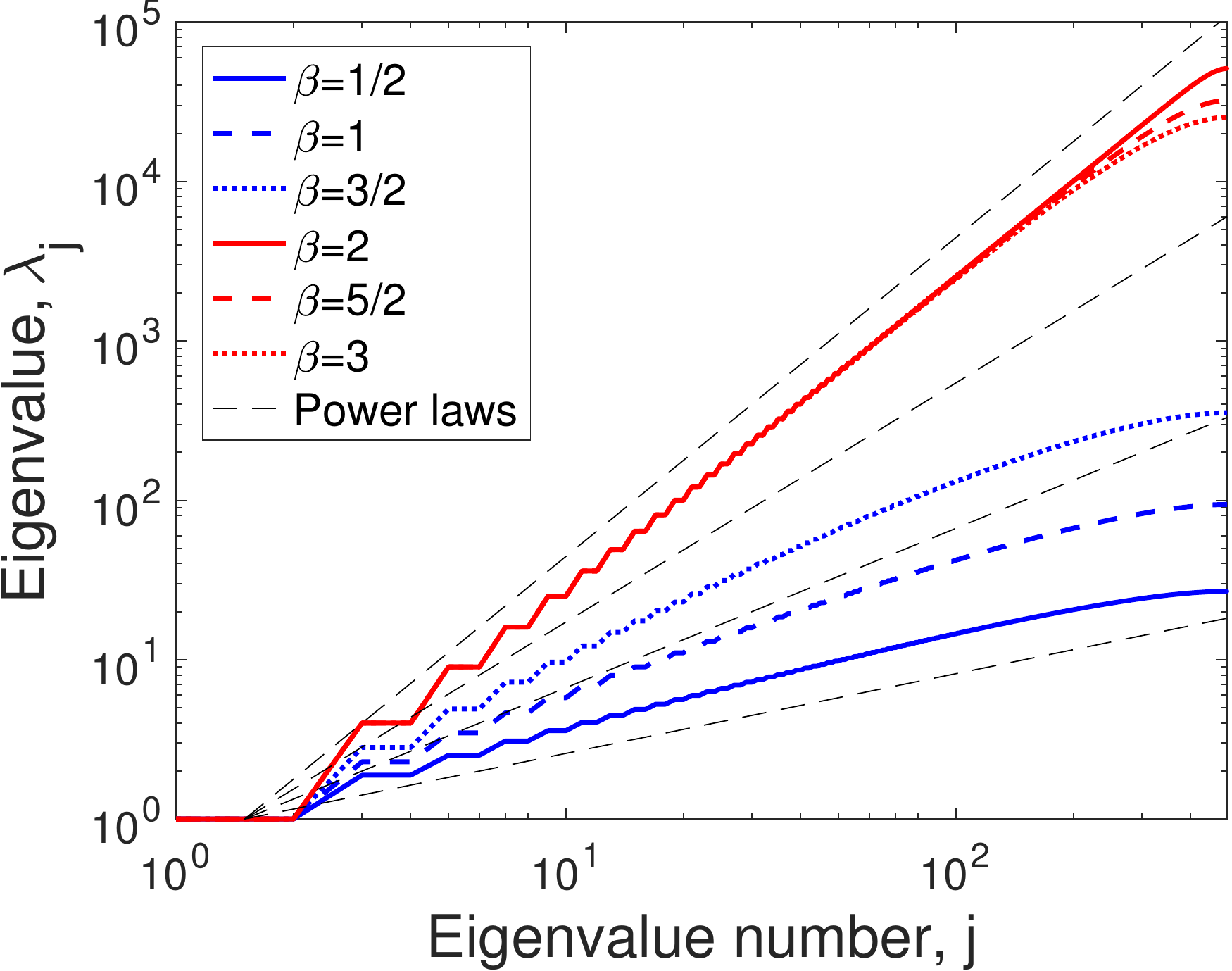}
\centering\includegraphics[width=0.48\textwidth]{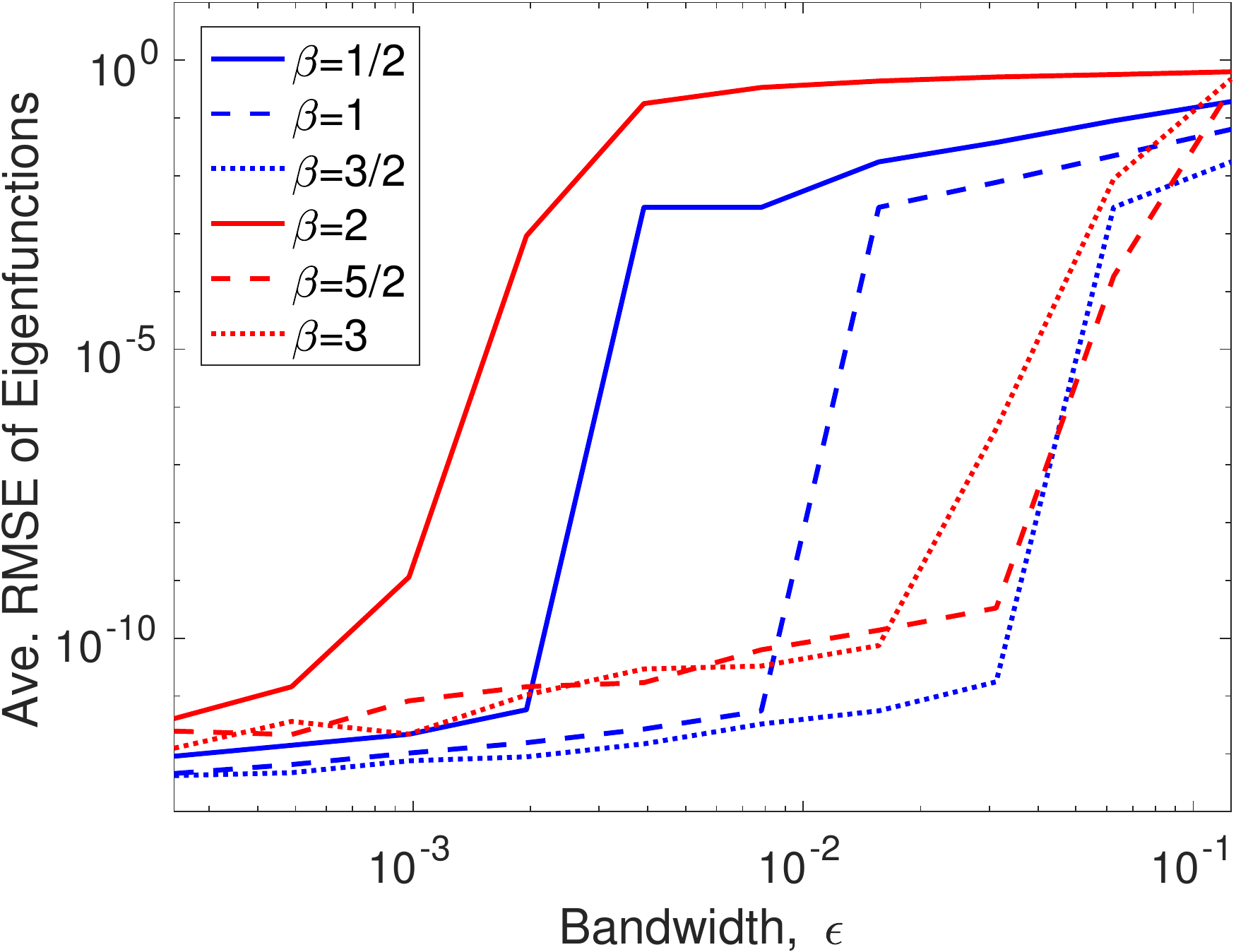}
\caption{Laplacian eigenfunction reconstructions for a {\bf uniform grid} of points sampled from the unit circle in the plane using various values of $\beta=2s$ using exponential (red) and polynomial (blue) kernels.  Left: For $\epsilon = 2^{-12}$ eigenspectra $\lambda_i$ show good agreement with the appropriate power laws (black dashed curves are, $j^2,j^{3/2},j^1,j^{1/2}$).  Right: Average RMSE of eigenfunctions compared to the true Fourier functions as a function of the bandwidth $\epsilon$.}
\label{figure1}
\end{figure}  

In order to verify the above theory, we first consider a grid of points uniformly spaced on the unit circle.  For $N=500$ we generated an equally spaced grid of $\theta_i = 2\pi i/N$ for $i=1,...,N$ and then mapped these points onto the unit circle with the standard embedding $x_i = (\cos(\theta_i),\sin(\theta_i))^\top$.    As mentioned above, the uniform grid yields a very low quadrature error and so the optimal choice of bandwidth $\epsilon$ was very small, on the order of $10^{-4}$ as shown in Fig.~\ref{figure1}(right).  A surprising result for this data set was that the standard Gaussian kernel, $\beta = 2s= 2$, was the least robust with respect to $\epsilon$.  Since the manifold is one-dimensional, Weyl's law for the Laplace-Beltrami operator gives eigenvalues $\lambda_j \propto j^2$.  Notice that since $\beta/2=1,5/4,3/2$ are all local kernels, the generator of each of these kernels is the Laplace-Beltrami operator (since we divide by $\epsilon^{\beta/2}$ as described in Section \ref{local}).  Thus, all of the spectra estimated from exponential kernels (red color in Fig.~\ref{figure1}) have the same power law growth, namely $j^2$.  We note that the error in the estimates of the eigenvalue grows as the eigenvalues increase, which explains the deviations from the power law for the largest eigenvalues; as the amount of data is increased and the bandwidth is decreased these eigenvalues would become more accurate and approach the power law.  On the other hand, the {polynomial} kernel (blue color in Fig.~\ref{figure1})) produces spectra power close to their corresponding $\beta/2=1/4,1/2,3/4$.

\begin{figure}[h!]
\centering\includegraphics[width=0.32\textwidth]{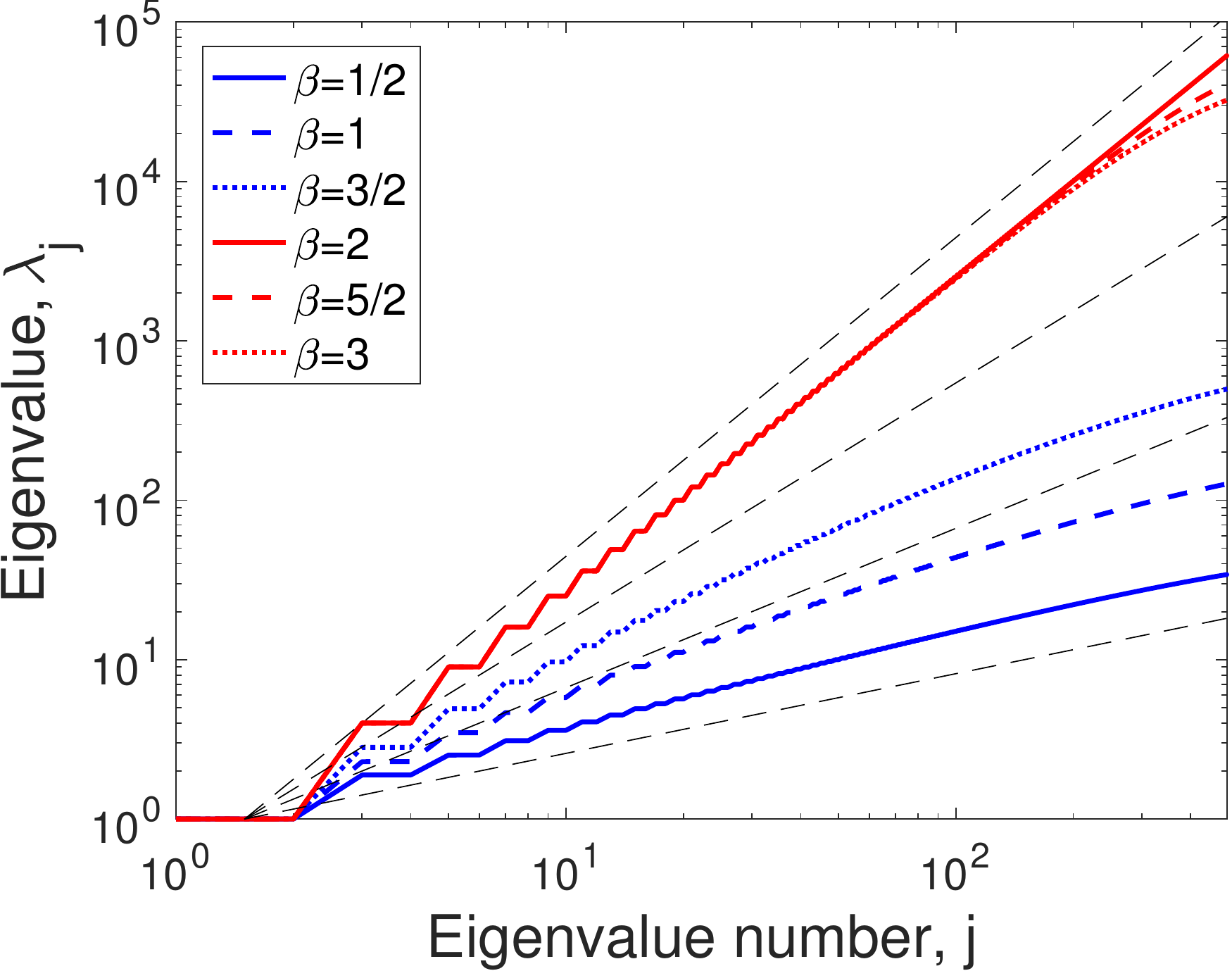}
\centering\includegraphics[width=0.32\textwidth]{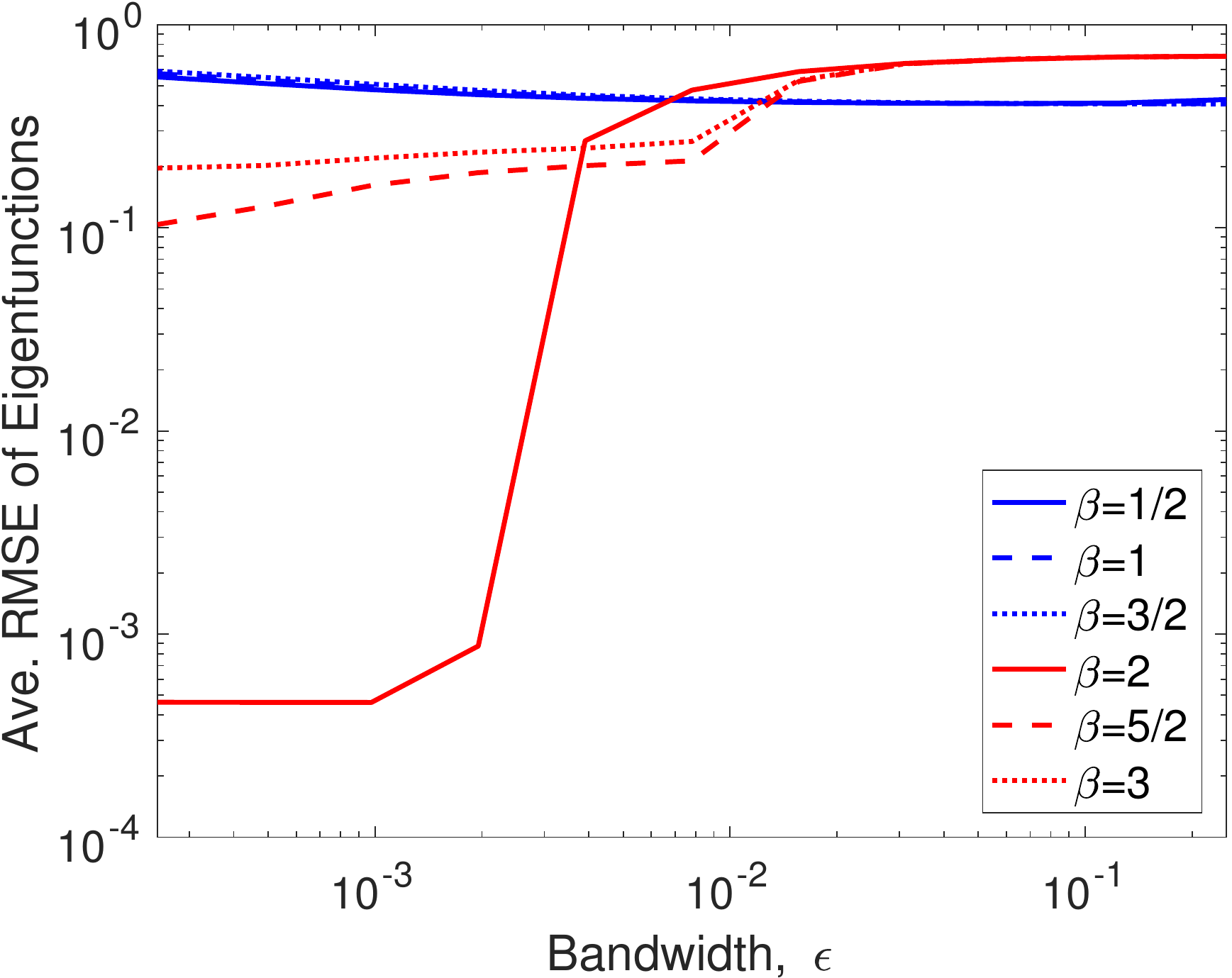}
\centering\includegraphics[width=0.32\textwidth]{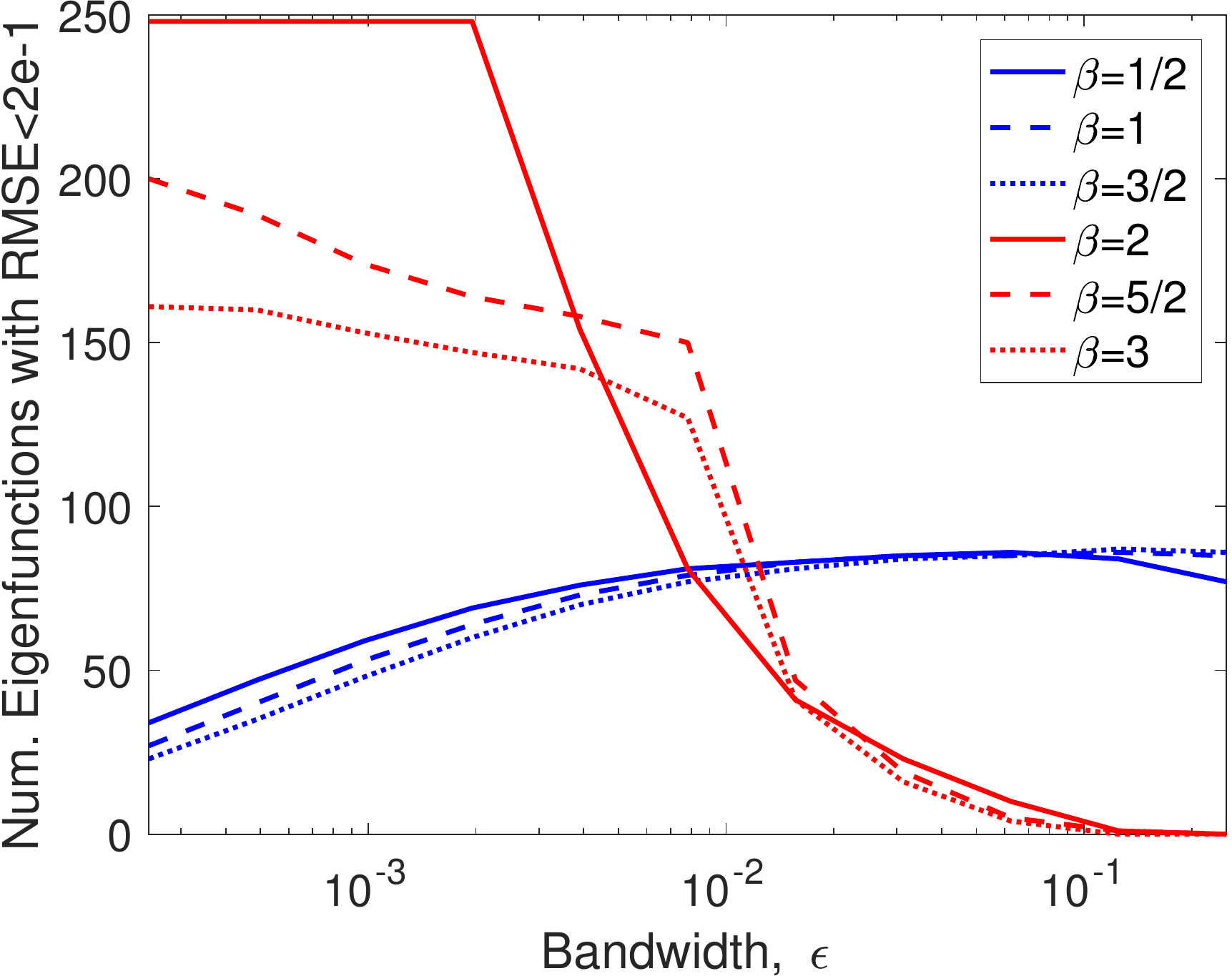}
\centering\includegraphics[width=0.32\textwidth]{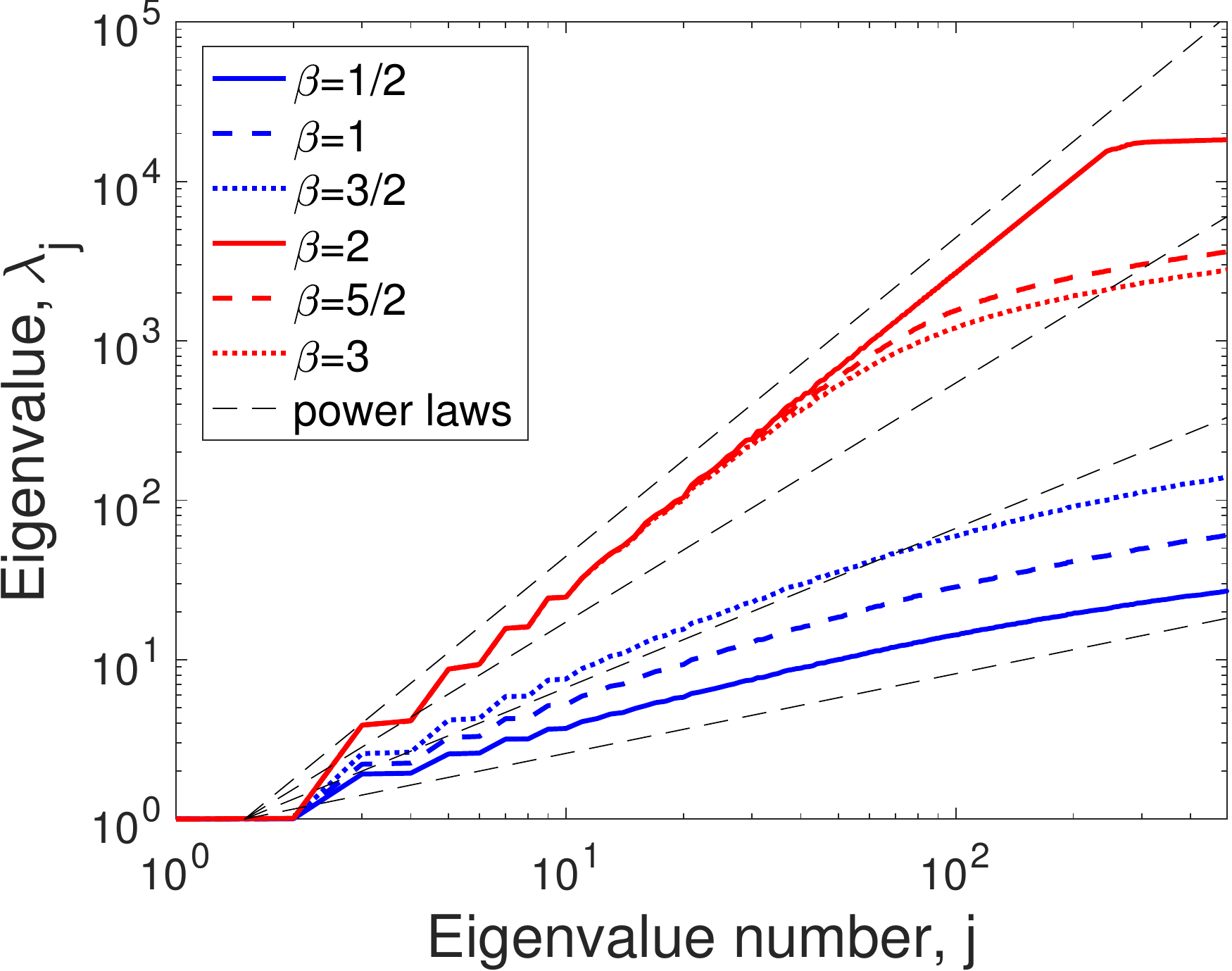}
\centering\includegraphics[width=0.32\textwidth]{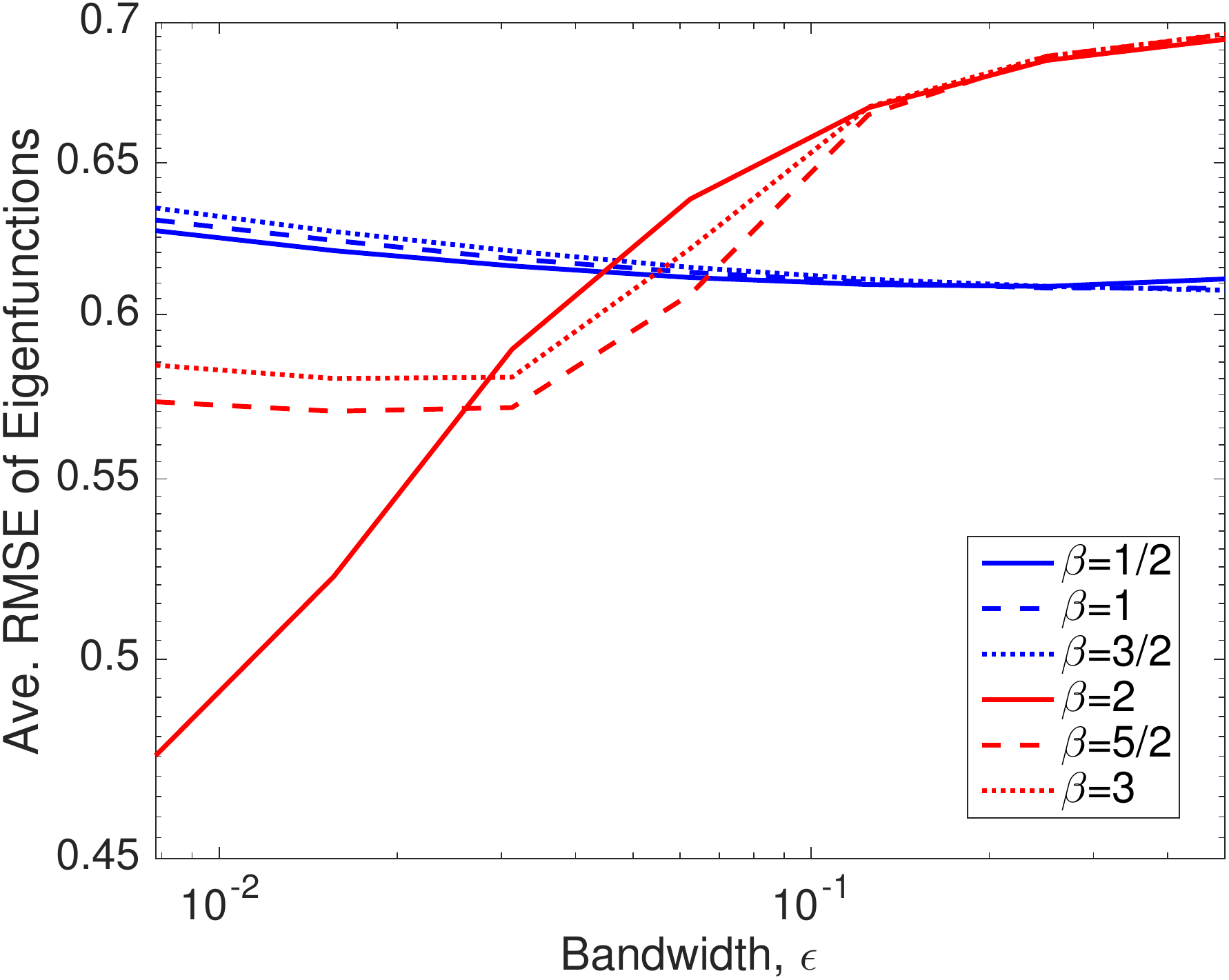}
\centering\includegraphics[width=0.32\textwidth]{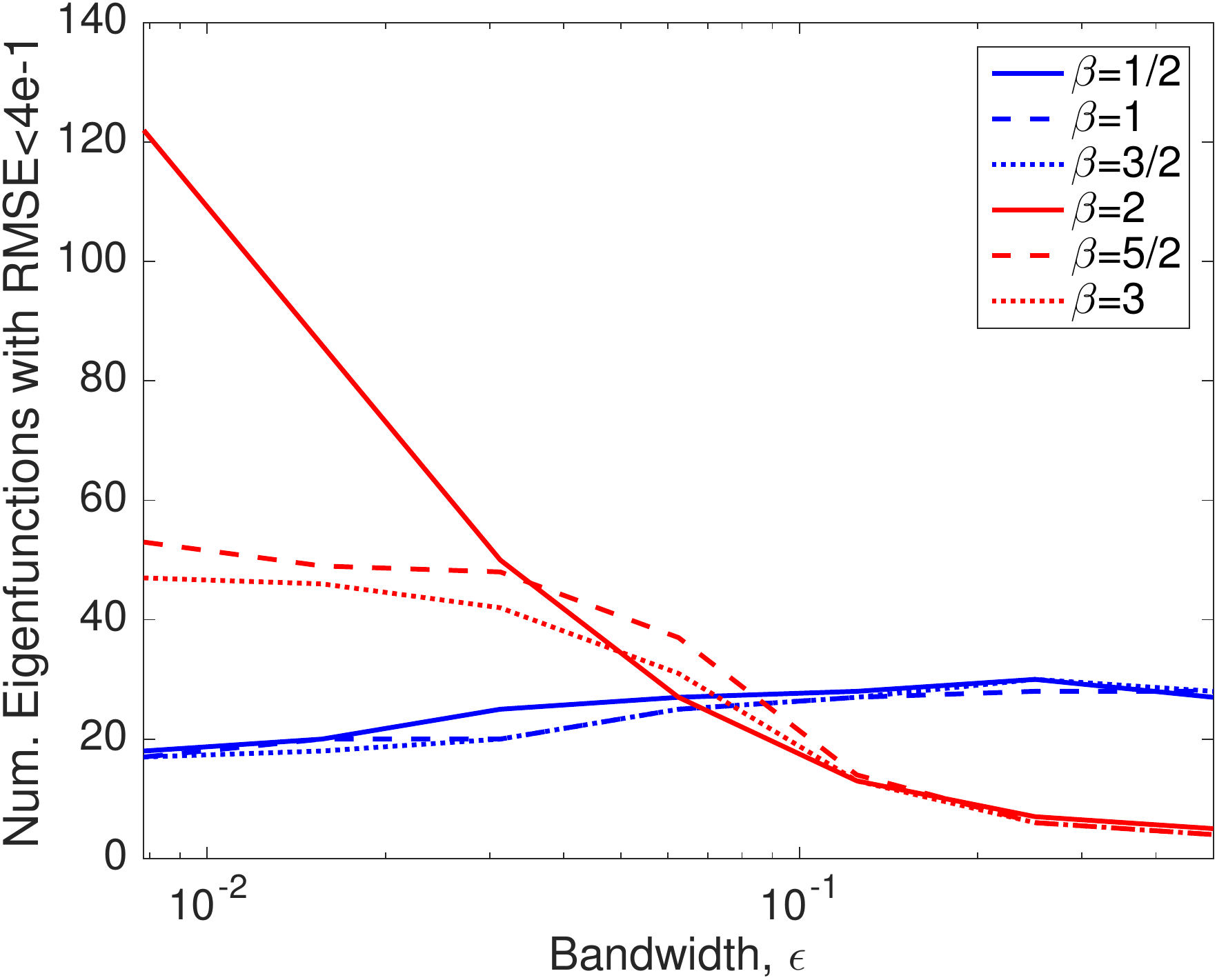}
\caption{Laplacian eigenfunction reconstructions for a {\bf non-uniform grid} of points (top row) and for a {\bf uniformly random sample} (bottom row) of points on the unit circle in the plane using various values of $s$ using exponential (red) and polynomial (blue) kernels.  Left: Spectra (as in Fig.~\ref{figure1}).  Middle: Average RMSE of eigenfunctions as a function of the bandwidth $\epsilon$.  Right: Number of eigenfunctions with RMSE less than $0.2$ (top) and $0.4$ (bottom) as a function of the bandwidth $\epsilon$.  Note that the polynomial kernel has a more consistent performance over a larger range of bandwidths, but the exponential kernel has the best performance with an optimally tuned bandwidth.}
\label{figure2}
\end{figure}  

Next we consider a non-uniform grid of points by applying the nonlinear mapping $\tilde \theta_i = \theta_i - \sin(\theta_i)/2$ and then defining the data set $x_i = (\cos\tilde\theta_i,\sin\tilde\theta_i)^\top$.  We say this is a `grid' since the points are not pseudo-random, but it is non-uniform in the sense that the grid spacing is non-uniform.  The results for this data set are shown in the top row of Fig.~\ref{figure2}.  For this dataset the errors in many of the eigenfunctions increase significantly, however the eigenfunctions associated to small eigenvalues are still reasonably well estimated (see Fig.~\ref{figure3}).  To give a better measure of performance, we also plot the number of eigenfunctions that have RMSE less than $0.2$ in Fig.~\ref{figure2}(top, right).  As in the previous example, the results of the polynomial kernels was more robust to the choice of bandwidth than the exponential kernels.  However, in this example the performance of the exponential kernel (particularly with $\beta=2$) was significantly better than the polynomial kernels when the bandwidth is optimally tuned for each kernel. This is not surprising since the fractional Laplacian in this setting (that takes $L^2$ functions defined on flat and periodic domain) is exactly the spectral fractional Laplacian as defined in Definition~\ref{specdef} where the Gaussian kernel is the heat kernel of the semigroup of Laplacian.

Finally, in the bottom row of Fig.~\ref{figure2} we show the results for a pseudo-random data set $\theta_i = 2\pi r_i$ where $r_i$ is a  psuedo-random value sampled uniformly from $[0,1)$.  Again, the polynomial kernels are more robust to choice of bandwidth, but for the optimally tuned bandwidth the $\beta=2$ kernel has a pronounced advantage.  In Fig.~\ref{figure3} we show that for the eigenfunctions associated to the smallest eigenvalues the polynomial kernel actually has the smallest error (the error is often less than half that of the exponential kernel).  However, for non-uniform or pseudo-random data the exponential kernel maintains a lower error for higher frequency eigenfunctions.

\begin{figure}[h!]
\centering\includegraphics[width=0.48\textwidth]{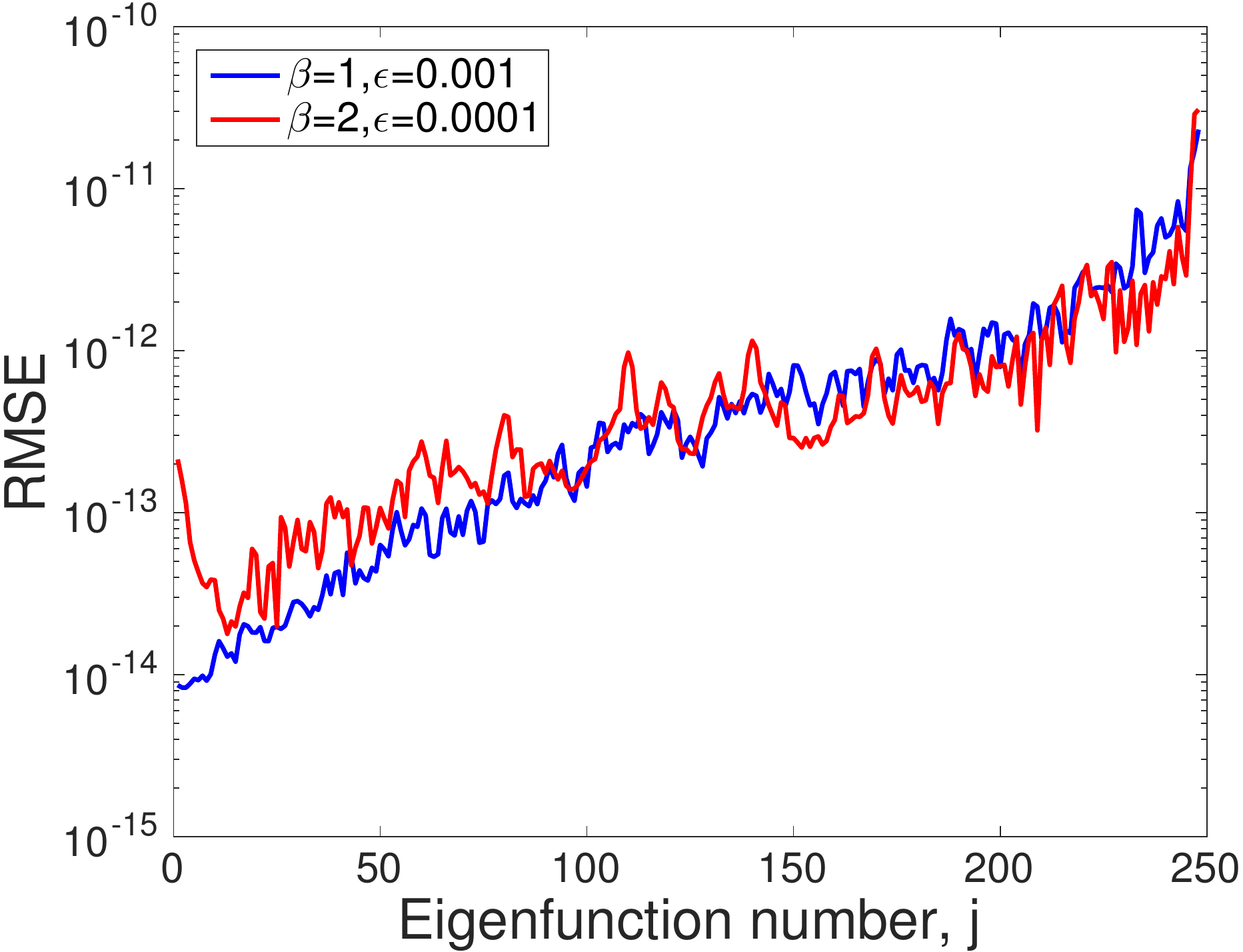}
\centering\includegraphics[width=0.48\textwidth]{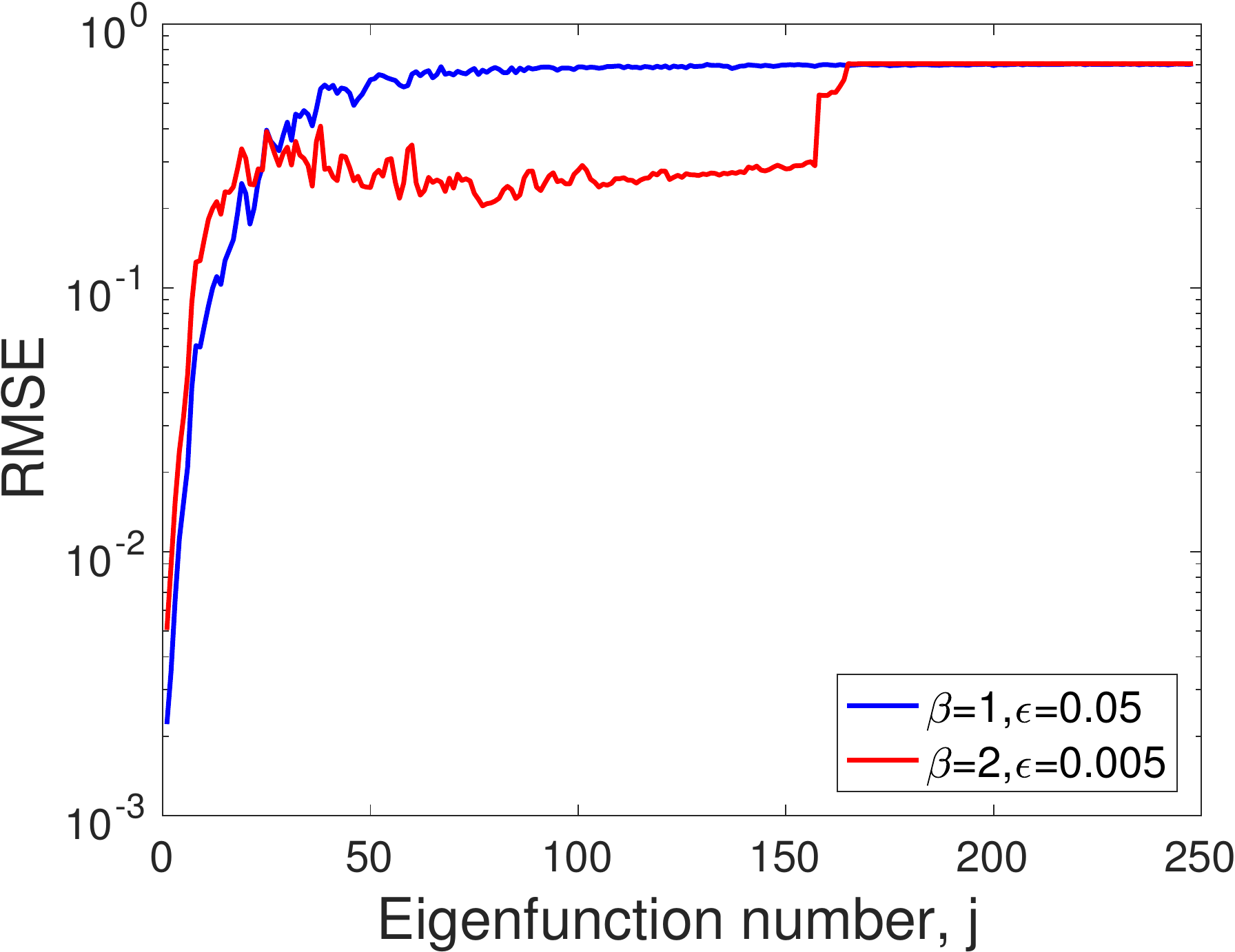}
\caption{Comparison of exponential (red) and polynomial (blue) kernels in terms of eigenfunction RMSE with optimal bandwidth for each kernel.  Left: For uniform grid sampled form the unit circle performance of the two kernels is comparable when well-tuned.  Right: For a uniformly random sample of points on the unit circle, the polynomial kernel has the best reconstruction for the low frequency eigenfunctions, and the exponential kernel has the best reconstruction for the middle frequency eigenfunctions.}
\label{figure3}
\end{figure}

\subsection{Example 2: Sphere}

The unit circle is a simple yet instructive test example to analyze since it is one-dimensional and has no curvature, which means that reasonably accurate results can be obtained with very small data sets.  In this example we consider the sphere which is two-dimensional and has non-zero curvature and thus will require much larger data sets. This is particularly problematic for the polynomial kernels because of our reliance on Dijkstra's algorithm to compute graph distances which approximate geodesic distances.  Optimized methods for computing the graph distances between all pairs of points still have computational complexity which grows cubically with the number of data points, which quickly becomes computationally infeasible.  However, a nice feature of the unit sphere in $\mathbb{R}^3$ is that the geodesic distance between points $x,y$ on the manifold can easily be computed as $\cos^{-1}(\iota(x)\cdot \iota(y))$ (where $\iota(x),\iota(y) \in \mathbb{R}^3$ are the embedded data points).  While this is obviously not a method which can be used for real data sets, it is helpful in this context to verify the theoretical results above.  Moreover, we expect that thorough error analysis may reveal that, for a given $t$, only pairs of points with sufficiently small graph distances may need to be computed, which could allow the computational complexity to be reduced below cubic growth in the number of data points.  In the examples below we will use Dijkstra's algorithm to estimate the geodesic distances except when using polynomial kernels with $N=10242$ data points in the bottom row of Fig.~\ref{figure5}.

Another challenge of the sphere is generating a uniform or approximately uniform grid of data points.  To solve this problem we used a Matlab package called GridSphere written by the authors of \cite{gridsphere} which describes their method.  We should note that this method produces an approximately uniform grid, so for example, with $N=2562$ data points, each point has 6 nearest neighbors whose distances differ on the order of $10^{-4}$ (see Fig.~\ref{figure4} middle and right for the $N=2562$ grid).  

\begin{figure}[h!]
\centering
\includegraphics[width=.3\textwidth]{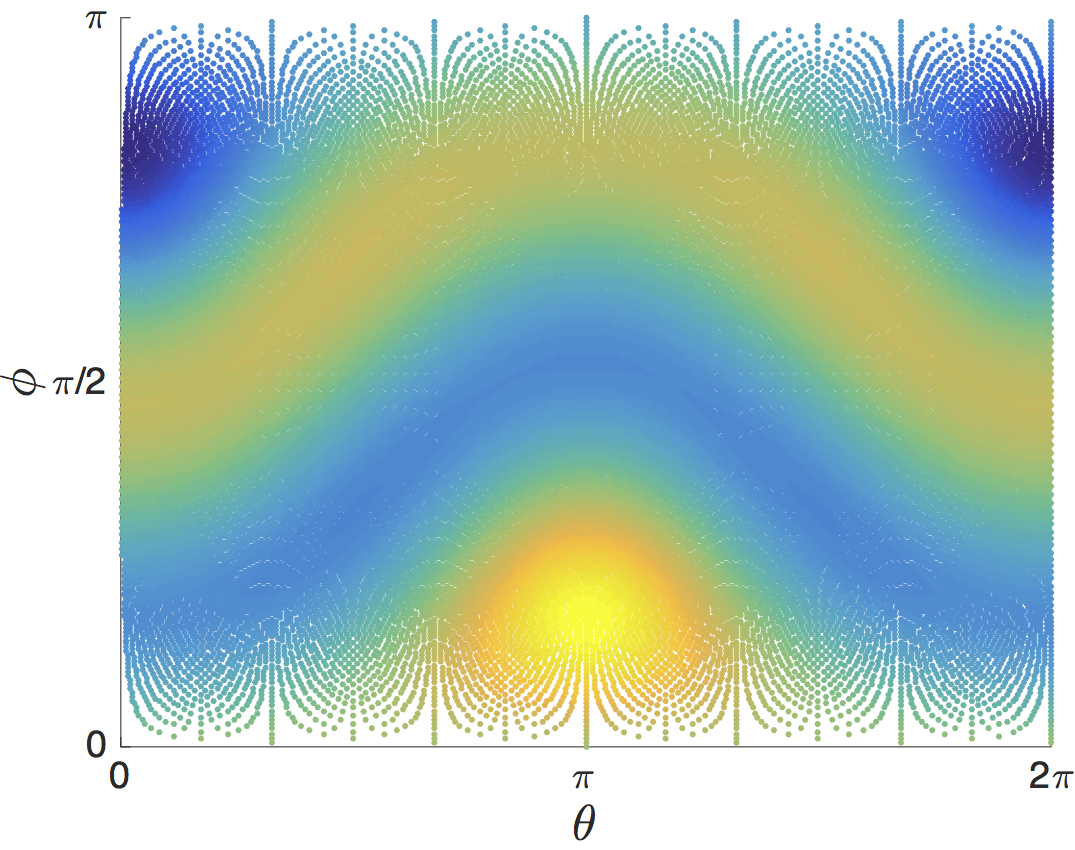}
\includegraphics[width=.3\textwidth]{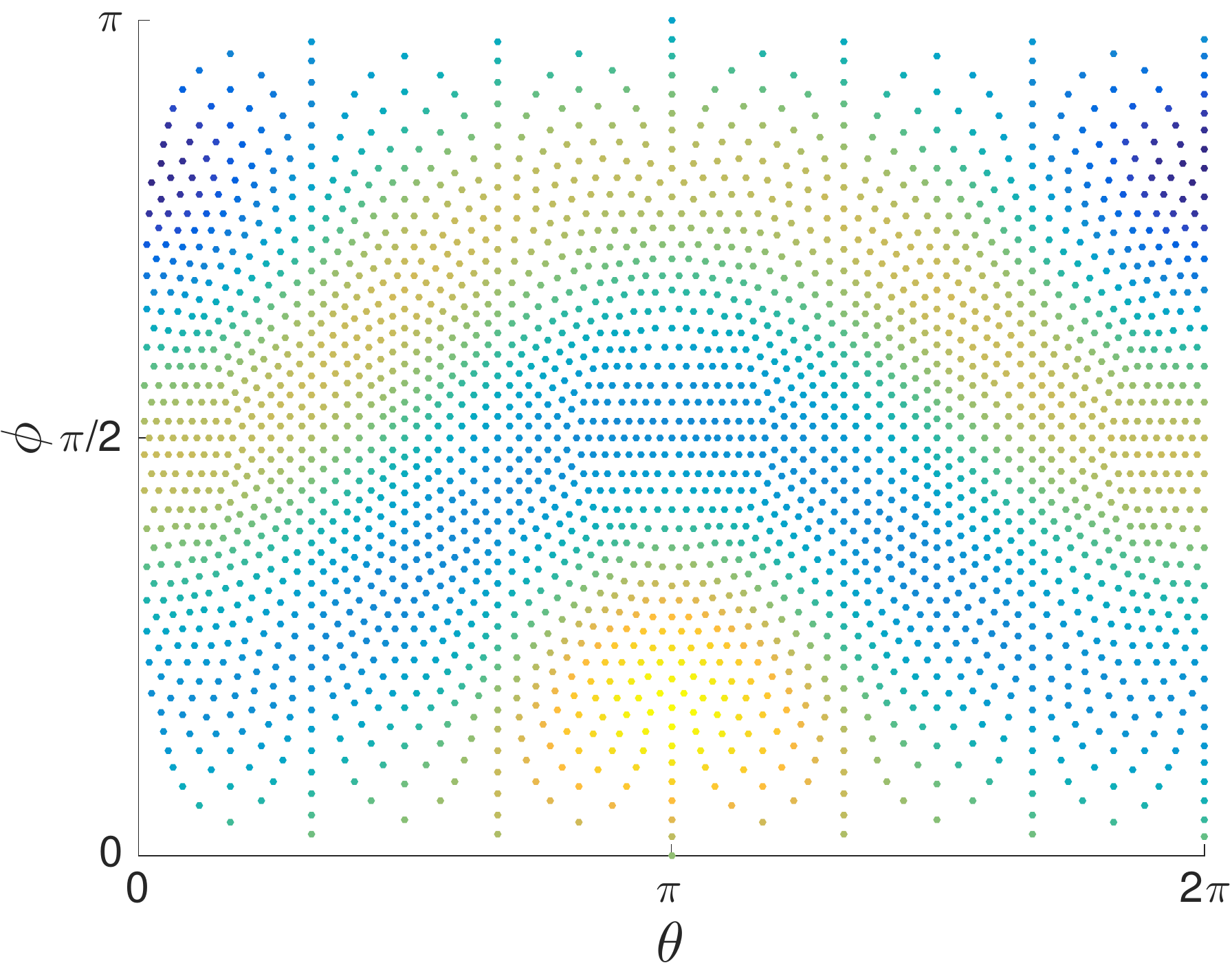}
\includegraphics[width=.34\textwidth]{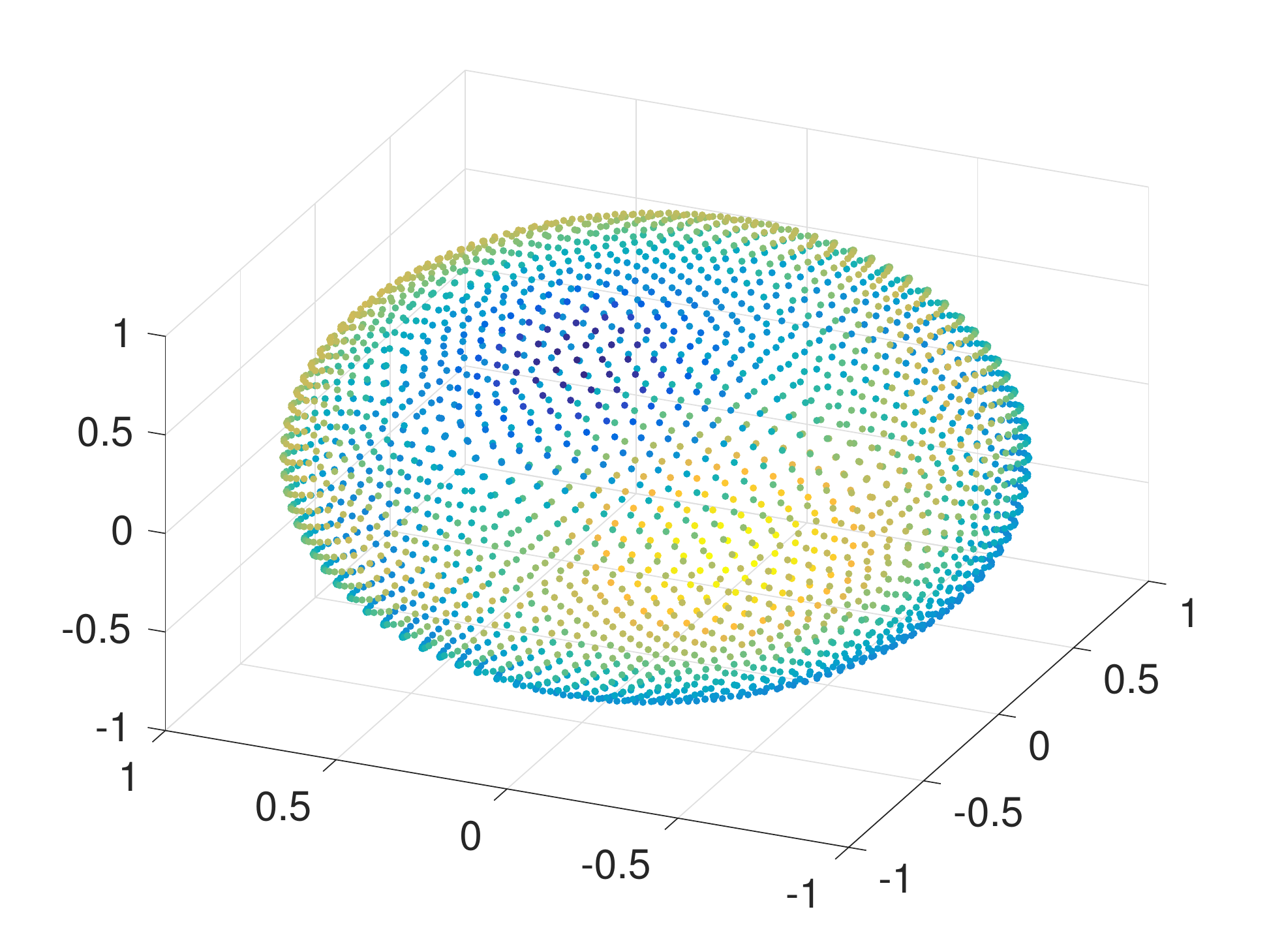}
\caption{The 10-th Laplacian eigenfunction, $\varphi_{10}(\theta,\phi)$, estimated for a nearly uniform grid of points sampled from the unit sphere in $\mathbb{R}^3$ using $s=1$ (exponential kernel) with $N=40962$ points (left) which is then subsampled to a nearly uniform grid of $N=2562$ points (middle, right).  Sample points were generated using the Matlab GridSphere package \cite{gridsphere}.  Eigenfunction is shown as a function of the polar coordinates $(\theta,\phi) \in [0,2\pi) \times [0,\pi)$ for the sphere (left,middle) and on the embedded sphere (right).}
\label{figure4}
\end{figure}  

Finally, in order to define a ground truth for comparison, we first generated an approximately uniform grid of $N=40962$ data points, and used the exponential kernel with $\beta=2$ to estimate the eigenfunctions.  Since the $N=2562$ and $N=10242$ grids are subsets of the $N=40962$ grid, we were able to decimate these eigenfunctions to find their values on the subgrids.  See Fig.~\ref{figure4} for the estimate of $\varphi_{10}$ on the $N=40962$ grid (left) and the decimated estimate on the $N=2562$ grid (middle).  We used these estimates as our `truth' for the eigenfunctions. 

Next we apply the diffusion maps algorithm for the exponential (red) and polynomial (blue) kernels with various values of $\beta$ and various values of $\epsilon$ (see Fig.~\ref{figure5} first two rows $N=2562$, third row $N=10242$).  In Fig.~\ref{figure5} (top,left) we show the eigenvalues for $N=2562$ have good agreement with Weyl's law; since the sphere is 2-dimensional the Laplacian eigenvalues grow like $\lambda_j \propto j^{d/2} = j^1$.  This power law was observed for $\beta \geq 2$ since all the exponential kernels recover the Laplace-Beltrami operator, and for $\beta<2$ we observe reasonable agreement with the theoretical growth of $\lambda_j \propto j^\beta$.

\begin{figure}[h!]
\centering \hspace{-5pt}\includegraphics[width=0.42\textwidth]{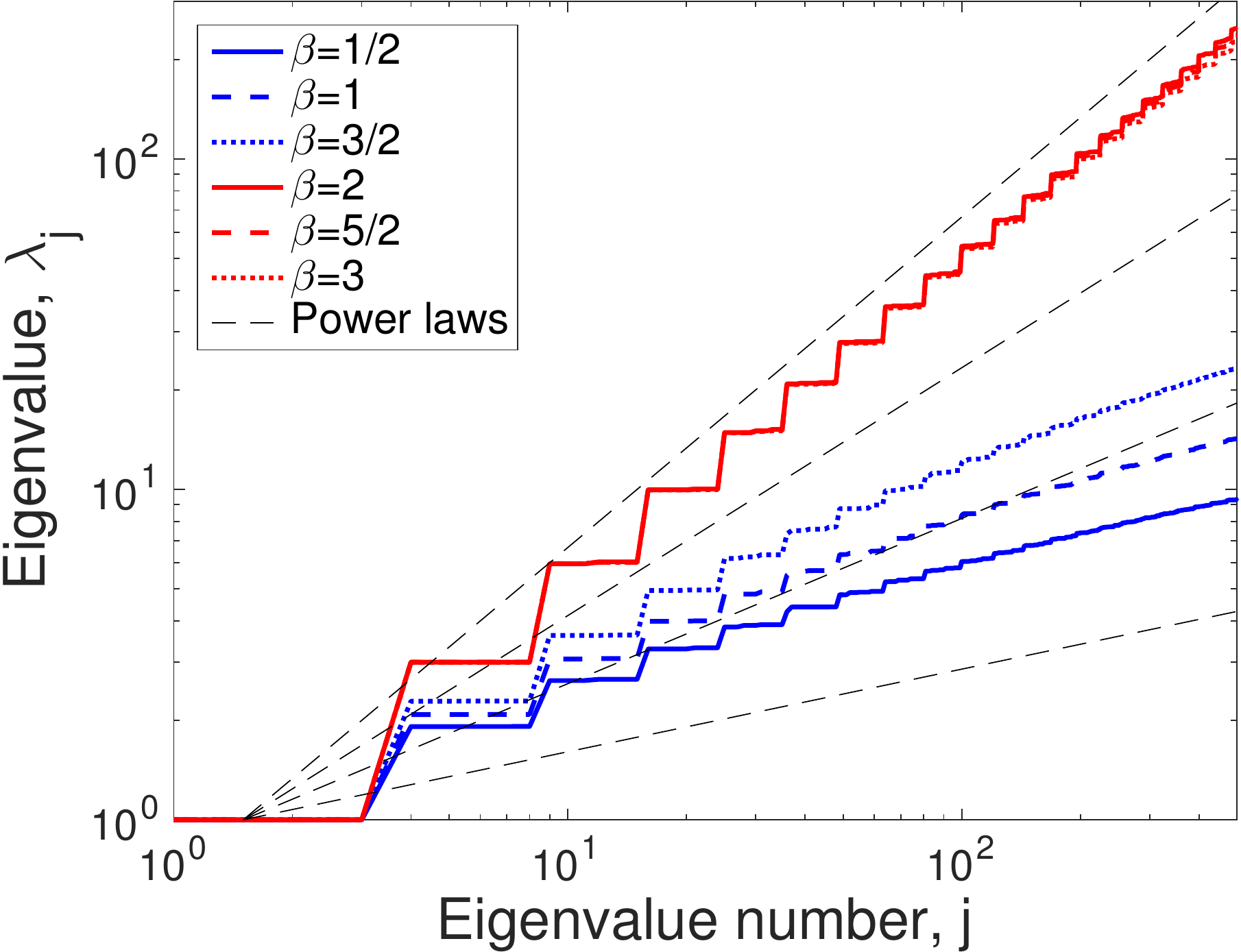}\hspace{5pt}
\centering\includegraphics[width=.4\textwidth]{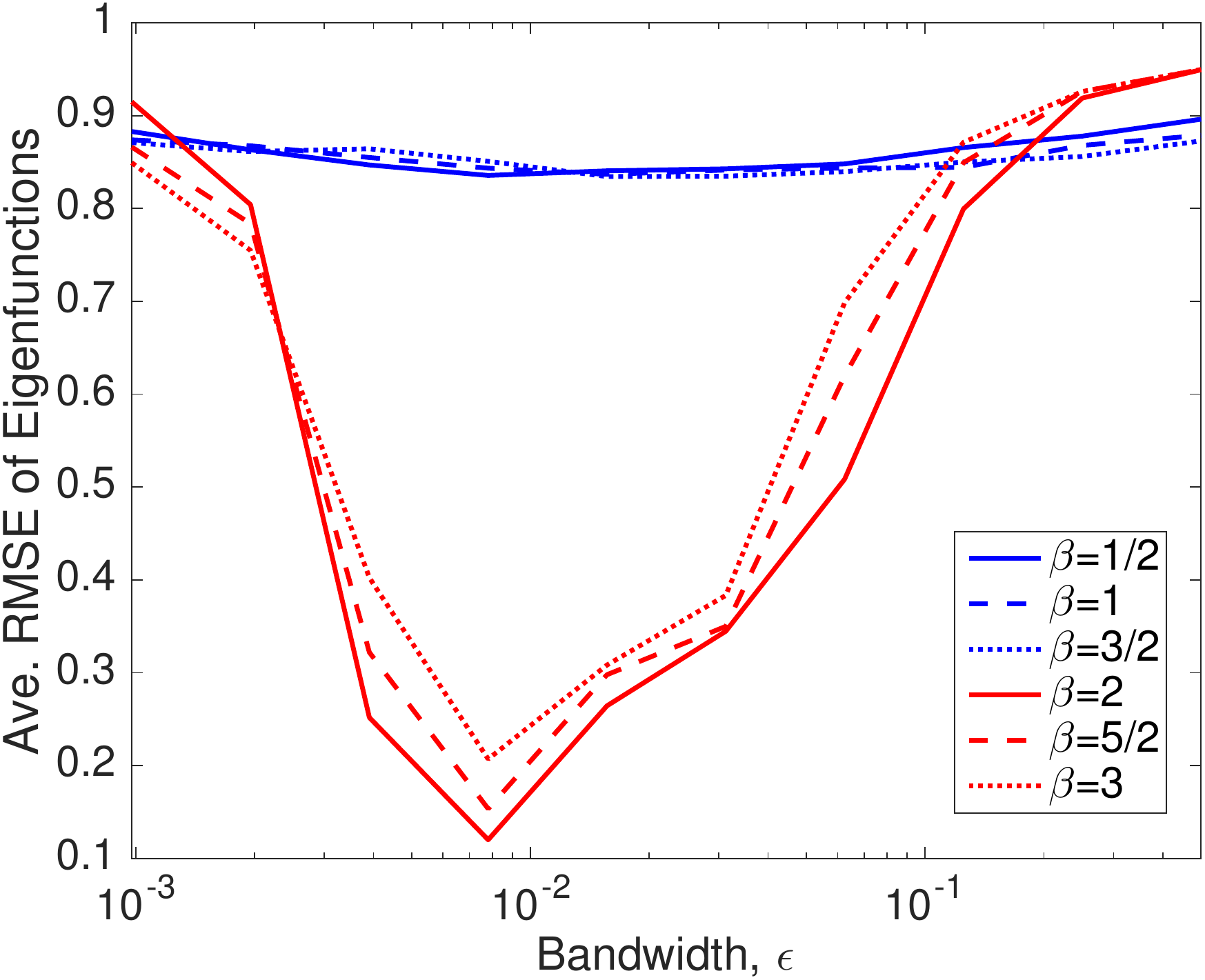}
\centering\includegraphics[width=0.4\textwidth]{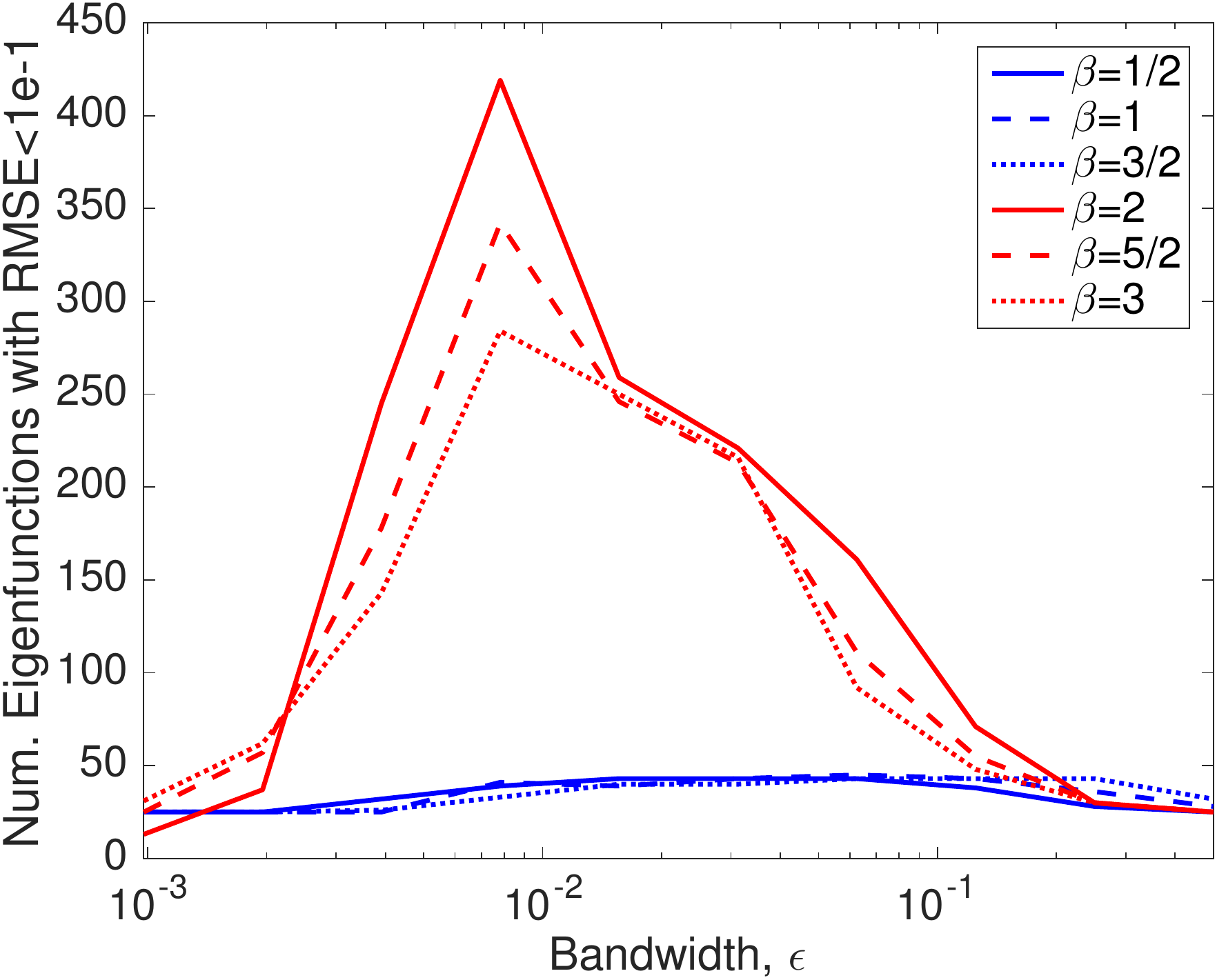}
\centering\includegraphics[width=0.42\textwidth]{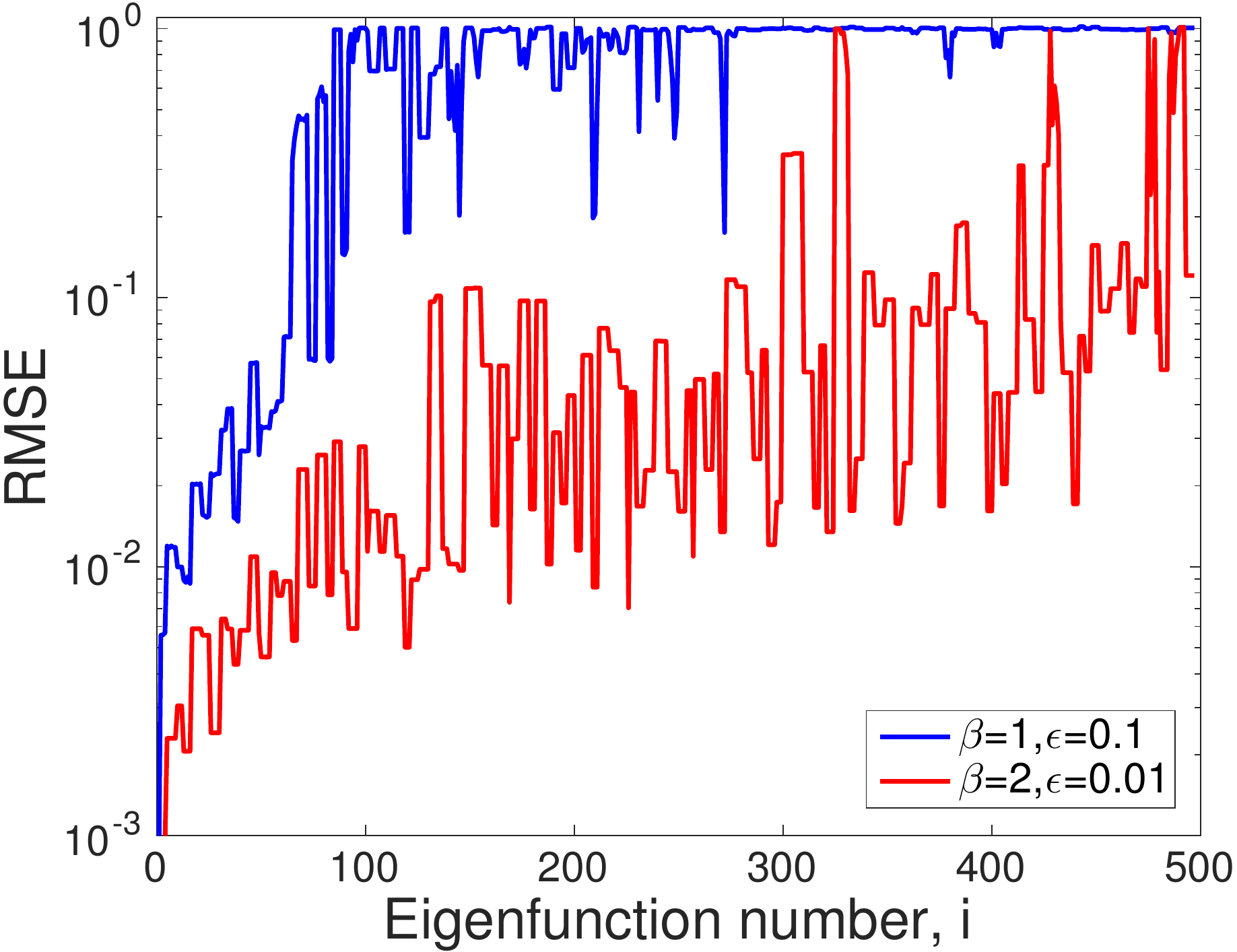}
\centering\includegraphics[width=0.4\textwidth]{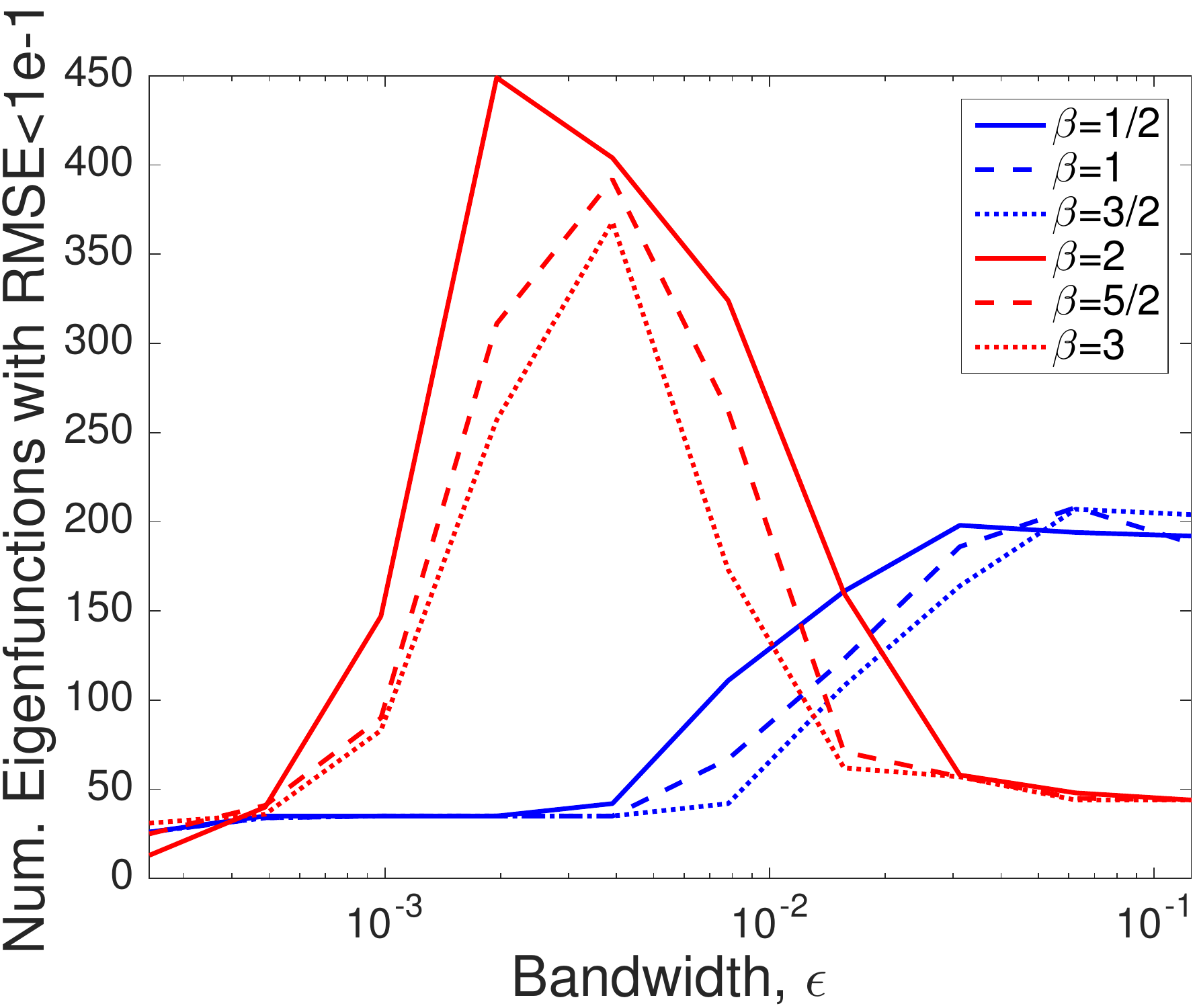}
\centering\includegraphics[width=0.42\textwidth]{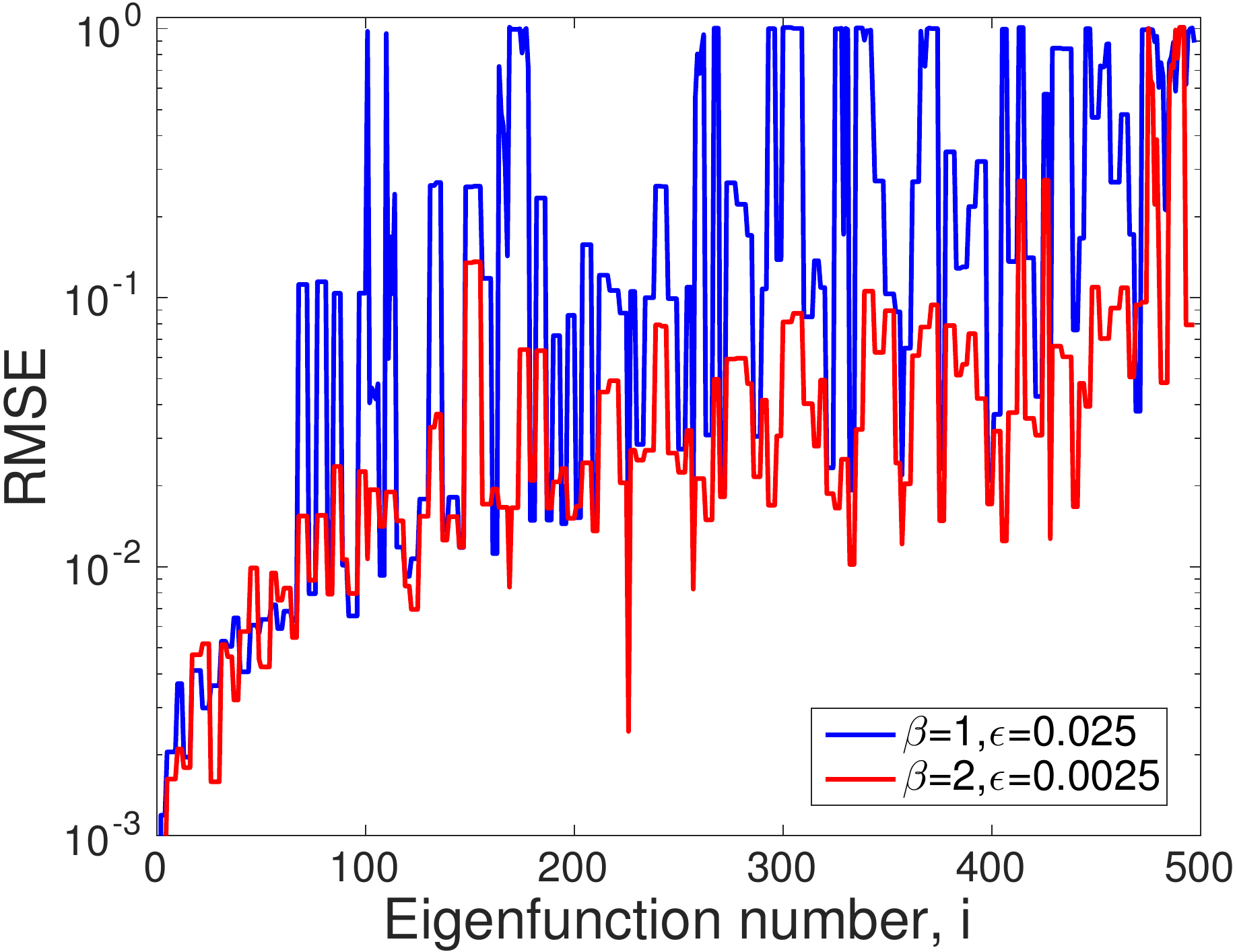}
\caption{Laplacian eigenvalue/eigenfunction estimates (corresponding to the smallest $500$ eigenvalues) for a nearly uniform grid of $N=2562$ points on the unit sphere in $\mathbb{R}^3$ using various values of $s$ using exponential (red) and polynomial (blue) kernels.  Top, left: Spectra compared to the power laws $j,j^{3/4},j^{1/2},j^{1/4}$.  Top, right: Average RMSE of eigenfunctions as a function of the bandwidth $\epsilon$.  Second row, left: Number of eigenfunctions with RMSE less than $0.1$ as a function of the bandwidth $\epsilon$. Second row, right: Comparison of the RMSE as a function of eigenvector number for $s=1/2$ (blue) and $s=1$ (red) with optimally tuned bandwidth.  Bottom row, same as second row but with $N=10242$ data points.}
\label{figure5}
\end{figure}  

An additional challenge was the repeated eigenvalues due to the symmetry of the sphere leading to multidimensional eigenspaces.  We used the eigenvalues from the $N=40962$ grid to detect repeated eigenvalues (with a threshold of $10^{-6}$) and applied the same method as for the circle.  Namely, for each group of eigenfunctions having the same eigenvalue we estimated the linear transformation (a $\ell \times \ell$ matrix where $\ell$ is the dimension of the eigenspace) that optimally mapped the estimated eigenfunctions onto the true eigenfunctions before computing the RMSE.

As with the circle, the best results were obtained with $\beta =2$, however the $\beta = 5/2$ and $\beta=3$ results were much closer to $\beta=2$ on the sphere than on the circle.  We expect that this is due to the curvature on the sphere; namely, since the unit circle has zero curvature the kernel $e^{-d_g(x,y)^2/t}$ is exactly the heat kernel (up to constants and rescaling time), whereas on a curved manifold the heat kernel will be equal to this exponential multiplied by a polynomial with coefficients that depend on the curvature and its derivatives.  In other words, on a flat manifold (such as the circle) the $\beta=2$ exponential kernel has fewer error terms than other values of $\beta$, but this does not hold for general manifolds (such as the sphere).  Finally we note that as $N$ increases the difference advantage of the exponential kernel in terms of optimal RMSE appears to decrease. However, the exponential kernels maintain a significant computational advantage since they do not require one to estimate the geodesic distance.

\subsection{Flat manifold with boundary}

So far our focus has been compact manifolds $\mathcal{M}$ without boundaries. Recall that when the manifold  
$\mathcal{M}$ is a flat torus or $\mathbb{R}^n$, the spectral fractional Laplacian given in Definition~\ref{specdef} is equivalent (up to constants) 
to the so-called integral fractional Laplacian (see \cite[Pg.~15]{NAbatangelo_EValdinoci_2017a}) defined for $u \in C_c^\infty(\mathcal{M})$ (i.e., $u$ is compactly supported on 
$\mathcal{M} \subset \mathbb{R}^n$) as
    \begin{equation}\label{eq:dwb1}
        (-\Delta)^s_I u(x) 
            = c_{n,s} \mbox{ P.V.} \int_{\mathbb{R}^n} \frac{u(x)-u(y)}{|x-y|^{n+2s}}\; dy 
    \end{equation}
for every $x \in \mathcal{M}$. In the above equation, ``P.V." stands for the 
``Cauchy Principle Value", i.e., we understand \eqref{eq:dwb1} as 
    \begin{equation*}
        (-\Delta)^s_I u(x) 
           = c_{n,s} \; \lim_{\varepsilon \downarrow 0} 
               \int_{\mathbb{R}^n\setminus B(x,\varepsilon)} \frac{u(x)-u(y)}{|x-y|^{n+2s}}\; dy . 
    \end{equation*}
Here $c_{n,s} = \frac{s 2^{2s} \Gamma\left( \frac{n+2s}{2}\right)}{\pi^{\frac{n}{2}} \Gamma(1-s)}$ is a normalization constant depending only on $n$ and $s$, see \cite{MWarma_2015a,antil2018a,antil2018b}. The purpose of this section is to consider flat manifolds $\mathcal{M}$ with boundaries. Our main 
goal is to understand whether the approximate generator in the nonlocal case is the spectral or 
integral fractional Laplacian. As it was thoroughly discussed in \cite{RServadei_EValdinoci_2014a} these
two operators are completely different in general. Before we embark on our journey we also recall
that our data lies on the manifold and not outside the manifold. Indeed using \eqref{eq:dwb1}
we have that 
    \begin{equation}\label{eq:dwb2}
        (-\Delta)^s_I u(x) 
            = c_{n,s} \mbox{ P.V.} \int_{\mathcal{M}} \frac{u(x)-u(y)}{|x-y|^{n+2s}}\; dy 
             + c_{n,s} \; u(x)\int_{\mathbb{R}^n \setminus \mathcal{M}}\frac{1}{|x-y|^{n+2s}}\; dy, 
    \end{equation}
where in the last integral we have used the fact that $u$ is supported only on $\mathcal{M}$. 
In other words, we can only expect to recover the integral over $\mathcal{M}$ in \eqref{eq:dwb2}
unless we consider points in $\mathbb{R}^n\setminus\mathcal{M}$. We shall illustrate with the help 
of a numerical example that the proposed fractional diffusion maps algorithm estimates the integral over $\mathcal{M}$ in \eqref{eq:dwb2}, which is also known as the \emph{regional} fractional Laplacian:   
    \begin{equation}\label{eq:dwb3}
        (-\Delta)^s_{R} u(x) 
            = c_{n,s} \mbox{ P.V.} \int_{\mathcal{M}} \frac{u(x)-u(y)}{|x-y|^{n+2s}}\; dy.
    \end{equation}
Notice that $(-\Delta)^s_I$ and $(-\Delta)^s_R$ differ only by a potential term, nevertheless
this term is difficult to manipulate. 

Next we provide a numerical example to support this claim. We consider the following configuration:
$\mathcal{M} = [0,1]$ with boundary at 0 and 1. Moreover, we set $u(x) = x^2$
and $s = \frac12$. Using the definition of P.V. according to \eqref{eq:dwb3} we
arrive at 
    \begin{eqnarray*}
        (-\Delta)^s_{R} u(x) 
            &=& c_{n,s} \; \lim_{\varepsilon \downarrow 0} 
                \int_{(0,1)\setminus (x-\varepsilon,x+\varepsilon)} 
                    \frac{u(x)-u(y)}{|x-y|^{n+2s}}\; dy \\
            &=& c_{n,s} \; \lim_{\varepsilon \downarrow 0} 
                \int_{(0,1)\setminus (x-\varepsilon,x+\varepsilon)} 
                    \frac{x+y}{x-y}\; dy        
    \end{eqnarray*}
where in the last step we have used the definition of $u$ in conjunction with 
the facts that $n=1$ and $s = \frac12$. Whence 
    \BEA\label{reglaplaceexample}
        (-\Delta)^s_{R} u(x) &=& c_{n,s} \; \lim_{\varepsilon \downarrow 0} 
          \left(-2x \log(y-x) - y \right) |_{(0,1)\setminus (x-\varepsilon,x+\varepsilon)} \nonumber\\
          &=& c_{n,s} \left(-1+2x \log\left( \frac{x}{1-x} \right) \right) 
    \EEA
where the last equality follows after basic algebraic manipulations. 

We note that local kernels can only approximate eigenfunctions $\phi_k$ of the Neumann-Laplacian on $\mathcal{M}=[0,1]$, which are $\phi_k(x) = \cos(\pi k x)$ with eigenvalues $\lambda_k = \pi^2 k^2$.  In order to obtain the ground truth for $(-\Delta)_S^{s}u$, we compute the spectral fractional Laplacian using these analytic eigenfunctions instead of using the diffusion maps estimated eigenfunctions. That is, we compute
$(-\Delta)^s_S u = \sum_{k=1}^M \lambda_k^s \hat u_k \phi_k(x)$ with $\hat u_k = \left<u,\phi_k\right>/||\phi_k||^2$, where the inner product involves integral of functions $u(x)=x^2$ and the analytic $\phi_k(x)$ that can be done explicitly. Numerically, we found nearly identical results using various $M=500, 1000, 2000$.

    \begin{figure}[!h]
    \centering
        \includegraphics[width=0.7\textwidth]{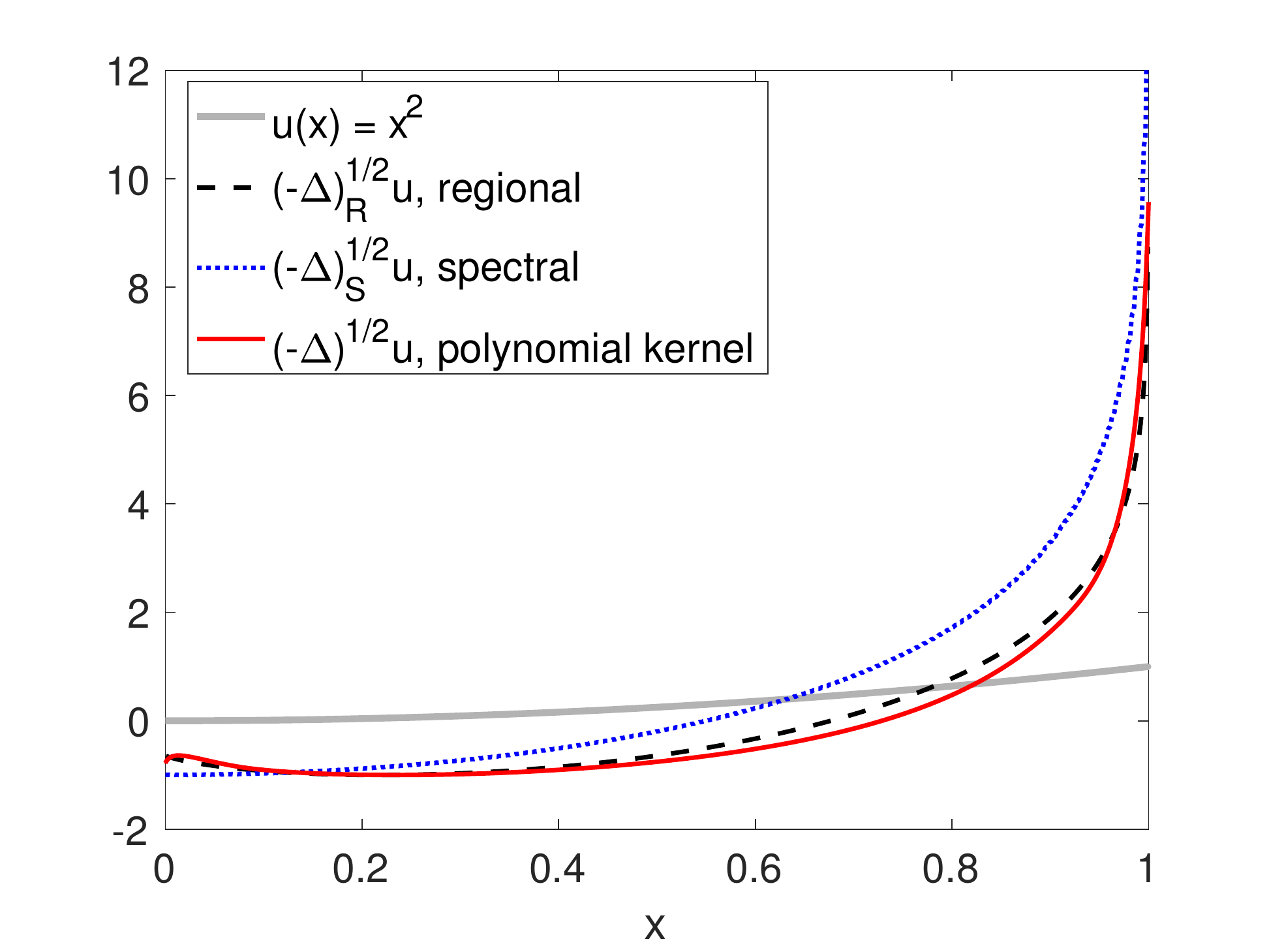}
        \caption{Let $\mathcal{M} = [0,1]$ with boundary at 0 and 1. We let 
        $n=1$, $u(x) = x^2$, and $s = \frac12$. In the above panel we compare
        the application of spectral fractional Laplacian (cf. Definition~\ref{specdef})
        and regional fractional Laplacian (cf.~\eqref{eq:dwb3}) onto $u(x)$. We observe
        that with our approach we are recovering the regional fractional Laplacian.
        We emphasize that boundary behavior of our dichotomy approach is part of 
        the future work.}
        \label{f:compSpecReg}
    \end{figure}

In Figure~\ref{f:compSpecReg}, the estimation of the spectral fractional Laplacian $(-\Delta)^s_S u$ with analytic eigenfunctions and the explicit solution of the regional fractional Laplacian, $(-\Delta)^s_R u$ in \eqref{reglaplaceexample} are compared to the results of applying the fractional diffusion maps algorithm with a polynomial kernel. In this figure, we have normalized the minimum of all the functions to $-1$ to cancel the effect of multiplicative constants. Notice that the numerical estimate obtained from the non-local heat kernel are much closer to the regional fractional Laplacian. This result leads us to hypothesize that polynomial kernels are estimating the regional fractional Laplacian which is inaccessible using local kernels. While this result is encouraging, it stimulates a more thorough investigation for improving the approximation, especially near the boundaries, which is beyond the scope of the current paper.

{

\subsection{Application to Kernel Ridge Regression}

A common application of kernel methods in machine learning is the kernel ridge regression.  The goal is to learn an unknown function $f: \mathbb{R}^n \to \mathbb{R}$ given noisy data $y_i = f(x_i) + \omega_i$ for $i = 1,...,N$.  A kernel regression is an approximation of $f$ based on a kernel $k:\mathbb{R}^n \times \mathbb{R}^n \to \mathbb{R}$ given by,
\[ \hat f(x) = \sum_{i=1}^N k(x,x_i)c_i, \hspace{40pt} \hat f(x_i) = \left({\bf K}\vec c  \right)_i. \]
We note that the kernel function $k$ is typically chosen to be positive definite so that ${\bf K}$ is invertible and the space of such functions $\hat f$ is a finite dimensional Reproducing Kernel Hilbert Space (RKHS).  A \emph{kernel regression} would simply set $\vec c = {\bf K}^{-1}\vec y$ where $\vec y_i = y_i$ is the vector of noisy data.  However, when we have certain a priori knowledge about the unknown function $f$, for instance, it is known to be regular (for example we may know that $f$ is $k$-times continuously differentiable) we can improve our estimate using a \emph{kernel ridge regression} which is defined by the optimization problem,
\begin{equation} 
	\label{eq:FDopt}
	\min_{\vec c} \left\{ | \vec y - {\bf K}\vec c |^2 + \delta |\vec c|^2 \right\} = \min_{\hat y = {\bf K}\vec c} \left\{ |\vec y - \hat y|^2 + \delta |{\bf K}^{-1}\hat y|^2 \right\} 
\end{equation}
which has solution $\vec c = ({\bf K}^\top{\bf K} + \delta {\bf I})^{-1} {\bf K}^\top \vec y$ that defines the function $\hat y_i = \hat f(x_i)$ with $\hat y = {\bf K}\vec c$.  A significant issue is the choice of kernel function, $k$, and the type of regularity imposed by various kernel choices.  Two common choices of kernel functions are the radial basis functions kernel with Gaussian exponential decay (corresponding to $\beta=2$) and polynomial decay kernels (corresponding to $\beta<2$).

For simplicity we will assume that $x_i \subset \iota(\mathcal{M}) \subset\mathbb{R}^n$ are uniformly distributed on a smooth compact submanifold $\iota(\mathcal{M})$.  When ${\bf K}$ approximates a heat kernel, in the limit of large $N$ we have 
	\begin{equation}
		\label{eq:Kinv}
		({\bf K}^{-1} \hat y)_i = \mathcal{K}_{-t} \hat f(x_i) 
	\end{equation} 
where $\mathcal{K}_{-t}$ is the ``backward-in-time" semigroup. In view of \eqref{eq:Kinv}, we can formally interpret the finite-dimensional kernel ridge regression optimization problem \eqref{eq:FDopt}, as an approximation to the following optimization problem 
\begin{equation}
	\label{eq:IDopt}
	\min_{\hat f \in H \subset \subset L^2(\mathcal{M})} \left\{ ||f - \hat f||_{L^2(\mathcal{M})}^2 + \delta ||\mathcal{K}_{-t}\hat f||_{L^2(\mathcal{M})}^2 , \right\} 
\end{equation}
in function spaces. A rigorous proof of the limiting behavior of \eqref{eq:FDopt} to \eqref{eq:IDopt} will be part of a forthcoming work. The notation $H \subset\subset L^2(\mathcal{M})$ indicates a compact embedding of the function space $H$ in $L^2(\mathcal{M})$. In other words, the second term in \eqref{eq:IDopt} enforces certain smoothness onto $\hat{f}$. Indeed, in case of exponential kernel, we are enforcing Laplacian type regularization onto $\hat{f}$, i.e., two derivatives. On the other hand, in the nonlocal case, we are enforcing a fractional Laplacian type regularization, i.e., $\beta = 2s$ derivatives.  We demonstrate the advantage of a polynomial kernel, due to reduced imposition of regularity, in the following example. 

Consider the case of the unit circle, $\mathcal{M} = S^1$, and the indicator function $f(\theta) = 1_{[0,\pi]}(\theta)$.  In Fig.~\ref{kernelreg} we show $f$ as a function of $\theta$, along with various data sets with $N \in \{100,500,2500\}$ (left to right) with independent Gaussian noise with mean zero and standard deviation $0.05$ (top) and $0.005$ (bottom) where we have zoomed to focus on the discontinuity.  In order to compare the exponential ($\beta=2$) and polynomial ($\beta=1$) kernels, we used a data set of $N$ points to learn the parameters $\epsilon$ and $\delta$ by minimizing the cross-validation error.  In particular, for a grid of values of $\epsilon \in [10^{-3},10^0]$ and $\delta \in [10^{-20},10^{-2}]$ we split the data set in half and learned $\vec c$ from half of the data points and computed the error on the other half of the data points.  We then chose the optimal pair $(\epsilon,\delta)$ and we computed the expected value of the regression function by setting $\vec y^{\rm true}_i = f(x_i)$ and $\vec c = ({\bf K}^\top{\bf K} + \delta {\bf I})^{-1} {\bf K}^\top \vec y^{\rm true}$.  Note that this is the expected value of the regression since $\vec y = \vec y^{\rm true} + \vec \omega$, the mean of $\omega$ is zero, and $\vec c$ is linear in $\vec y$.  

In Fig.~\ref{kernelreg} we compare the expected regression function for the exponential (red, dashed, $\beta=2$) and polynomial (blue, solid, $\beta=1$) kernels.  Notice that the exponential kernel exhibits oscillations near the discontinuity due to higher smoothness, this can be seen from its connection to Laplacian $(-\Delta)$ as argued in \eqref{eq:IDopt}.  
On the other hand, the heat kernel associated to the polynomial kernel is connected to the fractional Laplacian $(-\Delta)^{1/2}$ and which imposes far less smoothness.
This results in the polynomial kernel regression function having much smaller oscillation near the discontinuity, which also decreases as the noise decreases (bottom row).  This example shows how understanding the large data limit of various kernel functions can improve our understanding of the behavior of a kernel based learning algorithm.

    \begin{figure}[!h]
    \centering
        \includegraphics[width=0.32\textwidth]{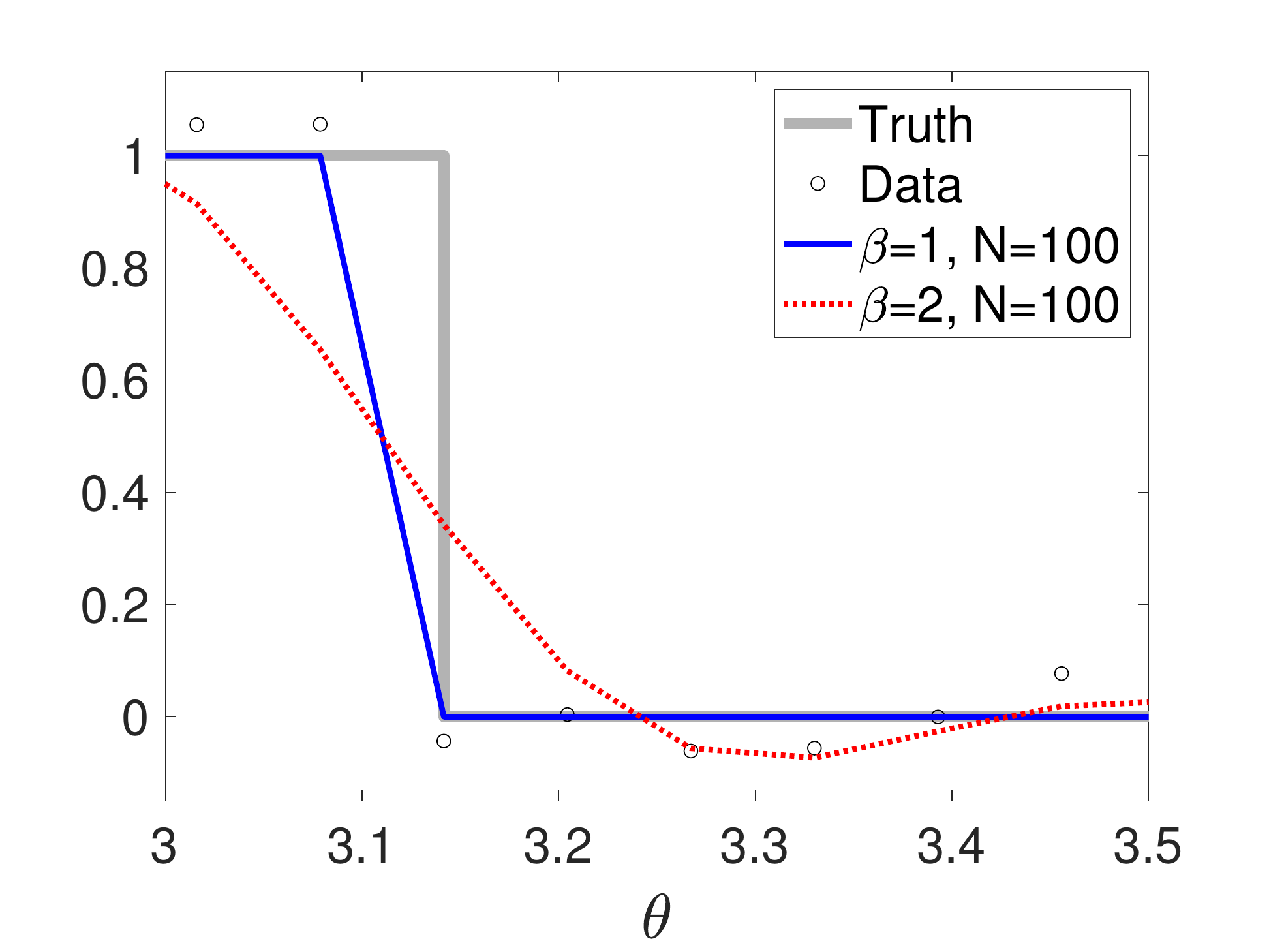}
        \includegraphics[width=0.32\textwidth]{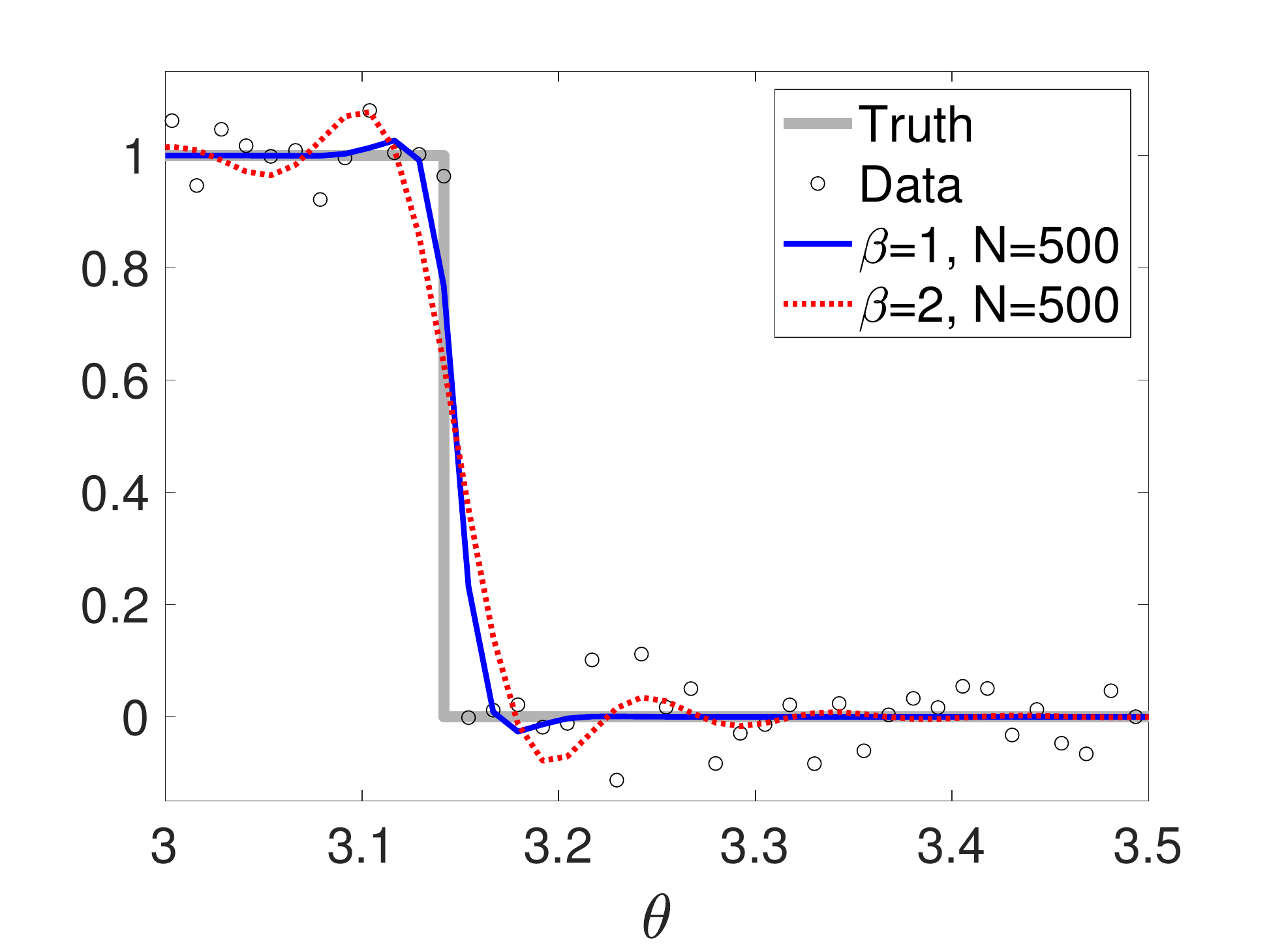}
        \includegraphics[width=0.32\textwidth]{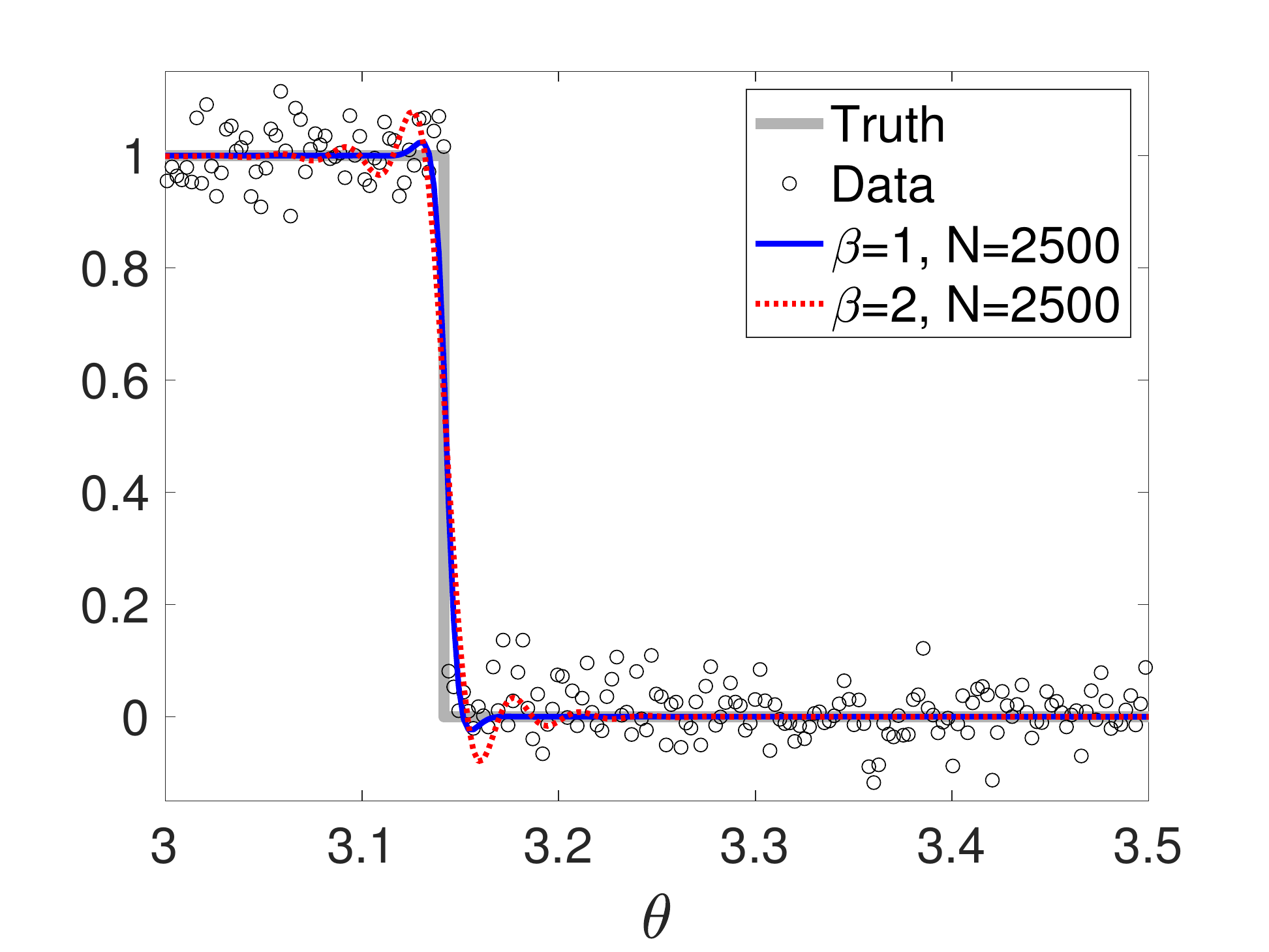}\\
        \includegraphics[width=0.32\textwidth]{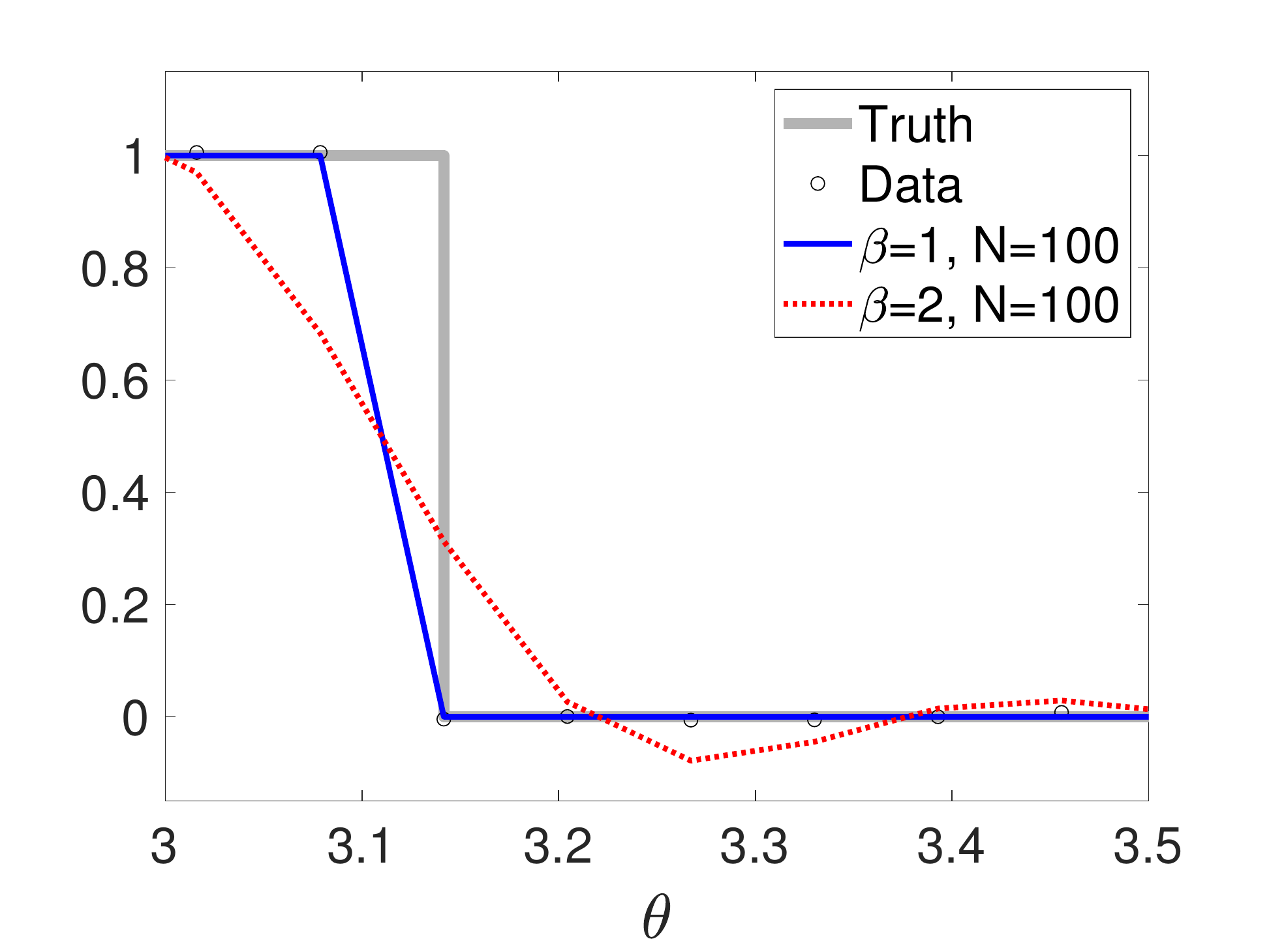}
        \includegraphics[width=0.32\textwidth]{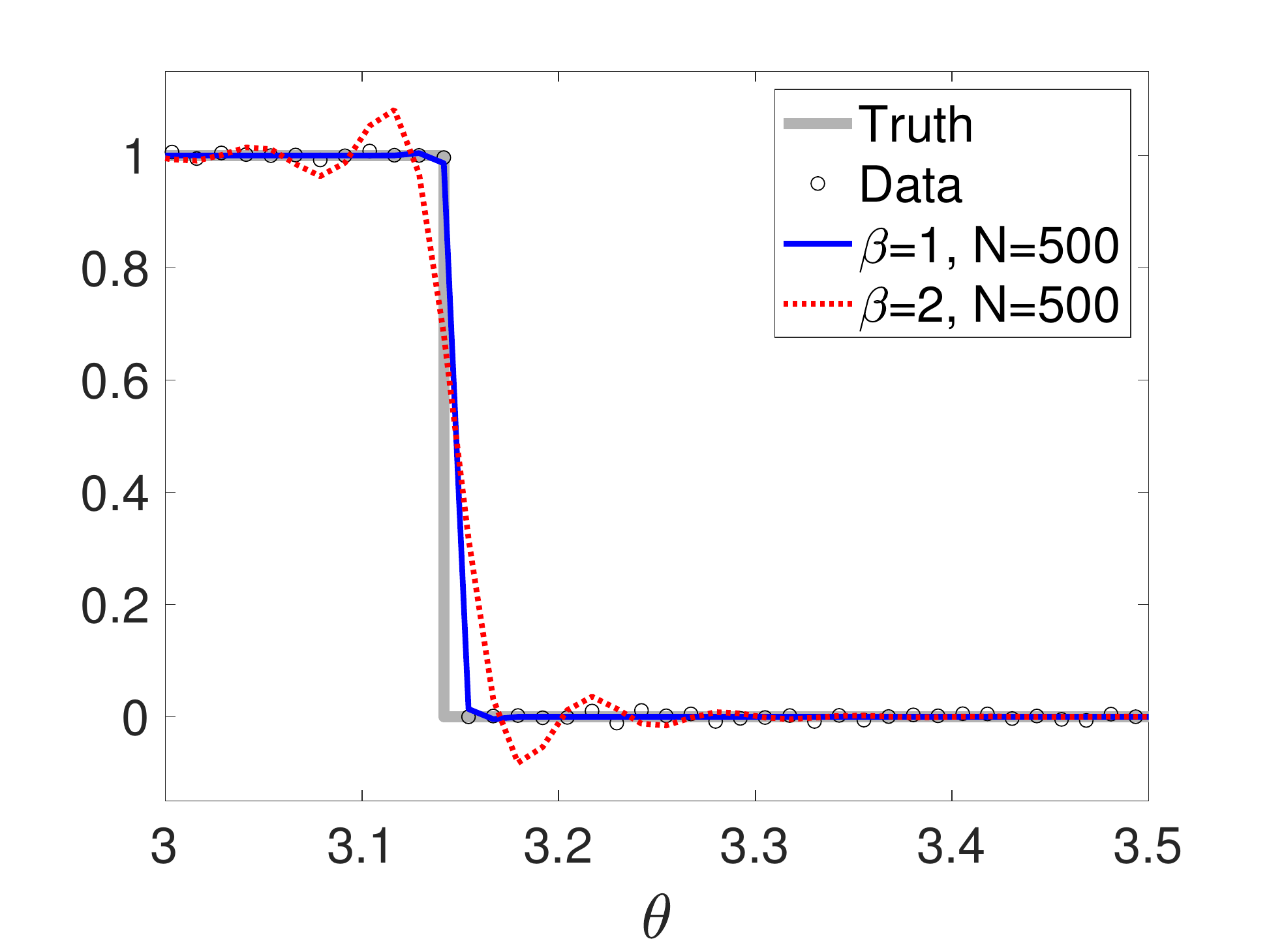}
        \includegraphics[width=0.32\textwidth]{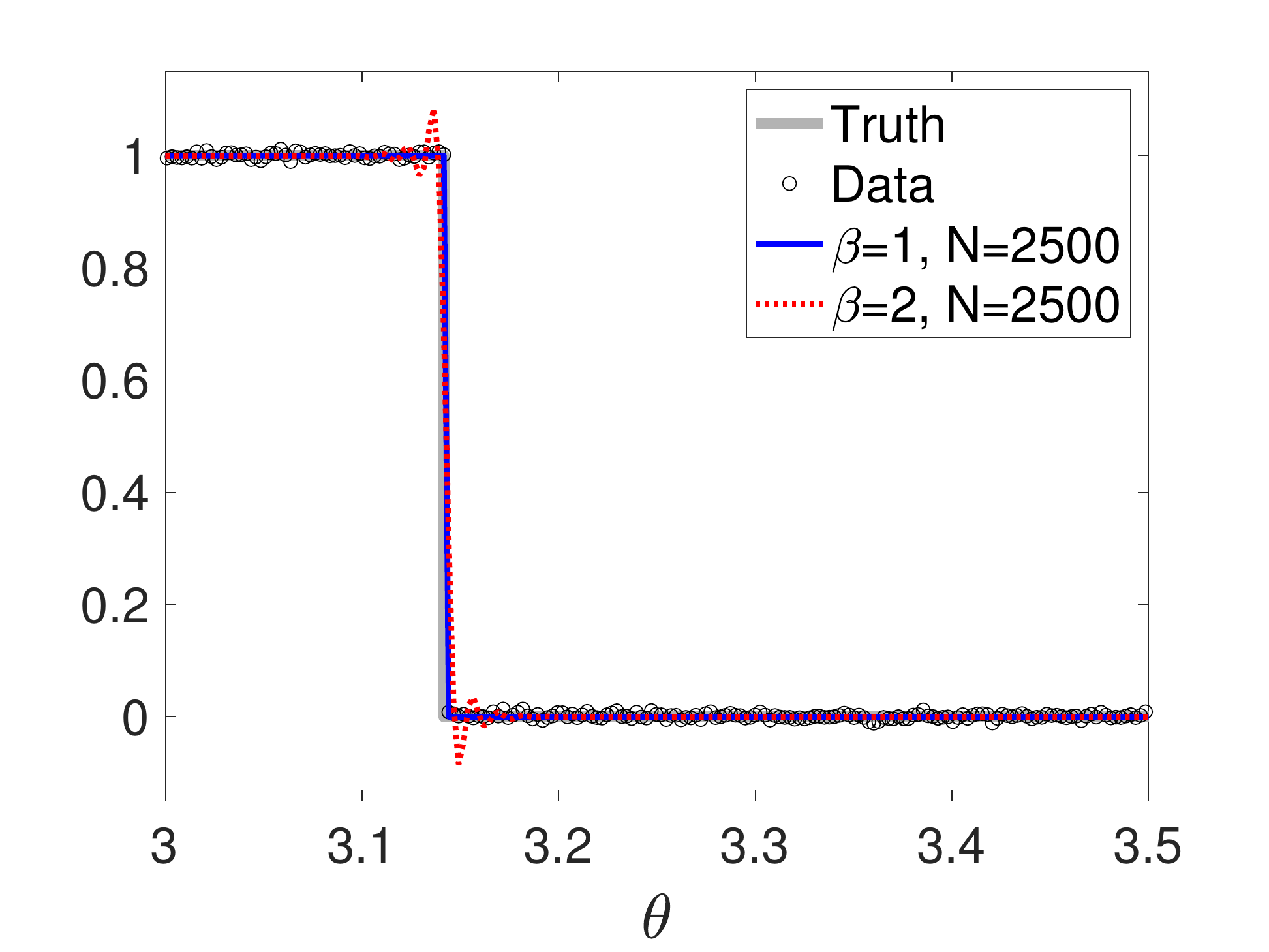}
        \caption{Kernel ridge regression on the unit circle $S^1$ to learn the indicator function $f(\theta) = 1_{[0,\pi]}(\theta)$. We show the expected regression function for the polynomial kernel (blue, solid, $\beta=1$) and exponential kernel (red, dashed, $\beta=2$) for $N\in \{100,500,2500\}$ (left to right) and mean zero Gaussian noise with  standard deviation $0.05$ (top row) and $0.005$ (bottom row).  The exponential kernel is associated to the classical Laplacian and the regression function is forced into a space of high regularity which creates oscillations at the discontinuity, whereas the fractional Laplacian associated to the polynomial kernel imposes less regularity and gives a better fit.}
        \label{kernelreg}
    \end{figure}

}

\section{Conclusion and future work}\label{conclusion}

The geometric understanding of certain kernel-based algorithms advanced by diffusion maps has given a valuable new perspective on what these algorithms are constructing and how to use it to understand the underlying data.  For example, by interpreting the kernel eigenfunctions as eigenfunctions of the intrinsic Laplacian operator on a manifold, we can use the associated eigenvalues to enforce regularity restrictions or Sobolev norm regularity conditions on interpolation problems.  However, diffusion maps (and generalizations \cite{localkernels}) is restricted to kernels with exponential decay, whereas kernel methods used in statistical learning theory may require general kernels.  The fractional diffusion maps approach extends the geometric understanding of data to a much larger class of kernels (including the well-known polynomial kernels).  At the same time, the fractional diffusion maps algorithm offers a new way to estimate certain fractional Laplacian operators on manifolds.  

{This kernel-based perspective is related to the popular Isomap and Parallel Transport Unfolding (PTU) algorithms.  Isomap \cite{tenenbaum2000global,MBernstein_VDSilva_JCLangford_JBTenenbaum_2000a} and PTU \cite{PTU} can be easily interpreted under the restrictive assumptions of manifolds that are both topologically trivial (contractible) and geometrically trivial (developable, which implies flat in the Riemannian sense).  However, Isomap and PTU are still useful for many examples that do not satisfy these assumptions, and in these contexts, there is a limited theoretical interpretation of their behavior.  While our analysis cannot address these algorithms directly, our kernel-based algorithms are closely related, and we rely on their methods for the key step of approximating geodesic distances.  Thus, our results give an alternative final step to Isomap and PTU, by replacing Multi-Dimensional Scaling (MDS) with a rigorously interpretable kernel method.}

Many directions of future work remain open.  For many data sets (such as data generated by a chaotic dynamical system) the assumption of an underlying manifold may be unrealistic, whereas a metric measure space would be a much less stringent assumption. The parallel between the dichotomies in the heat kernel and the associated fractional diffusion maps algorithm suggests that a generalization to a larger class of metric measure spaces may be possible.  Another important direction for future work would be the generalization to manifolds with boundary and other boundary conditions, { possibly using the ghost points \cite{jh:20} when the points at the boundary are given or using the distance to boundary estimator introduced in \cite{BERRY2017} and uniform expansions near the boundary \cite{ryan} when the boundary points are not given.} This generalization may make fractional diffusion maps a reasonable method for solving equations involving fractional Laplacians on domains that are difficult to mesh, generalizing the local kernels to solve elliptic PDE's on smooth manifolds as proposed in \cite{JohnPDE1,JohnPDE2,jh:20}.  Generalizing to allow anisotropic kernels as in \cite{localkernels} may provide access to geometries that are not inherited from the embedding, and variable bandwidth kernels \cite{variablebandwidth} may allow for data sampled from non-compact manifolds with finite volume. {Another possible research direction is to implement the non-local kernels with improved geodesic distance estimator \cite{PTU}, which consistency requires a careful treatment of the error analysis. }

\section*{Acknowledgments}
H.A. is partially supported by NSF grants DMS-1818772 and DMS-1913004, the Air Force Office of Scientific Research under Award NO: FA9550-19-1-0036, and  the Department of Navy, Naval PostGraduate School under Award NO: N00244-20-1-0005.  T.B. 
is partially supported by NSF grants DMS-1723175 and DMS-1854204. J.H. was partially supported by the ONR Grant N00014-16-1-2888, NSF Grants DMS-1619661 and DMS-1854299.


\appendix

\section{Proofs of technical lemmas}\label{proofs}

\subsection{Proof of Lemma \ref{distcomp}}

\begin{proof}
If we let $s$ be the exponential coordinates for $y$ centered at $x$ so that $y=\exp_x(s)$ then we can Taylor expand $\iota(\exp_x(s))$ around $s=0$.  We first note that since $\exp_x(0)=x$ and $D\exp_x(0) = I_{d\times d}$ we have
\[ \left. D_s \iota(\exp_x(s)) \right|_0 = \left. D\iota(\exp_x(s))D\exp_x(s) \right|_0 = D\iota(\exp_x(0))D\exp_x(0) = D\iota(x). \]
Next, since $D^2\exp_x(0) = 0$ we have
\[  \left. D_s^2 \iota(\exp_x(s)) \right|_0 = \left. D\exp_x(s)^\top H(\iota)(\exp_x(s))D\exp_x(s) + D\iota(\exp_x(s))D^2\exp_x(s) \right|_0 = H(\iota)(x).\]
Together these equalities give the Taylor expansion
\[ \iota(y) = \iota(\exp_x(s)) = \iota(x) + D\iota(x)s + H(\iota)(x)(s,s) + \mathcal{O}(|s|^3). \]
A key feature of this expansion is that the Hessian $H(\iota)$ of the embedding is orthogonal to the tangent \cite{rosenberg} space, so when computing the norm we have
\[ |\iota(y)-\iota(x)|^2 = |D\iota(x)s + H(\iota)(x)(s,s) + \mathcal{O}(s_i^3)|^2 = |D\iota(x)s|^2 + \mathcal{O}(|s|^4) \]
where the only third order term is the cross term $\left<D\iota(x)s,H(\iota)(x)(s,s)\right> = 0$ by the orthogonality mentioned above.  Note that the term $\mathcal{O}(|s|^4)$ assumes that the Hessian and the third derivative of the embedding are bounded, so we require the manifold and embedding to be $C^3$.  Since $\iota$ is isometric, the columns of $D\iota(x)$ are orthonormal, so $|D\iota(x)s|=|s|=d_g(x,y)$ meaning
\[ |\iota(y)-\iota(x)|^{\alpha} = d_g(x,y)^{\alpha} + \mathcal{O}(d_g(x,y)^{\alpha+2}). \]
\end{proof}

\subsection{Proof of Lemma \ref{fastdecay}}
\begin{proof}
Since $\gamma>0$, the radius of the ball $|y|=t^{-\gamma}$ is expanding and thus we are integrating over the tail of an exponential, which decays faster than any polynomial in $t$ for any $\gamma$.  To see this we first make the change of variables $w_i = y_i^{\alpha/2}$ so that 
\[ |y|^{\alpha} = \left( \left( \sum_i y_i^2 \right)^{\alpha/4}\right)^2 = \left( \left( \sum_i w_i^{4/\alpha} \right)^{\alpha/4}\right)^2 = |w|_{4/\alpha}^2 \]
and since $dy = \frac{2}{\alpha} \prod_i w_i^{2/\alpha -1} dw$ we have
\begin{align} \int_{|y|>t^{-\gamma}} e^{-c |y|^{\alpha}} \, dy  &= \frac{2}{\alpha}\int_{|w|_{4/\alpha}^{2/\alpha} >t^{-\gamma}} e^{-c|w|_{4/\alpha}^2} \prod_i w_i^{2/\alpha-1} \, dw  \nonumber \\
&\leq  \frac{2}{\alpha}\int_{c_2 |w| > t^{-\gamma\alpha/2}} e^{-c_1 |w|^2} \prod_i w_i^{2/\alpha-1} \, dw  \nonumber \end{align}
for some $c_1,c_2>0$ where the last inequality follows by equivalence of norms we have $a|w| < |w|_{4/\alpha} < b|w|$.  Next we further expand the domain of integration to allow us to split up the integrals.  Notice that the cube with sides $|w_i| \leq \frac{t^{-\gamma \alpha/2}}{c_2 2^{1/d}}$ fits inside the ball of radius $\frac{t^{-\gamma\alpha/2}}{c_2}$ so we can extend the integral to the outside of the cube and it will only get larger.  Over this domain we can split up the integral and the integrals are the same over each variable $w_i$, so continuing the previous inequality we have 
\begin{align} \hspace{40pt} &\leq \frac{2}{\alpha}\int_{c_2 |w_i| > t^{-\gamma\alpha/2}/2^{1/d}}  \prod_i w_i^{2/\alpha-1} e^{-c_1 w_i^2} \, dw_1,...,dw_d \nonumber \\
&= \frac{2}{\alpha}\left(\int_{c_2 |w_i| > t^{-\gamma\alpha/2}/2^{1/d}}   w_i^{2/\alpha-1} e^{-c_1 w_i^2} \, dw_i \right)^d \nonumber \\
&\leq \frac{2}{\alpha}\left(\int_{c_2 |w_i| > t^{-\gamma\alpha/2}/2^{1/d}}   w_i e^{-c_1 w_i^2} \, dw_i \right)^d \nonumber \\
&= \frac{2}{\alpha} \left(\frac{1}{2c_1} e^{-\frac{c_1}{2^{2/d} c_2^2} t^{-\gamma\alpha}}\right)^d \nonumber \\
&= c_3 e^{-\frac{c_4}{t^{\gamma\alpha}}}
\end{align}
for constants $c_3,c_4 > 0$, where the last inequality follows from the fact that $\alpha\geq 1$ so that $2/\alpha - 1 \leq 1$ and $w_i^{2/\alpha-1}\leq w_i$ for $t$ sufficiently small so that $w_i>1$ on the domain of integration.  Since $\gamma\alpha > 0$, as $t\to 0^+$ the final term goes to zero faster than any polynomial.
\end{proof}

\subsection{Proof of Lemma \ref{localization}}
\begin{proof}
Substituting the upper bound for $\alpha$-local kernels we have the upper bound
\begin{align} \left| \int_{\stackrel{z\in\mathcal{M},}{ |\iota(z)-x|>t^{1-1/\alpha-\gamma}}} \right. &\hspace{1pt} t^\ell  \left. e^{-c\left|\left| \frac{x-\iota(z)}{t^{1-1/\alpha}} \right|\right|^{\alpha}}f(z)\, d\textup{vol} \right| \nonumber \\
&\leq ||f||_{L^p(\mathcal{M})} t^\ell \left( \int_{z\in\mathcal{M},|\iota(z)-x|>t^{1-1/\alpha-\gamma}} e^{-\frac{c p}{p-1}\left|\left| \frac{x-\iota(z)}{t^{1/\alpha}} \right|\right|^{\alpha}}\, d\textup{vol} \right)^{\frac{p-1}{p}} \nonumber \\
&\leq ||f||_{L^p(\mathcal{M})} t^\ell \left( \int_{|y-x|>t^{1-1/\alpha-\gamma}} e^{-\frac{c p}{p-1}\left|\left| \frac{x-y}{t^{1/\alpha}} \right|\right|^{\alpha}}\, dy \right)^{\frac{p-1}{p}} \nonumber \\
&= ||f||_{L^p(\mathcal{M})} t^\ell  \left(\int_{|w|>t^{-\gamma}} e^{-\frac{c p}{p-1}|w|^{\alpha}} \, t^{d(1-1/\alpha)}dw \right)^{\frac{p-1}{p}} \nonumber \\
&\leq ||f||_{L^p(\mathcal{M})} t^{\ell+\frac{(p-1)d(1-1/\alpha)}{p}} \left(  \int_{|w|>t^{-\gamma}} e^{-\frac{c p}{p-1} |w|^{\alpha}} \, dw \right)^{\frac{p-1}{p}} \nonumber \end{align}
where we applied H\"older's inequality and then extend the integral to all of $y \in \mathbb{R}^n$ outside the ball of radius $t^{-\gamma}$.   We then set $w = \frac{y-x}{t^{1-1/\alpha}}$ so that $dw = t^{-d(1-1/\alpha)} dy$.  Since $\alpha>1$ the result follows from Lemma \ref{fastdecay}.
\end{proof}

\section*{References}
\bibliography{refs}
\end{document}